\documentclass[11pt, oneside]{amsart}
\usepackage{amsmath}
\usepackage{amsthm}
\usepackage{amssymb} 
\usepackage[mathscr]{euscript}
\usepackage{amsmath}
\usepackage{amsthm}
\usepackage{amssymb}

\usepackage{comment}
\excludecomment{toexclude} 

\usepackage{fancyhdr}
\usepackage{enumerate}
\usepackage{amscd}
\usepackage[all]{xy}
\usepackage{graphicx}
\usepackage{color}
\usepackage{hyperref}
\hypersetup{
    colorlinks,
    citecolor=black,
    filecolor=black,
    linkcolor=black,
    urlcolor=black
}
\usepackage{tikz}
\usepackage[all]{xy}
\usepackage{graphicx}
\usetikzlibrary{backgrounds}
\usepackage{soul}

\theoremstyle{plain}
\newtheorem{lemma}{Lemma}[section]
\newtheorem{proposition}[lemma]{Proposition}
\newtheorem{theorem}[lemma]{Theorem}
\newtheorem{Theorem}{Theorem}
\newtheorem{Proposition}[Theorem]{Proposition}

\newtheorem{corollary}[lemma]{Corollary}
\newtheorem{Corollary}[Theorem]{Corollary}

\newtheorem{conjecture}[Theorem]{Conjecture}

\numberwithin{equation}{section}

\newtheorem{manualtheoreminner}{Theorem}
\newenvironment{manualtheorem}[1]{%
  \renewcommand\themanualtheoreminner{#1}%
  \manualtheoreminner
}{\endmanualtheoreminner}

\newtheorem{manualpropositioninner}{Proposition}
\newenvironment{manualproposition}[1]{%
  \renewcommand\themanualpropositioninner{#1}%
  \manualpropositioninner
}{\endmanualpropositioninner}

\theoremstyle{remark}
\theoremstyle{definition}
\newtheorem{remark}[lemma]{Remark}

\newtheorem{example}[lemma]{Example}
\newtheorem{definition}[lemma]{Definition}
\newtheorem{Definition}[Theorem]{Definition}

\newtheorem{manualdefinitioninner}{Definition}
\newenvironment{manualdefinition}[1]{%
  \renewcommand\themanualdefinitioninner{#1}%
  \manualdefinitioninner
}{\endmanualdefinitioninner}

\DeclareMathOperator{\pt}{\tfrac{\partial}{\partial t}} 
 
\DeclareMathOperator{\p2s}{\tfrac{\partial^2}{\partial s^2}} 

\DeclareMathOperator{\ps}{\tfrac{\partial}{\partial s}} 
\DeclareMathOperator{\pst}{\tfrac{\partial^2}{\partial s\partial t}} 
\DeclareMathOperator{\pts}{\tfrac{\partial^2}{\partial t\partial s}}

\DeclareMathOperator{\ds}{\tfrac{d}{d s}} 
 
\DeclareMathOperator{\2ds}{\tfrac{d^2}{d s^2}} 

\DeclareMathOperator{\tr}{tr} 
\DeclareMathOperator{\Div}{div}
\DeclareMathOperator{\Ric}{\mathrm{Ric}}
\DeclareMathOperator{\dvol}{d\mathrm{vol}}
\DeclareMathOperator{\da}{d\sigma}
\DeclareMathOperator{\S_2}{\mathcal{S}_2(\Sigma)}
\DeclareMathOperator{\bi}{\beta}
\DeclareMathOperator{\D}{\mathcal{D}}

\DeclareMathOperator{\C}{\mathcal{C}}
\DeclareMathOperator{\Ker}{\mathrm{Ker}}
\DeclareMathOperator{\Range}{\mathrm{Range}}
\DeclareMathOperator{\Dim}{\mathrm{dim}}
 
\DeclareMathOperator{\Int}{Int}
\DeclareMathOperator{\I}{\mathtt{I}}

\usepackage[tmargin=0.96in,bmargin=1in,lmargin=1.5in,rmargin=1.5in]{geometry}
\begin{document}

\title[Existence of static vacuum extensions]{Existence of static vacuum extensions with  prescribed Bartnik boundary data}
\author{Zhongshan An}
\address{Department of Mathematics, University of Connecticut, Storrs, CT 06269, USA}
\email{zhongshan.an@uconn.edu}
\author{Lan-Hsuan Huang}
\address{Department of Mathematics, University of Connecticut, Storrs, CT 06269, USA}
\email{lan-hsuan.huang@uconn.edu}
\thanks{The second author was partially supported by the NSF CAREER Award DMS-1452477 and NSF DMS-2005588.}

\maketitle
\begin{abstract}
We prove the existence and local uniqueness of asymptotically flat, static vacuum metrics with arbitrarily prescribed Bartnik boundary data that are close to the induced boundary data  on any star-shaped hypersurface or a general family of perturbed  hypersurfaces in the Euclidean space. It confirms the existence part of the Bartnik static extension conjecture for large classes of boundary data, and the static vacuum metric obtained is geometrically unique in a neighborhood of the Euclidean metric.
\end{abstract}

\tableofcontents

\section{Introduction}

 Let $n\ge 3$ and $(U, g)$ be an $n$-dimensional Riemannian manifold and let $u$ be a scalar-valued function on $U$, not identically zero. We say that $(g, u)$ is a \emph{static vacuum pair}   in $U$ if it solves the static vacuum system:
\begin{align}\label{sv}
	\left\{ \begin{array}{l} u\Ric_g -\nabla^2_g u=0\\
	\Delta_g u=0\end{array}\right. \quad \mbox{ in } U.
\end{align}
Such  $g$ is also referred to as a  \emph{static vacuum metric} with $u$ as a \emph{static potential}. The study of a static vacuum pair is of fundamental importance in general relativity. It defines a spacetime  $- u^2 dt^2 + g$ on $\mathbb{R}\times \{ x\in U: u(x)>0\}$ that solves the vacuum Einstein equation. A spacetime that takes this form is called \emph{static}, meaning that $\partial_t$ is a global timelike Killing vector field everywhere normal to the $t$-slices.  In a similar way, a static vacuum pair is  related to the study of Ricci curvature as $(g, u)$  defines a Ricci flat Riemannian metric $u^2 dt^2 + g$ that carries a global Killing vector field.

The static vacuum metric also arises in the minimization problems of the ADM mass.  By Riemannian positive mass theorem of R.~Schoen and S.-T.~Yau~\cite{Schoen-Yau:1979-pmt1} and Riemannian Penrose inequality of G.~Huisken and T. Ilmanen~\cite{Huisken-Ilmanen:2001} and of H.~Bray \cite{Bray:2001},  the ADM mass minimizing metric among an appropriately defined class of asymptotically flat metrics is  characterized by a class of static metrics, the Schwarzschild metrics. In a fundamentally related way, a static vacuum metric  appears in the study of scalar curvature deformation and gluing, see, e.g. the work of J.~Corvino~\cite{Corvino:2000}.

In some situations, the asymptotically flat, static vacuum metrics are quite rigid. The \emph{uniqueness of static black holes} \cite{Israel:1968, Robinson:1977, Bunting-Masood-ul-Alam:1987} says that the Schwarzschild family of positive mass is the only asymptotically flat, static vacuum metrics that admit a static potential $u>0$ in the interior of manifold with $u=0$ on the boundary. On the other hand, the following conjecture of Bartnik asserts the existence and uniqueness of static vacuum metrics realizing arbitrarily prescribed geometric boundary data~\cite[Conjecture 7]{Bartnik:2002}.

\begin{conjecture}[Bartnik's static extension conjecture]\label{conjecture}
Let $(\Omega_0, g_0)$ be an $n$-dimensional, connected, compact Riemannian manifold with scalar curvature $R_{g_0}\ge 0$ and with boundary $\Sigma$. Denote by $g_0^\intercal $ and  $H_{g_0}$ the induced metric and mean curvature on $\Sigma$ respectively. Suppose $H_{g_0}$ is positive somewhere. Then there exists a unique  asymptotically flat, static vacuum manifold $(M, g)$ with boundary $\partial M$ diffeomorphic to $\Sigma$ such that  $(g^\intercal, H_g) =  (g_0^\intercal, H_{g_0})$, where  $g^\intercal $ and  $H_{g}$ are respectively the induced metric and mean curvature on $\partial M \subset (M, g)$. 
\end{conjecture}

 In the above, the mean curvatures $H_{g_0} = \Div_\Sigma \nu_{g_0}$ and $H_g = \Div_\Sigma \nu_g$, where  $\nu_{g_0}$ is the outward unit normal to $\Omega_0$ and  $\nu_g$ is the inward unit normal  to $M$. Throughout the paper we shall refer to the geometric boundary data $(g^\intercal, H_g)$ as the \emph{Bartnik boundary data}, and $(M, g)$ as a static vacuum extension for the Bartnik boundary data.
 
 We remark on the assumption that  $H_{g_0}$ is  positive somewhere. Without this assumption, an asymptotically flat, static vacuum extension, if exists, would contain a closed minimal hypersurface, which violates the ``no-horizon'' condition. See Definition 4 and p. 236 in \cite{Bartnik:2002}.  In fact, if $n=3$ and $H_{g_0}=0$, then a static vacuum extension must be a Schwarzschild metric, ~see Bray~\cite{Bray:2001}, P.~Miao~\cite{Miao:2005}, A.~Carlotto, O.~Chodosh, M.~Eichmair~\cite{Carlotto-Chodosh-Eichmair:2016},  also \cite{Mantoulidis-Schoen:2015, Huang-Martin-Miao:2018}. Let us also remark that it is not known whether the scalar curvature assumption~$R_{g_0}\ge 0$ is needed.

 We explain uniqueness in Conjecture~\ref{conjecture}.  If $(M, g)$ is an asymptotically flat, static vacuum  manifold, then for any diffeomorphism $\psi:M \to M$ that fixes the boundary and is asymptotic to a Euclidean isometry at infinity, the pull-back metric $\psi^*g$ is also an asymptotically flat, static vacuum metric in $M$ having the same Bartnik boundary data. So the uniqueness assertion in the above conjecture should be interpreted as ``unique up to diffeomorphisms''.  
 
Conjecture~\ref{conjecture}  originated from other conjectures of Bartnik about the quasi-local mass in general relativity, see, for example, \cite{Bartnik:1989, Bartnik:1997, Bartnik:2002}. Nevertheless, the conjecture itself is  of independent interest.  In contrast to the existing classification results, an affirmative answer toward the conjecture will add more examples of asymptotically flat, static vacuum pairs to the current (quite limited) list (see, e.g. \cite[Chapters 18.6 and 20.2]{Stephani:2003}).

There are partial progresses toward Conjecture~\ref{conjecture}.  P. Miao~\cite{Miao:2003} confirmed the existence for $(g_0^\intercal, H_{g_0})$  that is sufficiently close to $(g_{S^2}, 2)$ and invariant under reflectional symmetry, where $(g_{S^2}, 2)$ is the Bartnik boundary data on a standard unit sphere, and obtained  static vacuum extensions with reflectional symmetry.  M.~Anderson and M.~Khuri formulated the problem into an elliptic boundary value problem of Fredholm index zero~\cite{Anderson-Khuri:2013}, and  Anderson~ \cite{Anderson:2015} confirmed the existence and local uniqueness for $(g_0^\intercal, H_{g_0})$ sufficiently close to $(g_{S^2}, 2)$. We also note some numerical results for axial-symmetrical cases by C. Cederbaum, O. Rinne, and M. Strehlau~\cite{Cederbaum-Rinne-Strehlau:2019}.

In this paper, we confirm the conjecture in that, for  larger classes of Bartnik boundary data, static vacuum extensions exist and are locally unique.  Before we state our main results, we give some definitions. Let  $\Omega$ be a bounded open subset in $\mathbb{R}^n$ whose boundary $\Sigma = \partial \overline{\Omega}$ is an embedded smooth hypersurface in $\mathbb{R}^n$. Throughout the paper, we assume $\Sigma$ is connected, unless otherwise indicated. We denote by $\bar{g}$ a Euclidean metric in $\mathbb{R}^n$. The definition of weighted H\"older spaces is given in Section~\ref{section:basic}, and we always assume  $\alpha\in(0,1), q\in (\frac{n-2}{2} , n-2)$, unless otherwise indicated.

\begin{Definition}\label{definition:static-regular}
The boundary $\Sigma$ is said to be \emph{static regular in  $\mathbb R^n\setminus\Omega$} if  for any  pair of a symmetric $(0,2)$-tensor $h$ and a scalar-valued function $v$ satisfying $(h, v)\in  \C^{2,\alpha}_{-q}(\mathbb{R}^n\setminus \Omega)$ and
\begin{align*}
D\Ric|_{\bar{g}} (h)- \nabla^2_{\bar{g}} v&=0,  \quad  \Delta_{\bar{g}} v=0 \quad \mbox{ in } \mathbb{R}^n\setminus \Omega,\\
h^\intercal&=0, \quad DH|_{\bar{g}}(h)=0\quad \mbox{ on } \Sigma,
\end{align*}
we must have $DA|_{\bar{g}}(h)=0$ on~$\Sigma$. 
\end{Definition}
In the above definition, $D\Ric|_{\bar{g}}$ denotes the linearization of the Ricci curvature at $\bar{g}$, and $DH|_{\bar{g}}, DA|_{\bar{g}}$ denote the linearizations of the mean curvature and second fundamental form on $\Sigma$, respectively.

Our first main result gives a general criterion for the existence and local uniqueness of static vacuum pairs in an arbitrary exterior region $\mathbb{R}^n\setminus \Omega$ with Bartnik boundary data on $\Sigma$ sufficiently close to the induced data $(\bar{g}^\intercal, H_{\bar{g}})$.
\begin{Theorem}\label{theorem:criterion-main}
Suppose the boundary $\Sigma$ is static regular in $\mathbb{R}^n\setminus \Omega$. Then there exist positive constants $\epsilon_0, C$ such that for each $\epsilon < \epsilon_0$, if $(\tau, \phi)$ satisfies~$\| (\tau,\phi)- (\bar{g}^\intercal, H_{\bar{g}}) \|_{\C^{2,\alpha}(\Sigma)\times \C^{1,\alpha}(\Sigma)}<\epsilon$, then there exists an asymptotically flat  pair $(g, u)$ with $\|(g, u) - (\bar{g}, 1) \|_{\C^{2,\alpha}_{-q}(\mathbb{R}^n \setminus \Omega)} < C\epsilon$ such that $(g, u)$ is a static vacuum pair in $\mathbb{R}^n\setminus \Omega$ having the Bartnik boundary data $(g^\intercal, H_g) = (\tau, \phi)$ on~$\Sigma$.

Furthermore, the solution $(g, u)$ is geometrically unique in a neighborhood $\mathcal U$ of $(\bar g, 1)$ in the $\C^{2,\alpha}_{-q}(\mathbb{R}^n\setminus \Omega)$-norm; namely, there exists a unique solution satisfying both the static-harmonic gauge and the orthogonal gauge in $\mathcal U$.
\end{Theorem}

Note that the constants $\epsilon_0, C$ depend on the global geometry of $\Sigma $ and $\mathbb{R}^n\setminus \Omega$. See Remark~\ref{re:constant}.

In the above theorem, we say that an asymptotically flat pair $(g, u)$ satisfies  the \emph{static-harmonic gauge} in $\mathbb{R}^n\setminus \Omega$ if 
\begin{align}\label{static-harmonic}
\bi g + du = 0\quad \mbox{ in } \mathbb{R}^n\setminus \Omega,
\end{align}
where the Bianchi operator $\bi$ (at $\bar{g}$) is defined by $\bi g :=- \Div_{\bar{g}} g + \tfrac{1}{2} d (\tr_{\bar{g}} g)$. We say that $(g, u)$ satisfies the \emph{orthogonal gauge} if 
\begin{align}\label{orthogonal}
	\int_{\mathbb{R}^n\setminus \Omega}  \big((g- \bar{g})\cdot L_X\bar{g} \big)\rho\, d\mathrm{vol}=0,\qquad \mbox{ for all $X\in \mathcal{X}_0$}
\end{align}
where the weight function  $\rho(x) = (1+|x|^2)^{-1}$, the dot ``$\cdot$'' denotes the inner product of $\bar{g}$, and $\mathcal{X}_0$ is the subspace of $\C^{3,\alpha}_{\mathrm{loc}}(\mathbb{R}^n\setminus \Omega)$  vector fields, consisting of  $X$ satisfying $\Delta_{\bar{g}} X=0$ in $\mathbb{R}^n\setminus \Omega$  with $X=0$ on $\Sigma$ and $X- K\in \C^{3,\alpha}_{2-n}(\mathbb{R}^n\setminus \Omega)$ for some Killing vector field $K$ of $(\mathbb{R}^n, \bar{g})$. More details about $\mathcal{X}_0$ and the gauge conditions can be found in Section~\ref{section:gauge} and  at the beginning of Section~\ref{section:existence}.

We remark how the imposed gauge conditions relate to uniqueness in Conjecture~\ref{conjecture}. We show that for any $(g, u)$ sufficiently close to $(\bar{g}, 1)$, there is a unique diffeomorphism $\psi$ of $M$ that is close to the identity map and  fixes the boundary such that $(\psi^* g, \psi^* u)$ satisfies both gauges (see Lemma~\ref{lemma:both-gauge}). Therefore, Theorem~\ref{theorem:criterion-main} implies that the static vacuum extension with the prescribed Bartnik boundary data is unique (up to diffeomorphisms) within an open neighborhood of $(\bar{g}, 1)$.

Part of our proof adapts the work of Anderson and Khuri \cite{Anderson-Khuri:2013} to formulate the  problem into finding solutions to an elliptic boundary value system of Fredholm index $0$. However, because such elliptic system always admits a nontrivial kernel and cokernel, our main argument  uses a new modified elliptic system. 
  
More importantly, we can show that several classes of hypersurfaces in $\mathbb{R}^n$ are static regular. Together with Theorem~\ref{theorem:criterion-main}, we confirm the existence and local uniqueness of Conjecture~\ref{conjecture} for $(g_0^\intercal, H_{g_0})$ sufficiently close to the induced Bartnik boundary data on those static regular boundaries.

We first discuss about convex surfaces in $\mathbb{R}^3$, i.e. surfaces of positive Gauss curvature. Suppose $\Sigma  \subset (\mathbb{R}^3, \bar{g})$ is convex.   For an asymptotically flat metric $g$ in $\mathbb{R}^3\setminus \Omega$ with $g^\intercal$ sufficiently close to $\bar{g}^\intercal$ on $\Sigma$, there exists an isometric embedding $f: (\Sigma, g^\intercal)\to (\mathbb{R}^3, \bar{g})$, unique up to a rigid motion of~$\mathbb{R}^3$. Let  $\mathring{H}_{g^\intercal}$  be the mean curvature of the isometric image $f(\Sigma) \subset (\mathbb{R}^3,\bar{g})$ (with respect to the unit normal $\nu$ pointing to infinity). We can define the following functional
\begin{align*}
		\mathscr{G}(g) &= -16\pi m_{\mathrm{ADM}}(g) + \int_{\mathbb{R}^3\setminus \Omega} R_g \, d\mathrm{vol}_g + \int_\Sigma 2(\mathring{H}_{g^\intercal}-H_g) \, d\sigma_{g}.
\end{align*}
This functional takes a similar form as the functional considered by S. Hawking and G. Horowitz~\cite{Hawking-Horowitz:1996}, but we use it in different ways. One reason to consider  $\mathscr{G}$ is that $\bar{g}$ is a critical point of $\mathscr{G}$ among any variations, partly motivated by the work of S.~Brendle, F. Marques, and A. Neves~\cite{Brendle-Marques-Neves:2011}. It is also partly motivated by the following result of Y.~Shi and L.-F.~Tam:
\begin{Theorem}[Shi-Tam~\cite{Shi-Tam:2002}]\label{theorem:Shi-Tam}
Let $(\overline{\Omega}, g_{\I})$ be a 3-dimensional compact manifold with boundary $\Sigma$. Suppose the scalar curvature $R_{g_{\I}}\ge 0$ and the boundary $(\Sigma, (g_{\I})^\intercal)$ has positive Gauss curvature and positive mean curvature (with respect to outward unit normal). Then
\[
	 \int_\Sigma  \Big( \mathring{H}_{(g_{\I})^\intercal} -H_{g_{\I}} \Big) \, d\sigma_{g_{\I}}\ge 0
\]
with equality if and only if $(\overline{\Omega}, g_{\I})$ is isometric to a Euclidean region. 
\end{Theorem}
In the above theorem, we use the subscript $\I$ to emphasize that the metric $g_{\I}$ and the  corresponding quantities are of the compact region $(\overline{\Omega}, g_{\I})$. In particular, the integral in Theorem~\ref{theorem:Shi-Tam} need not relate to the boundary integral of $\mathscr{G}$, unless $g$ and $g_{\I}$ induce the same metric and mean curvature on~$\Sigma$. By computing the second variation of $\mathscr{G}$ along a specific path of metrics that relates the boundary integral of $\mathscr{G}$ with Theorem~\ref{theorem:Shi-Tam}, we derive a \emph{linearized} mean curvature comparison inequality.

\begin{Proposition}\label{proposition:harmonic-function}
Let  $\Omega$ be a bounded open subset in $\mathbb{R}^3$ whose boundary $\Sigma = \partial \overline{\Omega}$ has  positive Gauss curvature. If $v\in\C^{2,\alpha}_{-q}(\mathbb R^3\setminus\Omega)$ satisfies $\Delta_{\bar{g}} v=0$ in $\mathbb{R}^n\setminus \Omega$, then  
\[
\int_\Sigma 2v\Big(D \mathring{H}|_{\bar{g}^\intercal}\big(2v\bar{g}^\intercal\big) - DH|_{\bar{g}}(2v\bar{g}) \Big) \, d\sigma_{\bar{g}}\geq 0,
\]
where $D \mathring{H}|_{\bar{g}^\intercal}$ denotes the linearization of $\mathring{H}_{g^\intercal}$ at $\bar{g}^\intercal$.
\end{Proposition}

Note that $DH|_{\bar{g}}(2v\bar{g})= -vH + 2\nu(v)$ by conformal transformation. While there is no explicit formula for the other term $D \mathring{H}|_{\bar{g}^\intercal}\big(2v\bar{g}^\intercal\big)$ (since its definition involves  isometric embeddings), we know that $D \mathring{H}|_{\bar{g}^\intercal}\big(2v\bar{g}^\intercal\big)$ depends only on the restriction of $v$ on $\Sigma$. Thus, Proposition~\ref{proposition:harmonic-function} reveals  new positivity relating the Dirichlet and Neumann boundary data of a harmonic function and may be of independent interest. Proposition~\ref{proposition:harmonic-function} is  essential to prove the following main theorem.

\begin{Theorem}\label{theorem:static-convex}
Let  $\Omega$ be a bounded open subset in $\mathbb{R}^3$ whose boundary $\Sigma=\partial \overline{\Omega}$ has positive Gauss curvature. Then $\Sigma$ is static regular in~$\mathbb{R}^3\setminus \Omega$.
\end{Theorem}

Our next result implies that any embedded   hypersurface can be perturbed to become static regular. We say that  a family of embedded hypersurfaces $\{ \Sigma_t\}\subset \mathbb{R}^n$ forms a smooth \emph{generalized foliation} if the deformation vector $X$ of $\{ \Sigma_t\}$ is smooth and on each $\Sigma_t$, $\bar g( X, \nu) =\zeta$ where $\zeta>0$, except a set of $(n-1)$-dimensional Hausdorff measure zero on $\Sigma_t$, and $\nu$ is the unit normal of $\Sigma_t$. In other words, $\{ \Sigma_t\}$ is slightly more general than a foliation in that they are allowed to touch each other from one-side on a set of measure zero.
 
\begin{Theorem}\label{theorem:static-generic}
Let $\delta>0$ and let each $\Omega_t\subset \mathbb{R}^n$ for $t\in [-\delta, \delta]$ be a bounded open subset with connected hypersurface boundary $\Sigma_t = \partial\overline{ \Omega_t}$  embedded  in $\mathbb{R}^n$.  Suppose  the boundaries $\{ \Sigma_t\}$ form a smooth generalized foliation. Then  there is an open dense subset $J\subset (-\delta,\delta)$ such that  $\Sigma_t$  is static regular in $\mathbb{R}^n\setminus \Omega_t$ for all $t\in J$.
\end{Theorem}

 By a dilation property of the Euclidean static vacuum pair, Theorem~\ref{theorem:static-generic} has a strong consequence.
\begin{Corollary}\label{cor:dilation}
Let  $\Omega$ be a bounded open subset in $\mathbb{R}^n$ whose boundary $\Sigma=\partial \overline{\Omega}$ is a star-shaped hypersurface. Then $\Sigma$ is static regular in $\mathbb{R}^n\setminus \Omega$. 
\end{Corollary}

Since a convex hypersurface is star-shaped, Corollary~\ref{cor:dilation} gives a second proof to Theorem~\ref{theorem:static-convex}.  To compare those two proofs, the dilation property used to prove Corollary~\ref{cor:dilation} is very specific to the Euclidean static vacuum pair, while the approach to Theorem~\ref{theorem:static-convex} can possibly be extended for a general static vacuum metric. While most of the proofs presented in this paper heavily use that the background metric is Euclidean, we are able to extend several of current results, including Theorem~\ref{theorem:criterion-main} and Theorem~\ref{theorem:static-generic}, to the setting that the background metric is an arbitrary asymptotically flat, static vacuum metric by incorporating new arguments in~\cite{An-Huang:2022}. 

We make some remarks about the geometric assumptions on $\Sigma$. For all of our results, we assume the boundary $\Sigma$ is an embedded hypersurface in $\mathbb{R}^n$. As is pointed out in \cite{Anderson-Khuri:2013, Anderson-Jauregui:2016}, if the boundary $\Sigma=\partial \overline{\Omega}$  is only \emph{inner} embedded, i.e., $\Sigma$ touches itself from the exterior region $\mathbb{R}^n\setminus \Omega$, the induced data $(\bar{g}^\intercal, H_{\bar{g}})$ is valid Bartnik boundary data, but $(\bar{g}, 1)$ in $\mathbb{R}^n\setminus \Omega$ is \emph{not} a valid static vacuum extension as $\mathbb{R}^n\setminus \Omega$ is not a manifold with boundary. It is clear that our arguments  don't work without embeddedness. We note that those inner embedded hypersurfaces are conjectured to be counter-examples to Conjecture~\ref{conjecture} by \cite[Conjecture 5.2]{Anderson-Jauregui:2016} (see also \cite{Anderson:2019}). On the other hand, Theorem~\ref{theorem:static-generic} implies that those possible counter-examples can be perturbed inward (with arbitrarily small perturbation) to be static regular hypersurfaces. 

Note the assumption that $\Sigma$ is connected in Theorem~\ref{theorem:criterion-main}. While Proposition~\ref{proposition:harmonic-function}, Theorems \ref{theorem:static-convex} and~\ref{theorem:static-generic} can be extended for $\Sigma$ of multiple components, we still need the connectedness assumption to derive the existence of static vacuum extensions by Theorem~\ref{theorem:criterion-main}. This restriction seems  related to  Static $\mathtt{n}$-body Conjecture. It asserts that an asymptotically flat static metric with $\mathtt{n}$ boundary components for $\mathtt{n}> 1$ must be trivial, under some separation condition of the boundary components. See, e.g.~\cite{Beig-Schoen:2009}. Therefore, it is expected to be considerably harder, if possible at all, to find a static vacuum extension with arbitrarily prescribed Bartnik boundary data when $\Sigma$ is not connected.  

Lastly, we remark that in general the ADM mass of the metric constructed Theorem~\ref{theorem:criterion-main} may not be positive. The phenomenon seems to be closely related to another outstanding problem whether there exists a ``fill-in'' metric of nonnegative scalar curvature with prescribed Bartnik boundary data. Nevertheless, in the special situation that $\tau= \bar g^\intercal$ and $\phi \lneq H_{\bar g}$, the constructed static metric~$g$ in $\mathbb{R}^n\setminus \Omega$ clearly has a fill-in metric  $\bar g$ in $\overline{\Omega}$ (the Euclidean metric) whose  Bartnik boundary data  satisfy $g^\intercal = \bar g^\intercal$ and $H_g \lneq H_{\bar g}$. The ADM mass of such $g$ must be positive by Miao's positive mass theorem with corners \cite{Miao:2002}.

The paper is organized as follows. In Section~\ref{section:prep}, we give basic definitions and formulas of linearized curvatures, and we identify Ricci flat deformations on a Euclidean exterior region. In Section~\ref{section:Green}, we consider variations of the Regge-Teitelboim Hamiltonian to obtain a \emph{Green-type identity} and other identities essential for the main results. Section~\ref{section:existence} deals with Theorem~\ref{theorem:criterion-main}. In Section~\ref{section:convex}, we prove Proposition~\ref{proposition:harmonic-function} and apply it to prove Theorem~\ref{theorem:static-convex}. In Section~\ref{section:generic}, we prove Theorem~\ref{theorem:static-generic} and Corollary~\ref{cor:dilation}.

\subsection*{Acknowledgement}
The first author would like to thank Michael Anderson for valuable discussions.   In a previous version of the paper, Corollary~\ref{cor:dilation} was stated only for round spheres, and we are very grateful to Pengzi Miao for notifying us that the dilation argument extends to all star-shaped hypersurfaces.
 
\section{Preliminaries}\label{section:prep}
\subsection{Basic definitions and notations}\label{section:basic}

Let $\mathbb{R}^n$ be the $n$-dimensional space with the Cartesian coordinates $ x=(x_1,\dots, x_n)$, and $n\ge 3$. Throughout the paper we shall fix the coordinate chart $\{ x \}$ and reserve the notation  $\bar{g}$ for the Euclidean metric given by $\bar{g} =( \delta_{ij})$ in this coordinate chart.  Let $\Omega\subset \mathbb{R}^n$ be a bounded open subset whose boundary~$\Sigma = \partial \Omega$ is a compact, embedded, smooth\footnote{Most of our arguments directly extend for $\C^{3,\alpha}$ hypersurfaces while some others may require extra care to pinpoint the optimal regularity, e.g. Theorem~\ref{theorem:Shi-Tam}, which we do not pursue here.} hypersurface. We assume that  $\mathbb{R}^n\setminus \Omega$ is connected. Most of results in this section are valid for $\Sigma$ having multiple components, except Corollary~\ref{corollary:trivial}. 

Let $g$ be a $\C^2_{\mathrm{loc}}$ Riemannian metric on $\mathbb{R}^n\setminus \Omega$.  For a tensor $h$ on $\mathbb{R}^n\setminus \Omega$, we denote by $h^\intercal$  the restriction of $h$ on the tangent bundle of~$\Sigma$. We define the second fundamental form $A_g=\tfrac{1}{2} (L_{\nu_g} g)^\intercal $ and the mean curvature  $H_g = \Div_\Sigma \nu_g$ on $\Sigma$, where $\nu_g$ is the unit normal vector pointing \emph{away} from~$\Omega$.

We will frequently work with variations of curvatures along a family of Riemannian metrics $g(s)$ on $\mathbb{R}^n\setminus \Omega$ with $g(0) = g$.  We shall denote the linearization of the second fundamental form $\left.\ds\right|_{s=0} A_{g(s)} $ by $DA|_g(h)$ where $h = g'(0)$. All other linearized curvature operators are denoted in the same fashion. We give a list of formulas of the linearized operators in the following lemma.  Consider a local frame $\{ e_0, e_1, \dots, e_{n-1}\}$ in a collar neighborhood of $\Sigma$ in $\mathbb{R}^n\setminus \Omega$ such that $e_0$ is the parallel extension of $\nu_g$ along itself and $e_1, \dots, e_{n-1}$ are tangent to $\Sigma$. We use the subscript $_{;i}$ to denote the covariant derivative by $e_i$ with respect to $g$, and all the indices are raised or lowered by $g$. 
\begin{lemma} \label{lemma:formula}
The linearizations of the Ricci tensor and scalar curvature are given by
\begin{align}
\begin{split}\label{equation:Ricci}
	(D\Ric|_g(h))_{ij} &=-\tfrac{1}{2} g^{k\ell} h_{ij;k\ell} + \tfrac{1}{2} g^{k\ell} (h_{i k; \ell j} + h_{jk; \ell i} ) \\
	&\quad - \tfrac{1}{2} (\tr h)_{ij}  + \tfrac{1}{2} (R_{i\ell} h^\ell_j + R_{j\ell} h^\ell_i )\\
	&\quad - R_{ik\ell j} h^{k\ell} \qquad \qquad  \mbox{ for all } i, j =0,1,\dots, n-1
\end{split}\\
\begin{split}\notag
	D R|_g (h)
	&=-\Delta \tr h + \Div \Div h - h \cdot \Ric_g.
\end{split}
\end{align}
Here, ``$\cdot$'' denotes the inner product with respect to $g$. 

We write $\nu = \nu_g$ for short. We let $\omega(e_a) = h(\nu, e_a)$ be the one-form on the tangent bundle of $\Sigma$ and $\{\Sigma_t\}$ be the foliation by $g$-equidistant hypersurfaces to $\Sigma$.
For tangential directions $a, b, c \in \{ 1, \dots, n-1\}$ on $\Sigma$, we have 
\begin{align}
	D\nu|_g(h) &= - \tfrac{1}{2} h(\nu, \nu) \nu - g^{ab} \omega(e_a) e_b\label{equation:normal}\\
\begin{split}
	DA|_g(h) & = \tfrac{1}{2} (L_{\nu} h)^\intercal - \tfrac{1}{2} L_\omega g^\intercal - \tfrac{1}{2} h(\nu, \nu) A_g\\
	&= \tfrac{1}{2}( \nabla_{\nu} h)^\intercal+ A_g\circ h - \tfrac{1}{2} L_\omega g^\intercal- \tfrac{1}{2} h(\nu, \nu) A_g\label{equation:linearized-sff}
	\end{split}\\
	DH|_g(h) &=\tfrac{1}{2} \nu(\tr h^\intercal) - \Div_\Sigma \omega - \tfrac{1}{2} h(\nu, \nu) H_g\notag.
\end{align}
In the third line $(A_g\circ h)_{ab} = \tfrac{1}{2} (A_{ac} h^c_{b}+A_{bc} h^c_{a})$,  and in the last equation, $\tr h^\intercal$ denotes the tangential trace of $h$ on $\Sigma_t$, defined in a collar neighborhood of~$\Sigma$. 

For any family  $g(s)$ with $g(0) = g$, we let $H_{g(s)}$, defined in the collar neighborhood, be the mean curvature of $\Sigma_t$ with respect to $g(s)$. Then on $\Sigma$, 
\begin{align*}
	\nu (DH|_g(h))& =  - \tfrac{1}{2} DR|_g(h) + \tfrac{1}{2} DR^\Sigma|_{g^\intercal} (h^\intercal) + A_g \cdot (A_g \circ h) \\
	&\quad  - A_g \cdot DA|_g(h  -H_g DH|_g(h) \\
	&\quad + \tfrac{1}{4}\big( - R_g +R^\Sigma_g -  |A_g|^2 - H_{g}^2\big)h(\nu, \nu)\\
	&\quad 	- \tfrac{1}{2}\Delta_\Sigma h(\nu, \nu)  + g^{ab} \omega (e_a) e_b(H_g)
\end{align*}
where $R^\Sigma_g, \Delta_\Sigma$ are respectively the scalar curvature and the Laplace operator of  $(\Sigma, g^\intercal)$. Consequently, if we assume $h(\nu, \nu)=0$ and $\omega=0$, then 
\begin{align}\label{equation:Riccati}
\begin{split}
\nu (DH|_g(h))&=-\tfrac{1}{2} DR|_g(h)+\tfrac{1}{2} DR^\Sigma|_{g^\intercal} (h^\intercal)  + A_g\cdot(A_g\circ h )  \\
&\quad - A_g\cdot DA|_g(h) - H_gDH|_g(h).
\end{split}
\end{align}
\end{lemma}
\begin{proof}
The formulas for $D\Ric|_g(h), D R|_g (h)$ can be found in, for example, \cite[Lemma 2]{Fischer-Marsden:1975}.

The formula for the linearized unit normal vector can be derived by linearizing $g(\nu, \nu)=1$ and $g(\nu, e_a)=0$. For the linearized second fundamental form, we linearize $A_g =\tfrac{1}{2} (L_\nu g)^\intercal $ to get 
\begin{align*}
	(DA|_g(h))_{ab} &= \tfrac{1}{2} \left[ (L_\nu h)_{ab}+ (L_{D\nu|_g(h)} g )_{ab}\right]\\
&=   \tfrac{1}{2}(L_\nu h)_{ab} - \tfrac{1}{2}\left[ e_a(\omega( e_b) )+ e_b (\omega( e_a))\right] - \tfrac{1}{2} h(\nu, \nu) A_{ab}
\end{align*}
where we have used the formula for $D\nu|_g(h)$ to get
\[
 	(L_{D\nu|_g(h)} g )_{ab}=  - e_a( \omega( e_b))-e_b( \omega( e_a))- h(\nu, \nu) A_{ab}.
\]	
The formula for $DH|_g$ follows by linearizing $H_g = g^{ab} A_{ab}$ and applying the formula for $DA|_g(h)$. 

We now compute $\nu(DH|_g(h))$. We express $\nu = \eta(s) \nu_{g(s)} + \nu^\intercal (s)$ on $\Sigma$ into the normal and tangential components with respect to $g(s)$. So
\begin{align*}
	\eta(0)=1,\quad \nu^\intercal (0)= 0, \quad \eta'(0) =  \tfrac{1}{2} h(\nu, \nu),
	 \quad (\nu^\intercal)'(0)= g^{ab} \omega( e_a) e_b.
\end{align*}
By the variation formula of mean curvature, we have
\begin{align*}
	\nu(H_{g(s)}) &= - \big(|A_{g(s)}|_{g(s)}^2 + \mathrm{Ric}_{g(s)}(\nu_{g(s)}, \nu_{g(s)}) \big)\eta(s)\\
	&\quad - \Delta_{(\Sigma, g(s))} \eta(s) + \nu^\intercal (s)(H_{g(s)})\\
		&= \left(  - \tfrac{1}{2} R_{g(s)} + \tfrac{1}{2} R^\Sigma_{g(s)} - \tfrac{1}{2} |A_{g(s)}|_{g(s)}^2 - \tfrac{1}{2} H_{g(s)}^2\right) \eta(s)\\
		&\quad  - \Delta_{(\Sigma, g(s))} \eta(s) + \nu^\intercal (s) (H_{g(s)})
\end{align*}
where in the second line we replace the Ricci term using the Gauss equation. Differentiating the previous identity in $s$ at $s=0$ yields the desired formula. 
\end{proof}

We also include the following general fact on the linearized mean curvature of the pull-back metric, which can be directly verified.  
\begin{lemma}\label{lemma:pullback}
Let $U_1, U_2$ be complete smooth manifolds with boundary $\Sigma_1, \Sigma_2$ respectively, and let $g$ be a Riemannian metric on $U_2$.  Let $\psi: U_1\to U_2$ be a diffeomorphism that sends $\Sigma_1$ to $\Sigma_2$.  Let $H_g$ be the mean curvature of $\Sigma_2 \subset (U_2, g)$ of $\nu_g$ and $H_{\psi^*g} $ be the mean curvature of $\Sigma_1\subset (U_1, \psi^*g)$ of $\psi^* \nu_g$. Then, $H_{\psi^*g} = \psi^*H_g$, and for any deformation $h$ of $g$ in $U_2$, 
\begin{align*}
	DH|_{\psi^* g}	(\psi^* h) = \psi^* \left( DH|_{g} (h)\right).
\end{align*} 
\end{lemma}

For a Riemannian metric $g$  and a function $u$ in $\mathbb{R}^n\setminus \Omega$,  we define 
\[
	S(g, u) = -u\Ric_g + \nabla^2_g u - (\Delta_g u )g.
\]
The right hand side is equal to $DR|_g^*(u)$, the formal $\mathcal{L}^2$-adjoint of $DR|_g$, but  our notation $S(g, u)$ emphasizes that we will treat both $(g, u)$ as unknowns. 
We say that the pair $(g, u)$ is \emph{static vacuum} in $U$ if $(g, u)$ solves the following static vacuum system:
\begin{align*}
	S(g, u) = 0 \quad \mbox{ and } \quad	R_g=0\quad \mbox{ in }\mathbb R^n\setminus\Omega.
\end{align*}
This system is equivalent to \eqref{sv}.

For a symmetric $(0,2)$-tensor $h$ and a function $v$, we say that $(h,v)$ is a {\it static vacuum deformation} at the static vacuum pair $(g,u)$ if it satisfies the linearized static vacuum system:
\begin{align}\label{lsv}
	DS|_{(g, u)}(h, v)=0  \quad \mbox{ and } \quad DR|_g(h)=0\quad \mbox{ in }\mathbb R^n\setminus\Omega.
\end{align}
If the Ricci curvature of $g$ is zero,  we say that $h$ is a \emph{Ricci flat deformation} (at $g$) if $D\Ric|_g(h)=0$. In general, the Ricci flat deformation can be thought as a subclass of static vacuum deformations at a Ricci flat metric $g$ by letting $u=1$ and $v= 0$. 
On the other hand,
in the special case  that $(g, u)=(\bar{g}, 1)$, any static vacuum deformation $(h, v)$ gives rise to a Ricci flat deformation $h+\tfrac{2}{n-2} v\bar{g}$. 
This is because the linearized system \eqref{lsv} at $(\bar g, 1)$ is equivalent~to 
\begin{align*}
	-D\Ric(h) + \nabla^2 v =0 \quad \mbox{ and  } \quad \Delta v=0 \quad\mbox{ in }\mathbb R^n\setminus\Omega,
\end{align*}
where we omit the subscript $\bar{g}$,  and it 
is also equivalent to the system
\begin{align*}
D\Ric\big(h+\tfrac{2}{n-2} v\bar{g}\big)=0\quad \mbox{ and } \quad \Delta v=0 \quad \mbox{ in }\mathbb R^n\setminus\Omega.
\end{align*}

We will consider the following two types of boundary conditions of a deformation $h$ on $\Sigma$:
\begin{itemize}
\item \emph{Bartnik boundary condition}:
\[
	h^\intercal=0 \quad \mbox{ and } \quad DH|_g (h)=0 \quad \mbox{ on } \Sigma. 
\]
\item \emph{Cauchy boundary condition}: 
\[
	h^\intercal =0\quad \mbox{ and } \quad DA|_g (h)=0 \quad \mbox{ on } \Sigma. 
\]
\end{itemize}
Clearly, the Cauchy boundary condition implies the Bartnik boundary condition. For a general  deformation $h$, we may refer $(h^\intercal, DH|_g(h))$ as its Bartnik boundary data and   $(h^\intercal, DA|_g(h))$ as its Cauchy boundary data. (In this paper we slightly abuse the terminology and use ``Bartnik boundary data'' for $(g^\intercal, H_g)$ and also the linearization $(h^\intercal, DH|_g(h))$, but it should be clear from context which one is referred to.)

Let $k\in \{ 0, 1, 2, \dots\}$,  $\alpha\in (0, 1)$, and $q\in \mathbb{R}$. For a $\C^{k,\alpha}_{\mathrm{loc}}$-function $f$ in $\mathbb{R}^n\setminus \Omega$, we define the $\C^{k,\alpha}_{-q} (\mathbb{R}^n\setminus \Omega)$-norm  by
\begin{align*}
	\|f\|_{\C^{k,\alpha}_{-q} (\mathbb{R}^n\setminus \Omega)} =& \textstyle\sum_{|I|=0,\dots, k} \sup_{x\in \mathbb{R}^n \setminus \Omega} (1+|x|)^{|I|+q} |\partial^I f(x)|  \\
	&\quad +\textstyle\sum_{|I|=k} \sup_{\substack{x, y\in \mathbb{R}^n\setminus \Omega\\0<|x-y| \le\frac{ |x|}{2}}} (1+|x|)^{\alpha+|I|+q} \frac{|\partial^I f(x) - \partial^I f(y)|}{|x-y|^\alpha}.
\end{align*} 
The \emph{weighted H\"older space} $\C^{k,\alpha}_{-q} (\mathbb{R}^n\setminus \Omega)$ consists of those functions whose $\C^{k,\alpha}_{-q} (\mathbb{R}^n\setminus \Omega)$-norms are finite. We slightly abuse the notation and say that a tensor is in $\C^{k,\alpha}_{-q}(\mathbb{R}^n\setminus \Omega)$ if all of its coefficient functions in the Cartesian coordinates are in $\C^{k,\alpha}_{-q}(\mathbb{R}^n\setminus \Omega)$. We may sometimes write $f= O^{k,\alpha}(|x|^{-q})$ for $f\in \C^{k,\alpha}_{-q}(\mathbb{R}^n\setminus \Omega)$ to emphasize the asymptotic rate.

For the rest of the paper, we shall always assume
\[
	q\in \left(\tfrac{n-2}{2}, n-2\right).
\]
We say that a Riemannian metric $g$ on $\mathbb{R}^n\setminus \Omega$ is \emph{asymptotically flat} if  $g-\bar g \in \C^{2,\alpha}_{-q}(\mathbb{R}^n\setminus \Omega)$. The \emph{ADM mass of $g$} is defined as
\begin{align*}
	m_{\mathrm{ADM}}(g)&= \frac{1}{2(n-1)\omega_{n-1}} \lim_{r\to \infty} \int_{|x|=r}\sum_{i,j=1}^n \left(\frac{\partial g_{ij}}{\partial x_i} - \frac{\partial g_{ii}}{\partial x_j}\right)\frac{x_j}{|x|} \, d\sigma,
\end{align*}	
where $d\sigma$ is the $(n-1)$-dimensional Euclidean Hausdorff measure and $\omega_{n-1}$ is the volume of the standard $(n-1)$-dimensional unit  sphere. 

If $g$ is asymptotically flat and $u$ is a scalar-valued function such that $u-1\in \C^{2,\alpha}_{-q}(\mathbb{R}^n\setminus \Omega)$, we will refer to $(g, u)$ as an  \emph{asymptotically flat pair}. We denote by $\mathcal{M}$ an open neighborhood of $(\bar{g}, 1)$ in the $\C^{2,\alpha}_{-q}(\mathbb{R}^n\setminus \Omega)$-norm, consisting of asymptotically flat pairs.

We will frequently use the following basic PDE lemma. See~\cite{Lockhart-McOwen:1985, Bartnik:1986}.
\begin{lemma}\label{lemma:PDE}
Let $k\ge 2$, $\delta\in \mathbb{R}$, and let $\Omega\subset \mathbb{R}^n$ be a bounded open subset whose boundary $\Sigma$ is an embedded hypersurface. Suppose $\mathbb{R}^n\setminus \Omega$ is connected (but $\Sigma$ may be disconnected). Let $g$ be an asymptotically flat Riemannian metric on $\mathbb{R}^n\setminus \Omega$ with $g- \bar g\in \C^{k,\alpha}_{-q}$. Consider the map on the space of vector fields
\[
\Delta_g:  \left\{ X\in \C^{k,\alpha}_{\delta} (\mathbb{R}^n\setminus \Omega) : \, X=0 \mbox{ on } \Sigma\right \} \to \C^{k-2,\alpha}_{\delta-2}(\mathbb{R}^n\setminus \Omega)
\]
that sends $X$ to $\Delta_g X$. Then 
\begin{enumerate}
\item For $0< \delta< 1$, the map is surjective and the kernel space is $n$-dimensional, spanned by $\{ X^{(1)}, \dots, X^{(n)}\}$ where $X^{(i)} = \frac{\partial}{\partial x_i} + O^{k,\alpha}(|x|^{-q})$.
\item For $2-n<\delta <0$, the map is an isomorphism. 
\item For $\delta<2-n$, the map is injective. 
\end{enumerate}
\end{lemma}

\subsection{Ricci flat deformations}\label{section:gauge}

Our main result in this section is Theorem~\ref{theorem:trivial} below where we identify all Ricci flat deformations $h \in \C^{2,\alpha}_{-q}(\mathbb{R}^n\setminus \Omega)$ that satisfy the Cauchy boundary condition. We remark that  Anderson and Herzlich  have proved that,  under the $H$-harmonic gauge, a Ricci flat metric $g$ is uniquely determined by its data $g^\intercal$ and $A_g$ on the boundary~\cite[Theorem 1.1]{Anderson-Herzlich:2008, Anderson-Herzlich:erratum}.  Implementing the similar idea, the first author proved that a Ricci flat deformation $h$, under the $H$-harmonic gauge, satisfying the Cauchy boundary condition must be zero~\cite[Proposition 4.4]{An:2020}. We will give an alternative and elementary proof that fits to our specific setting.

We begin with examples of  Ricci flat deformations satisfying the Cauchy boundary condition.

\begin{example}\label{example:vector}
Let  $g$ be a Riemannian metric on a smooth manifold $U$ satisfying  $\Ric_g=0$. Given a vector field $X$, let $\psi_s$ be a family of diffeomorphisms of $U$ satisfying $\left.\ds\right|_{s=0}\psi_s=X$. Define the family  $g(s)=\psi_s^*(g)$ and note $g'(0)= L_X g$. We call the Lie derivative $L_X g$ an \emph{(infinitesimal) deformation generated from diffeomorphisms}. Since $\Ric_{g(s)} = \psi_s^* \Ric_g=0$, differentiating in $s$ gives $D{\Ric}|_g(L_Xg)=0$. That is, at a Ricci flat metric $g$, any deformation generated from diffeomorphisms is a Ricci flat deformation.

Next, we find the condition of $X$ on the boundary $\Sigma$ such that $L_X g$ satisfies the Cauchy boundary condition. We say that $K$ is a \emph{Killing vector field} of $(U, g)$ if $L_K g=0$ in $U$.  Let $X$ satisfy $X = K$ on $\Sigma$ for some Killing vector $K$, including the zero vector field. 
We shall see that such $L_X g$ satisfies the Cauchy boundary condition. Let $\{ e_0, e_1,\dots, e_{n-1}\}$ be a local frame in $U$ where on $\Sigma$, $e_0=\nu$ is a unit normal and $\{e_1,\dots, e_{n-1}\}$ tangent to $\Sigma$. We decompose $X$ as $X = \eta \nu +  X^\intercal$ on $\Sigma$. The Cauchy boundary data can be expressed as, for $a, b, c= 1, \dots, n-1$, 
\begin{align*}
\begin{split}
	(L_X g)^\intercal &= 2\eta A + L_{X^\intercal} g^\intercal\\
	 DA|_g (L_X g) (e_a, e_b)&=\mathrm{Rm}(X, e_a, \nu, e_b)- \eta_{|a b} +\tfrac{1}{2} \eta (A_a^c A_{c b} + A_b^c A_{c a}) \\
	&\quad +  (X^\intercal)^c A_{c b|a} + (X^\intercal)^c_{|b} A_{ac} + ( X^\intercal)^c_{|a}A_{b c}
\end{split}
\end{align*}
where the subscript $_{|a}$ denotes the covariant derivative of the induced metric $g^\intercal$ on $\Sigma$. It follows immediately from those formulas that $(L_X g)^\intercal$  and $ DA|_g (L_X g) $  are entirely determined by $\eta$ and $X^\intercal$ on $\Sigma$. Thus if $X=K$ on $\Sigma$, then
$
(L_X g)^\intercal = ( L_K g)^\intercal =0,~
 DA|_g (L_X g)= DA|_g (L_K g) =0.
$

\qed
\end{example}

The example above shows that, at a Ricci flat metric $g$, for any vector field $X$ equal to some Killing vector of $g$ on the boundary, $L_X g$ is a Ricci flat deformation satisfying the Cauchy boundary condition.  Before we  proceed to study Ricci flat deformations in greater generality, we will consider two types of gauge conditions on $h$: the \emph{geodesic gauge} on $\Sigma$ and the \emph{harmonic gauge} in $U$. Both gauge conditions are well-studied in literature,  but for completeness we include the next two lemmas.  (Note we only use those lemmas for $g= \bar g$ the Euclidean metric in this paper.)

We say that a deformation $h$  satisfies the \emph{geodesic gauge condition on $\Sigma$} (at $g$) if 
\begin{align}\label{geodesic gauge}
h(\nu_g, \cdot)= 0,\quad (\nabla_{\nu_g} h)(\nu_{g}, \cdot)=0\quad\mbox{ on }\Sigma,
\end{align}
where  $\nu_g$ is parallelly extended along itself.  The ``dot'' means  $h(\nu_g, e_i )= 0$  and $(\nabla_{\nu_g} h)(\nu_{g}, e_i)=0$ on $\Sigma$ for any vector $e_i$, either tangential or normal.
\begin{lemma}[Geodesic gauge]\label{lemma:geodesic}
Let $h\in \C^{2,\alpha}(U)$ be a symmetric $(0,2)$-tensor. Then there exists a vector field $V\in \C^{3,\alpha}(U)$ with $V=0$ on $\Sigma$ and $V$ vanishing outside a collar neighborhood of $\Sigma$ such that $k:= h+L_Vg$ satisfies the geodesic gauge condition and that $h, k$ have the same Cauchy boundary data on $\Sigma$.
\end{lemma}
\begin{remark}\label{remark:geodesic}
It explains the reason that $h^\intercal=0,  DA|_g (h)=0$ on $\Sigma$ is referred to as the Cauchy boundary condition. For such $h$,  we let $k = h+L_V g$ as in the lemma. By \eqref{equation:linearized-sff},  $k$ must satisfy the Cauchy boundary condition on $\Sigma$ in the usual sense: for all $i, j=0,1, \dots, n-1$,
\[
k_{ij}=0\quad \mbox{ and } \quad ( \nabla_{\nu_g} k)_{ij}=0\quad \mbox{ on } \Sigma. 
\]
\end{remark}
\begin{proof}
Cover a collar neighborhood of $\Sigma\subset U$ by countably many open subsets $U_\ell$ such that for each $\ell$ there is a local orthonormal smooth frame $\{ e_0, e_1, \dots, e_{n-1}\}$ in $U_\ell$ where $e_0  =\nu_g$ is parallel along itself, and  $e_1, \dots, e_{n-1}$ are tangent to $\Sigma$. We will simply write $\nu=\nu_g$. Let $\{\xi_\ell\}$ be a partition of unity subordinate to $\{ U_\ell\}$.

 Suppose that $W =\textstyle\sum_{i=0}^{n-1} \eta_i e_i$ is a vector field in $U_\ell$ where the  coefficient function $\eta_i=0$ on $\Sigma$ for all $i$. The following identities hold on $U_\ell \cap \Sigma$, where $a= 1,\dots, n-1$ denotes the tangential directions:
\begin{align}\label{equation:geodesic-boundary}
\begin{split}
	L_W g (\nu, \nu) &= 2\nu(\eta_0)\\
	L_W  g(\nu, e_a) &= \nu(\eta_a) \\
	(\nabla_\nu (L_W g))(\nu, \nu)&=2\nu (\nu(\eta_0))\\
	(\nabla_\nu (L_W g))(\nu, e_a)&=\nu(\nu(\eta_a)) + 2\nu(\eta_b) g(\nabla_\nu e_b, e_a)\\
	&\quad  +e_a( \nu(\eta_0)) - \nu(\eta_b) A_{ab},
\end{split}
\end{align}
where $A_{ab} = g(\nabla_{e_a} \nu, e_b)$. 

We proceed to define the vector field $V$. We first define a local vector field  $W^{(\ell)}$ on each $U_\ell$. Let  $W^{(\ell)}=\sum_{i=0}^{n-1}\eta_i e_i$ (we omit the superscript $\ell$ in those $\eta_i$). Choose $\eta_i\in \C^{3,\alpha}(U_\ell)$ such that it vanishes outside a collar neighborhood of $\Sigma$, and 
\begin{align*}
\eta_i=0,\;\; 2\nu(\eta_0)= -h(\nu,\nu)\;\; \mbox{ and } \;\; \nu(\eta_a) = -h(\nu, e_a),\quad a=1,\dots, n-1
\end{align*}
on $U_\ell\cap\Sigma$. We let $W= \sum_\ell \xi_\ell W^{(\ell)}$ on $U$.

Next, we define  $\widetilde{W}^{(\ell)}=\sum_{i=0}^{n-1}\widetilde\eta_ie_i$ on $U_\ell$ (we omit the superscript $\ell$ in those $\widetilde\eta_i$). Choose $\widetilde\eta_i\in \C^{3,\alpha}(U_\ell)$ which vanishes outside a collar neighborhood of $\Sigma$, and
\begin{align*}
        \widetilde\eta_i&=0\\
        \nu(\widetilde\eta_i)&=0 \\
	2\nu (\nu(\widetilde\eta_0))&= -\nabla_\nu (h+L_Wg)(\nu, \nu)\\
	\nu(\nu(\widetilde\eta_a)) &=- \nabla_\nu (h+L_Wg)(\nu, e_a), \qquad a=1,\dots, n-1
\end{align*}
on $U_\ell\cap\Sigma$.  We define $\widetilde{W}=\sum_\ell \xi_\ell \widetilde W^{(\ell)}$ and let $V= W +\widetilde W$ on $U$. By those identities in  \eqref{equation:geodesic-boundary}, it is direct to verify that $V$ satisfies the desired properties.
\end{proof}



We define some differential operators which will be used throughout this paper: For a  symmetric $(0,2)$-tensor $h$ and a vector field $V$, we define the \emph{Bianchi operator} $\bi_g$, the \emph{Killing operator} $\D_g$, and the $\mathcal{L}^2$-formal adjoint  $\bi_g^*$ by 
\begin{align}\label{equation:operators}
\begin{split}
	\bi_g h &= - \Div_g h+ \tfrac{1}{2} d(\tr_g h)\\
	\D_g V &= \tfrac{1}{2} L_V g\\
	\bi^*_g V&= \tfrac{1}{2} \big( L_V g - (\Div_{g} V) g \big).
\end{split}
\end{align}
We say that \emph{$h$ satisfies the harmonic gauge} (at $g$) if $\bi_g h=0$.  By Ricci identity, for any vector field $V$,  $\bi_g (L_V g)$ becomes an elliptic operator on $V$:
\begin{align} \label{equation:Bianchi-vector}
	(\bi_g L_V g)_i = -(\Delta_g V)_i - \Ric_{ij} V^j.
\end{align}
\begin{lemma}[Harmonic gauge]\label{lemma:harmonic-gauge}
There is $\epsilon>0$ such that if $g$ is asymptotically flat with $g-\bar g \in \C^{3,\alpha}_{-q}(\mathbb{R}^n\setminus \Omega)$ satisfying $\|g-\bar{g}\|_{\C^{1,\alpha}_0(\mathbb{R}^n\setminus \Omega)}<\epsilon $  and 
$\|\Ric_g\|_{\C^{0,\alpha}_{-2}(\mathbb{R}^n\setminus \Omega)}< \epsilon$, then for any $h\in \C^{2,\alpha}_{-q}(\mathbb{R}^n\setminus \Omega)$, there is a vector field $V\in \C^{3,\alpha}_{1-q}(\mathbb{R}^n\setminus \Omega)$ with $V=0$ on $\Sigma$  such that if we define $k= h+L_Vg$, then $\bi_g k=0$ in $\mathbb{R}^n\setminus \Omega$, and $h$ and $k$ have the same Cauchy boundary data.
\end{lemma}

\begin{proof}
We show that there is a vector field  $V$ solving
\begin{alignat*}{2}
	\bi_g L_V g &= -\bi_g h \quad && \mbox{ in } \mathbb{R}^n\setminus \Omega\\
	V&=0\quad && \mbox{ on } \Sigma. 
\end{alignat*}
The first equation can be expressed as $-\Delta_g V - \Ric_g(V, \cdot) = -\bi_g h$  by \eqref{equation:Bianchi-vector}. Since the operator $\Delta$, at the Euclidean metric, is surjective by Lemma~\ref{lemma:PDE}, for $\epsilon$ sufficiently small,  the operator $-\Delta_g - \Ric_g(\cdot, \cdot)$ is close to $-\Delta$ in the operator norm and thus is also surjective. Then $V \in \C^{3,\alpha}_{1-q}(\mathbb{R}^n\setminus \Omega)$ by elliptic regularity. 
\end{proof}

Now we analyze the Ricci flat deformation on a Euclidean exterior region $(\mathbb{R}^n\setminus \Omega, \bar{g})$ and prove the main result of this section. Below we shall omit the subscript $\bar{g}$.  For example, we write $\bi =\bi_{\bar{g}}, \D = \D_{\bar{g}}, \bi^* =\bi^*_{\bar{g}}$, and 
\begin{align*}
	D\Ric(h) = D\Ric|_{\bar{g}}(h)\quad	DA(h) = DA|_{\bar{g}}(h).
\end{align*}
Using the above  notations and \eqref{equation:Ricci}, we can write 
\begin{align}\label{equation:Ricci-gauge}
	D\Ric(h) = - \tfrac{1}{2} \Delta h - \D \bi h.
\end{align}
Denote the formal $\mathcal{L}^2$-adjoint operator of  $D\Ric$ by $(D\Ric)^*$. Then for a symmetric $(0,2)$-tensor $\gamma$, 
\begin{align}\label{equation:adjoint-Ricci}
	(D\Ric)^* (\gamma) =-\tfrac{1}{2} \Delta \gamma + \bi^* (\Div \gamma).
\end{align}
We say that $h$ \emph{weakly solves $D\Ric(h)=0$ in $\mathbb{R}^n$} if for any symmetric $(0,2)$-tensor $\gamma\in \C^\infty_c(\mathbb{R}^n)$, 
\begin{align*}
	\int_{\mathbb{R}^n} h\cdot (D\Ric)^* (\gamma) \, d\mathrm{vol}=0.
\end{align*}
We  define that $h$ weakly solves $\bi h=0$ similarly. 

\begin{theorem}\label{theorem:trivial}
Let $\Omega\subset \mathbb{R}^n$ be a bounded open subset whose boundary $\Sigma = \partial \overline{\Omega}$ is an embedded hypersurface such that $\mathbb{R}^n\setminus \Omega$ is connected (but $\Sigma$ may be disconnected). If $h\in \C^{2,\alpha}_{-q}(\mathbb{R}^n \setminus \Omega)$ solves
\[ D\Ric(h)= 0\mbox{ in }\mathbb{R}^n \setminus \Omega,\quad \mbox{ and } \quad h^\intercal = 0,~ DA(h) = 0\mbox{  on }\Sigma,\]
then
\[ 
h=L_X \bar{g} 
\] 
where $X \in \C^{3,\alpha}_{1-q}(\mathbb{R}^n\setminus \Omega)$  is equal to a Killing vector field of $\bar{g}$ (possibly the zero vector) on each connected component of $\Sigma$.
\end{theorem}

\begin{proof}
By Lemma~\ref{lemma:geodesic} and Remark~\ref{remark:geodesic}, there is a vector field $V\in \C^{3,\alpha}(\mathbb{R}^n\setminus \Omega)$ such that $V=0$ on $\Sigma$, $V$ vanishes outside a collar neighborhood of $\Sigma$, and $h+L_V\bar{g}$ satisfies the Cauchy boundary condition on $\Sigma$ in the classical sense: for $i, j=0,1,\dots, n-1$,  
\begin{align}\label{equation:Cauchy}
\begin{split}
	\big(h+L_V\bar{g}\big)(e_i, e_j) &= 0\\
	\big(\nabla_\nu (h+L_V\bar{g})\big) (e_i, e_j) &= 0.
\end{split}
\end{align}
We extend $h+L_V\bar{g}$ by defining $k\in \C^{1}_{-q}(\mathbb{R}^n)$ as
\begin{align*}
	k &= h+L_V\bar{g}\quad \mbox{ in } \mathbb{R}^n\setminus \Omega\\
	k &=0 \quad \mbox{ in } \Omega.
\end{align*}

Let $Z\in \C^{1,\alpha}_{1-q}(\mathbb{R}^n)$ be a vector field that weakly solves $\Delta Z = \bi k$ in $\mathbb{R}^n$. That is, by \eqref{equation:Bianchi-vector}, $k+L_Z\bar{g}$ is a weak solution to
\begin{align} \label{equation:gauge}
	\bi ( k+L_Z \bar{g} )=0 \quad \mbox{ in } \mathbb{R}^n.
\end{align}
We will show that  $k+L_Z\bar{g}$ is a weak solution to $\Delta (k+L_Z\bar{g})=0$ in $\mathbb{R}^n$. Observe that both $k, L_Z\bar{g}$ weakly solve the linearized Ricci equation in $\mathbb{R}^n$. For any $\gamma\in \C^\infty_c (\mathbb{R}^n)$,
\begin{align*}
	\int_{\mathbb{R}^n}  k \cdot (D\Ric)^* (\gamma) \, d\mathrm{vol}&= \int_{\mathbb{R}^n\setminus \Omega} (h+L_V \bar{g}) \cdot (D\Ric)^* (\gamma) \, d\mathrm{vol}\\
	&=  \int_{\mathbb{R}^n\setminus \Omega}  D\Ric (h+L_V \bar{g}) \cdot \gamma \, d\mathrm{vol}=0.
\end{align*}
In the second identity we use  integration by parts and note  the boundary integral on  $\Sigma$ vanishes due to the boundary condition \eqref{equation:Cauchy}. For $L_Z\bar{g}$, we integrate by parts and directly compute $\Div(D\Ric)^* (\gamma)=0$ to get
 \[
 	\int_{\mathbb{R}^n} L_{Z} \bar{g}\cdot (D\Ric)^* (\gamma)\, d\mathrm{vol}= \int_{\mathbb{R}^n} Z \cdot \Div(D\Ric)^* (\gamma)\, d\mathrm{vol} =0.
 \]
Since $k+L_Z\bar{g}$ weakly solves the linearized Ricci equation, by the formula \eqref{equation:adjoint-Ricci}  for $(D\Ric)^*$,
\begin{align*}
	0 &= \int_{\mathbb{R}^n} (k+L_Z\bar{g})\cdot (D\Ric)^* (\gamma)\, d\mathrm{vol} \\
	&=\int_{\mathbb{R}^n} (k+L_Z\bar{g})\cdot \left(-\tfrac{1}{2} \Delta \gamma + \bi^*(\Div \gamma) \right)\, d\mathrm{vol}\\
	&=\int_{\mathbb{R}^n} (k+L_Z\bar{g})\cdot \left(-\tfrac{1}{2} \Delta \gamma \right)\, d\mathrm{vol},
\end{align*}
where we use \eqref{equation:gauge} in the last equality. It follows that $\Delta (k+L_Z\bar{g})=0$ weakly in $\mathbb{R}^n$.

Since $k+L_{Z}\bar g$ decays at the rate $O(r^{-q})$, by elliptic regularity and Lemma~\ref{lemma:PDE},  $k+ L_{Z} \bar{g}$ is identically zero, and thus
\begin{alignat*}{2}
h&= -L_{Z+V} \bar{g} \quad &&\mbox{in } \mathbb{R}^n\setminus \Omega\\
L_Z\bar{g} &= 0 \quad &&\mbox{in } \Omega.
\end{alignat*}
Let $X = -(Z+V)$ and recall $V=0$ on $\Sigma$. We conclude that $h=L_X \bar{g}$, $X \in \C^{3,\alpha}_{1-q}(\mathbb{R}^n\setminus \Omega)$, and  $X$ is equal to a Killing vector on each connected component of $\Sigma$. 

\end{proof}

In view of Theorem~\ref{theorem:trivial}, we will define the following linear spaces of vector fields with \emph{zero} boundary value on $\Sigma$ that are asymptotic to Killing vector fields. 

\begin{definition}\label{definition:harmonic}
We define $\mathcal{X}$ to be the linear subspace of $\C^{3,\alpha}_{\mathrm{loc}} (\mathbb{R}^n\setminus \Omega)$ vector fields as 
\begin{align*}
	\mathcal{X} &= \Big\{\mbox{ $X=0$ on $\Sigma$, $X - K\in  \C^{3,\alpha}_{1-q} (\mathbb{R}^n\setminus \Omega)$ for some Killing vector $K$ of $\bar g$}\Big\}.
\end{align*}
Let  $\mathcal{X}_0$ be the subspace of $\mathcal{X}$ defined as
\[
	\mathcal{X}_0 =  \{ X\in \mathcal{X}: \Delta X=0 \mbox{ in } \mathbb{R}^n\setminus \Omega\}.
\]
\end{definition}

Denote the number 
\[
	N=\frac{n(n+1)}{2}.
\] 
It is the dimension of the space of Killing vectors on a connected open subset of $(\mathbb{R}^n,\bar{g})$, which is spanned by   $n$ translation vectors  $\{ \frac{\partial}{\partial x_1},\dots, \frac{\partial}{\partial x_n}\}$ and  $\tfrac{n(n-1)}{2}$ rotation vectors $\{ x_i \frac{\partial}{ \partial x_j} -x_j \frac{\partial}{ \partial x_i} \}_{i<j}$.
\begin{lemma} \label{lemma:X_0}
The linear space $\mathcal X_0$ is equal to
\[
	\mathcal{X}_0 =\left\{ X: \begin{array}{l}  \Delta X=0 \mbox{ in } \mathbb{R}^n\setminus \Omega, \, X = 0 \mbox{ on } \Sigma, \\
	X-K\in \C^{3,\alpha}_{2-n}(\mathbb{R}^n\setminus \Omega) \mbox{ for some Killing vector $K$ of $\bar{g}$}
	 \end{array}\right\}
\]
and  is isomorphic to the space of Killing vector fields on $(\mathbb R^n\setminus\Omega, \bar{g})$. Hence, 
\[	
	\Dim \mathcal{X}_0=N.
\]
\end{lemma}
\begin{proof}
Let $X\in\mathcal X_0$, and we write $X = K+Y$ where $K$ is a Killing vector and $Y= O^{3,\alpha}(|x|^{1-q})$. Since a Killing vector is harmonic, the assumption $\Delta X = 0$ implies  $\Delta Y=0$. By harmonic expansion,  there are constants $a_1,\dots, a_n$ such that $Y= a_i \tfrac{\partial}{\partial x_i}+ O^{3,\alpha}(|x|^{2-n})$. Replacing $K$ by $K+  a_i \tfrac{\partial}{\partial x_i}$, we have $X = K+Z$ where  $Z= O^{3,\alpha}(|x|^{2-n})$. Such decomposition of $X$ is unique. Define the homomorphism from $ \mathcal{X}_0$ to the space of Killing vectors  by sending $X$ to $K$. The homomorphism is clearly surjective, and it is injective because if $X=O^{3,\alpha}(|x|^{2-n})$ satisfies $\Delta X=0$ in $\mathbb{R}^n \setminus \Omega$ and $X=0$ on $\Sigma$, then $X$ is zero by Lemma~\ref{lemma:PDE}. 
\end{proof}

Our discussions in this section so far are valid for  boundary $\Sigma$ having multiple components. The following corollary is  a fundamental reason that  connectedness of $\Sigma$ is assumed in our main results. We have identified those Ricci flat deformations $h$ in Theorem~\ref{theorem:trivial} by $h=L_X \bar{g}$ for some vector field $X \in \C^{3,\alpha}_{1-q}(\mathbb{R}^n\setminus \Omega)$  equal to a Killing vector $K$ on each component of $\Sigma$. If $\Sigma$ is connected, then there is a \emph{global} Killing vector field in $\mathbb{R}^n\setminus \Omega$ equal to $K$ on $\Sigma$, which we still denote by $K$. By subtracting the global Killing vector, we obtain $X-K\in \mathcal{X}$ and $h=L_{X-K} \bar{g}$. It leads to the following corollary.
  \begin{corollary}\label{corollary:trivial}
Let $\Omega\subset \mathbb{R}^n$ be a bounded open subset whose boundary $\Sigma = \partial \overline{\Omega}$ is an embedded hypersurface. Suppose $\Sigma$ is connected. Let $h\in \C^{2,\alpha}_{-q}(\mathbb{R}^n \setminus \Omega)$. Then $h$ is a Ricci flat deformation at $\bar g$ satisfying the Cauchy boundary condition if and only if
$h = L_X \bar{g}$ in $\mathbb{R}^n\setminus \Omega$ for some $X\in\mathcal X$. In addition, such deformation satisfies the harmonic gauge  $\bi h=0$ in $\mathbb{R}^n\setminus \Omega$ if and only if $h = L_X \bar{g}$ in $\mathbb{R}^n\setminus \Omega$ for some $X\in \mathcal{X}_0$.
\end{corollary}

\section{Green-type identity for static vacuum deformations}\label{section:Green}


We will derive several fundamental identities for static vacuum deformations in this section. Define the functional $\mathscr{F}$ on an asymptotically flat pair  $(g, u)$ by
\begin{align*}
	\mathscr{F}(g, u) &= -2(n-1)\omega_{n-1} m_{\mathrm{ADM}}(g) +\int_{\mathbb{R}^n\setminus \Omega} u R_g \, d\mathrm{vol}_g.
\end{align*}
This functional is a special case of the Regge-Teitelboim Hamiltonian \cite{Regge-Teitelboim:1974}.
The first variation of $\mathscr{F}$ has been computed in \cite[Proposition 3.7]{Anderson-Khuri:2013} (Cf. \cite[Lemma 3.1]{Miao:2007}).  The second variation formula follows from differentiating the first variation. We state both formulas below, where $A_g, H_g$ are all computed with respect to the unit normal $\nu_g$ pointing away from $\Omega$ and recall $S(g, u) = -u\Ric_g + \nabla^2_g u - (\Delta_g u )g$.

\begin{lemma}\label{lemma:variations}
Let $(g(s), u(s))$ be a one-parameter family of asymptotically flat pairs. In the following formulas, we use $'$ to denote the $s$-derivative and write
\[
	(h(s), v(s)) = (g'(s), u'(s)).
\]
We may omit the argument $s$ when the context is clear. Then,
\begin{align*}
	\ds \mathscr{F}(g(s), u(s)) &= \int_{\mathbb{R}^n\setminus \Omega} \Big\langle \big(S(g, u) + \tfrac{1}{2} u R_g g, R_g\big), \big(h, v\big)\Big\rangle_g \, \dvol_g\\
	&\quad +\int_\Sigma \Big\langle \big(uA_g - \nu_g(u) g^\intercal, 2u\big), \big(h^\intercal, DH|_g(h)\big) \Big\rangle_g \, \da_g
\end{align*}
and
\begin{align*}
	&\2ds \mathscr{F}(g(s), u(s)) \\
	&=\int_{\mathbb{R}^n\setminus \Omega}\bigg[ \Big\langle P(h, v),(h,v) \Big\rangle_g + \Big\langle \big(S(g, u) + \tfrac{1}{2} u R_g g, R_g\big), (g'', u'')\Big\rangle_g\bigg]\, \dvol_g\\
	&\quad +\int_\Sigma  \bigg[\Big\langle  Q(h,v), \big(h^\intercal, DH|_g(h)\big)\Big\rangle_g+ \Big\langle \big(uA_g - \nu_g(u) g^\intercal, 2u\big), \big(g''^\intercal, H_g''\big) \Big\rangle_g \bigg] \, \da_g
\end{align*}
where  the right hand sides are evaluated at $(g(s), u(s))$, the angle bracket $\langle ,  \rangle_g$  denotes the sum of the $g(s)$-inner product  of the corresponding components (i.e., $\big\langle (S_1, S_2), (h,v)  \big\rangle_g = S_1\cdot h + S_2 v$),  and $P, Q$ are the linear differential operators defined by
\begin{align*}
	P(h, v) &= \Big(\big(S(g, u) + \tfrac{1}{2} u R_g g\big)', \, R_g'\Big) - \Big(2\big(S(g, u) + \tfrac{1}{2} u R_g g\big)\circ h, 0\Big) \\
	&\quad + \tfrac{1}{2} (\tr_g h)  \Big(S(g, u) + \tfrac{1}{2} u R_g g,\, R_g \Big) \quad \quad  \mbox{ in } \mathbb{R}^n\setminus \Omega 
	\\
	Q(h, v)&= \Big( \big(uA_g-\nu_g(u) g^\intercal\big)', \, 2v\Big)  - \Big(2\big(uA_g-\nu_g(u) g^\intercal\big)\circ h^\intercal, 0\Big)\\
	&\quad +\tfrac{1}{2} (\tr_g h^\intercal )\Big(uA_g - \nu_g(u) g^\intercal , \,  2u \Big) \quad \quad \mbox{ on } \Sigma.
\end{align*}
\end{lemma}

We can use the above variation formulas to identify the elements that are $\mathcal{L}^2$-orthogonal to the range of the  map $(g, u)\mapsto (-u\Ric_g + \nabla^2_g u, \Delta_g u)$ and of its linearization. Recall the linear subspace of vector fields, $\mathcal{X}$, defined in Definition~\ref{definition:harmonic}. At an asymptotically flat pair $(g, u)$ with $u>0$, for any $X\in \mathcal{X}$ we define
\begin{align*}
	\zeta_0 (X)&= \Big( L_X g - \big(\Div_g X+ u^{-1} X(u) \big)g, \, - \Div_g X + u^{-1} X(u)\Big) \\
	\kappa_0(X)&= \Big(L_X \bar{g} - (\Div X)\bar{g},  \, -\Div X\Big),
\end{align*}
where  $\kappa_0(X)$ is obtained by evaluating $\zeta_0(X)$ at $(g,u)=(\bar g, 1)$. Using the operator $\bi_g^*$ defined in \eqref{equation:operators} and writing $\bi^* = \bi_{\bar{g}}^*$, we may also re-express them as  
\begin{align}\label{equation:cokernel-elements}
\begin{split}
	\zeta_0 (X)&= \Big( 2\bi_g^* X - u^{-1} X(u) g, \,- \Div_g X + u^{-1} X(u)\Big) \\
	\kappa_0(X)&= \Big(2\bi^* X,  \, -\Div X\Big).
\end{split}
\end{align}

\begin{proposition}\label{proposition:cokernel}
Let $(g, u)$ be an asymptotically flat pair in $\mathbb{R}^n\setminus \Omega$ with $u>0$. For any deformation $(h, v) \in \C^{2}_{-q}(\mathbb{R}^n\setminus \Omega)$ and any  $X\in \mathcal{X}$, the following holds:
\begin{align}\label{equation:cokernel1}
	\int_{\mathbb{R}^n\setminus \Omega}  \bigg\langle \Big(-u \Ric_g + \nabla^2_g u, \Delta_g u\Big),\, \zeta_0(X) \bigg \rangle_g\, d\mathrm{vol}_g =0\\
\label{equation:cokernel2}
	\int_{\mathbb{R}^n\setminus \Omega} \bigg\langle \Big(-D\Ric(h) + \nabla^2 v, \, \Delta v\Big),\, \kappa_0 (X)\bigg\rangle_{\bar g} \, d\mathrm{vol}_{\bar g} =0.
\end{align}
\end{proposition}
\begin{proof}
For $s$ sufficiently small, let $\psi_s: \mathbb{R}^n\setminus\Omega \to \mathbb{R}^n\setminus \Omega$ be the flow of $X$ such that $\psi_0$ is the identity map. Let $(g_s, u_s) = (\psi_s^* g, \psi_s^* u)$ be the pull-back of $(g, u)$.  Because the ADM mass is invariant under such diffeomorphisms (see \cite[Theorem 4.2]{Bartnik:1986}) and so is the volume integral in $\mathscr{F}$, we have 
\[
	\mathscr{F}(g_s, u_s) = \mathscr{F}(g, u).
\]
Using that $\left. \ps\right|_{s=0} (g_s, u_s) = (L_X g, X(u))$, we apply the first variation formula in Lemma~\ref{lemma:variations} and get
\begin{align*}
	0&=\left. \ds \right|_{s=0}\mathscr{F}(g_s, u_s) \\
	&=\int_{\mathbb{R}^n\setminus \Omega} \Big\langle \big(-u\Ric_g + \nabla_g^2 u -(\Delta_g u)g + \tfrac{1}{2} u R_g g, R_g\big), \big(L_X g, X(u)\big)\Big\rangle_g \, \dvol_g
\end{align*}
where the boundary integral vanishes because $X=0$ on $\Sigma$. Rearranging the integrands gives \eqref{equation:cokernel1}. Equation \eqref{equation:cokernel2} follows immediately by taking linearization of equation \eqref{equation:cokernel1} at $(\bar g,1)$.
\end{proof}

The above proof explores the property that the functional $\mathscr{F}$ is invariant among diffeomorphisms  that fix both the boundary $\Sigma$ and the ``structure of infinity'' of $\mathbb{R}^n\setminus \Omega$, i.e. those diffeomorphisms generated by $X\in \mathcal{X}$.  The first variation formula of $\mathscr{F}$ implies that  $(L_X g, X(u))$ is $\mathcal{L}^2$-orthogonal to $(S(g, u) + \tfrac{1}{2} u R_g g, R_g)$. Since we will instead work on the operator $(g, u)\mapsto (-u\Ric_g + \nabla^2_g u, \Delta_g u)$ in Section~\ref{section:existence}, after rearranging from $(S(g, u) + \tfrac{1}{2} u R_g g, R_g)$ to $(-u\Ric_g + \nabla^2_g u, \Delta_g u)$, the ``geometric'' deformation  $(L_X g, X(u))$  leads to $\zeta_0(X)$ which takes a less geometric form.


As is shown in \cite[Proposition 3.5]{An:2020}, the next proposition says that the differential operator $(h, v)\to \Big(P(h, v),\big (h^\intercal, DH|_g(h)\big)\Big)$ is ``self-adjoint'' in the following sense. 
\begin{proposition}[Green-type identity]\label{proposition:Green}
Let $(g, u)$ be an asymptotically flat pair in $\mathbb{R}^n\setminus \Omega$. For any $(h, v), (k,w)\in \C^{2}_{-q}(\mathbb{R}^n\setminus \Omega)$, we have 
\begin{align}\label{equation:Green}
\begin{split}
	&\int_{\mathbb{R}^n\setminus \Omega}\Big\langle P(h, v),(k,w)\Big\rangle_g \, \dvol_g-\int_{\mathbb{R}^n\setminus \Omega}\Big\langle  P(k, w), (h, v)\Big\rangle_g \, \dvol_g\\
	&=- \int_\Sigma  \Big\langle  Q(h,v) , \big(k^\intercal, DH|_g(k)\big)\Big\rangle_g \, \da_g+\int_\Sigma  \Big\langle Q(k, w), \big(h^\intercal, DH|_g(h)\big)\Big \rangle_g \, \da_g.
\end{split}
\end{align}
\end{proposition}
\begin{proof}
Let $\big(g(s,t), u(s,t)\big)$ be a two-parameter family of asymptotically flat pairs defined by
\begin{align*}
	g(s,t) = g+sh+tk\quad \mbox{ and } \quad u(s,t)=u+sv+tw. 
\end{align*}
Then we have
\begin{align*}
	(h, v )= \ps (g(s,t), u(s,t))\quad(k,w)= \pt (g(s,t), u(s,t) ),\quad\mbox{ at }t=s=0.
\end{align*}
Computing the two-derivatives of the functional $\mathscr{F}$ at $(g(t,s), u(t,s))$ and equating
\[
	\left. \pst \right|_{t=s=0}\mathscr{F}(g(t, s), u(t, s))  = \left.\pts\right|_{t=s=0}\mathscr{F}(g(t, s), u(t, s))
\]
gives \eqref{equation:Green}. (Note that the integrands involving the two derivatives such as $\pst (g, u)$, $\pst (g^\intercal, H_g)$  all cancel out by symmetry $\pst = \pts$.) 
\end{proof}

 We shall call the identity \eqref{equation:Green} as a \emph{Green-type identity} (for the operator $P$). To explain this, we briefly recall the classical Green identity for the Laplace operator. Let $\nu$ be the unit normal to $\Sigma$ pointing away from $\Omega$.  For any functions $v, w\in \C^{2}_{-q}(\mathbb{R}^n \setminus \Omega)$, the Green identity says
\begin{align*}
	\int_{\mathbb{R}^n \setminus \Omega}(w\Delta v - v\Delta w) \, \dvol = \int_\Sigma \left( -w \nu(v) + v \nu(w) \right)\, \da.
\end{align*}
 Our Green-type identity~\eqref{equation:Green} resembles  the classical one by replacing the operator $\Delta$ with the operator $P$, the Dirichlet boundary data with the Bartnik boundary data, and the Neumann boundary data with the operator~$Q$. 
 
When the variations are evaluated at a static vacuum pair $(g, u)$, the operator $P(h,v)$ becomes
\begin{align*}
	P(h, v)&= \bigg(\Big(S(g, u) + \tfrac{1}{2} u R_g g\Big)' , R_g' \bigg) \\
	&=  \bigg(DS|_{(g, u)}(h,v)  , DR|_g(h)\bigg) + \bigg( \tfrac{1}{2} u DR|_g(h) g, 0\bigg).
\end{align*}

Note that the first parenthesis in the right hand side corresponds to the linearized static vacuum system~\eqref{lsv}. Consequently,  $(h, v)$ is a static vacuum deformation at $(g, u)$  if and only if $P(h, v)=0$.

We will further focus on the special case when the linearization is taken at $(\bar{g},1)$. We often omit the subscript $g$ when the geometric quantities (such as covariant derivatives, trace, linearization, inner products, volume/area forms, second fundamental form, etc)  are computed with respect to $g=\bar{g}$. With respect to $(\bar{g}, 1)$, the operators $P, Q$ take the simpler forms:
\vspace{-8pt}
\begin{align}\label{equation:flat}
\begin{split}
	P(h, v) &=  \Big(-D\Ric(h) + \nabla^2 v + \big(- \Delta v + \tfrac{1}{2} DR(h) \big)\bar{g}, \, DR(h) \bigg)\\
	Q(h, v)&= \Big( vA + DA(h) - \nu(v) \bar{g}^\intercal -2A\circ h^\intercal+ \tfrac{1}{2} \tr h^\intercal A, \, 2v + \tr h^\intercal \Big).
\end{split}
\end{align}

We  re-state Lemma~\ref{lemma:variations} in the following special case that will be used in Proposition~\ref{proposition:variation2} later.
\begin{corollary}\label{corollary:variations}
Let $(g(s), u(s))$ be a one-parameter family of asymptotically flat pairs. Assume that $g(0) = \bar{g}$ and $u(s)\equiv 1$. Denote by $h= g'(0)$. Then  
\begin{align*}
	\left. \ds\right|_{s=0} \mathscr{F}(g(s), 1) &=\int_\Sigma  \Big(A\cdot h^\intercal +  2DH(h) \Big)\, \da\\
	\left. \2ds\right|_{s=0} \mathscr{F}(g(s), 1) &=\int_{\mathbb{R}^n\setminus \Omega} h\cdot \Big(-D\Ric(h)+\tfrac{1}{2} DR(h) \bar{g}\Big) \, \dvol\\
	+ \int_\Sigma \Big( \big(DA(h)  - &2A\circ h^\intercal+ \tfrac{1}{2} \tr h^\intercal A\big)\cdot h^\intercal+ \tr h^\intercal DH(h)+  A\cdot\big(g''\big)^\intercal+2 H_g'' \Big) \, \da.
\end{align*}
\end{corollary}

The next corollary is a direct consequence of Proposition~\ref{proposition:Green}.  In particular, Equations \eqref{equation:Green-boundary} and \eqref{equation:Ricci-flat} below give effective formulas to compute $DA(h)$, which is essential to our main arguments in Theorems \ref{theorem:static-convex} and \ref{theorem:static-generic}. 

\begin{corollary}
For any $(h, v), (k,w)\in \C^2_{-q}(\mathbb{R}^n\setminus \Omega)$, we have
\begin{align*}
	&\int_{\mathbb{R}^n\setminus \Omega}\Big\langle P(h,v),(k,w)\Big\rangle \, \dvol-\int_{\mathbb{R}^n\setminus \Omega}\Big\langle  P(k,w), (h, v)\Big\rangle \, \dvol\\
	&=- \int_\Sigma  \Big\langle Q(h,v), \big(k^\intercal, DH(k)\big) \Big\rangle \, \da+\int_\Sigma  \Big\langle Q(k, w), \big(h^\intercal, DH(h)\big)\Big \rangle \, \da
\end{align*}
where $P, Q$ take the forms \eqref{equation:flat}. We also have the following special cases:
\begin{enumerate}
\item If both $(h, v)$ and $(k, w)$ are static vacuum deformations at $(\bar g,1)$ and $h$ has zero Bartnik data $h^\intercal=0, DH(h)=0$ on $\Sigma$, then 
\begin{align} \label{equation:Green-boundary}
	\int_\Sigma  \Big\langle \big(vA + DA(h) -\nu(v) \bar{g}^\intercal, 2v \big), \big(k^\intercal, DH(k)\big) \Big\rangle \, \da=0.
\end{align}
\item If $h, k$ are both Ricci flat deformations at $ \bar g$ and $h^\intercal=0$ on $\Sigma$, then  
\begin{align}\label{equation:Ricci-flat}
	\int_\Sigma\Big( DA(h)\cdot k^\intercal -(\tr k^\intercal) DH(h) \Big)\, d\sigma =0.
\end{align} 
\end{enumerate}
\end{corollary}

\section{Sufficient condition for existence and local uniqueness}\label{section:existence}

Let $\Omega$ be a bounded open subset in $(\mathbb{R}^n, \bar{g})$ whose boundary $\Sigma= \partial \overline{\Omega}$ is an embedded hypersurface. We assume that $\Sigma$ is connected.   Recall the differential operators $\bi_g, \D_g, \bi^*_g$ defined in \eqref{equation:operators}, and note that the subscript $g$ is omitted when $g$ is the Euclidean metric $\bar{g}$. We also recall  the spaces of vector fields $\mathcal{X}, \mathcal{X}_0$ in Definition~\ref{definition:harmonic}, and the integer $N = \frac{n(n+1)}{2}$ denotes the dimension of $\mathcal{X}_0$.

We prove Theorem~\ref{theorem:criterion-main} in this section. Use  the terminologies introduced in Section~\ref{section:basic}, we re-state the definition of a static regular boundary:
\begin{manualdefinition}{\ref{definition:static-regular}}
The boundary $\Sigma$ is said to be \emph{static regular in  $\mathbb R^n\setminus\Omega$}  if for any static vacuum deformation   $(h, v)\in \C^{2,\alpha}_{-q}(\mathbb{R}^n\setminus \Omega)$ (at $(\bar{g}, 1)$) such that $h$ satisfies the Bartnik boundary condition on $\Sigma$, $h$ must satisfy the Cauchy boundary condition on $\Sigma$. 
\end{manualdefinition}

The following theorem is a more precise version of Theorem~\ref{theorem:criterion-main}.

\begin{theorem}\label{theorem:criterion}
Suppose the boundary $\Sigma$ is static regular in $\mathbb R^n\setminus\Omega$. Then there exist positive constants $\epsilon_0, C>0$ such that for each $0<\epsilon<\epsilon_0$, if $(\tau, \phi)$ satisfies $\| (\tau,\phi)- (\bar{g}^\intercal, H) \|_{\C^{2,\alpha}(\Sigma)\times \C^{1,\alpha}(\Sigma)}<\epsilon$, there exists a unique  $(g, u)\in \mathcal{M}$ with $\|(g, u) - (\bar{g}, 1) \|_{\C^{2,\alpha}_{-q}(\mathbb{R}^n \setminus \Omega)} < C\epsilon$ such that 
\begin{align}\label{equation:vacuum-static-gauge}
\begin{split}
	\begin{array}{ll}
	\left\{ \begin{array}{l}-u\Ric_g+ \nabla^2_g u=0\\
	\Delta_g u=0\\
	 \bi g+du=0	\end{array} \right.  &\mbox{ in } \mathbb{R}^n\setminus \Omega \\
 	\left\{ \begin{array}{l} 
	g^\intercal = \tau\\
	H_g = \phi
	\end{array} \right. &\mbox{ on } \Sigma
\end{array}
\end{split}
\end{align}
and 
\begin{align}\label{equation:orthogonal-gauge}
	\int_{\mathbb{R}^n\setminus \Omega}  \big((g-\bar{g})\cdot L_X \bar{g} \big) \rho\, d\mathrm{vol} =0 \quad \mbox{ for all } X\in \mathcal{X}_0
\end{align}
where $\rho(x)=(1+|x|^2)^{-1}$. 
\end{theorem}


We will formulate  Theorem~\ref{theorem:criterion} into an equivalent statement, Theorem~\ref{theorem:T}, in Section~\ref{section:kernel}   and give its proof  at the end of Section~\ref{section:proof}. 
Before we proceed with the proofs to the main result, we make general remarks on the gauge conditions that appear in Theorem~\ref{theorem:criterion}.  
\begin{definition}
Let $(g, u)\in \mathcal{M}$.
\begin{enumerate}
\item $(g, u)$ is said to satisfy the \emph{static-harmonic gauge}  (at $(\bar{g},1)$) in $\mathbb{R}^n\setminus \Omega$ if $\bi g + du=0$ in $\mathbb{R}^n\setminus \Omega$. 
\item $(g, u)$ is said to satisfy the \emph{orthogonal gauge} (at $(\bar{g},1)$) if \eqref{equation:orthogonal-gauge} holds. 
\end{enumerate}
\end{definition}

\begin{remark}
 We remark the connection between the static-harmonic gauge and harmonic gauge. 
Obviously if $u$ is constant, then those two are the same, but there is a deeper connection. If $u>0$, the warped product metric $\mathbf{g}=\pm u^2dt^2+g$ (either plus or minus sign in  the first term) satisfies the classical harmonic gauge at the flat metric if and only if $\bi g+ udu=0$. Although this equation is not exactly the static-harmonic gauge equation in the above definition, it yields the same linearized gauge equation for deformations at $(\bar{g}, 1)$.  All of the following analysis would work equally well if we would have defined the static-harmonic gauge as $\bi g + udu=0$ instead. 

 The choice of the positive weight function $\rho$ in \eqref{equation:orthogonal-gauge} is quite arbitrary. In fact, for $n>3$, we can simply let $\rho(x)\equiv 1$ instead. When $n=3$, we need a positive weight function that decays at least  faster than the order of  $|x|^{-\frac{1}{2}}$ to ensure the integral is finite.

\end{remark}

It is well-known that the harmonic gauge is effective for studying the Ricci curvature equation on a closed manifold or a bounded manifold with fixed boundary conditions. But for asymptotically flat manifolds, we also need to take account of diffeomorphisms that may not be the identity map at infinity but an isometry (up to negligible terms).  More precisely, let $\mathscr{D} $ denote the space of all $\C^{3,\alpha}_{\mathrm{loc}}$ diffeomorphisms $\psi: \mathbb{R}^n\setminus\Omega \to  \mathbb{R}^n\setminus\Omega$ satisfying $\psi|_\Sigma = \mathrm{Id}_{\Sigma}$, $\psi(x) - Ox\in \C^{3,\alpha}_{1-q}(\mathbb{R}^n\setminus \Omega)$ for some special orthogonal matrix $O$. Consider the canonical action $\mathscr{D}$ on the space of asymptotically flat pairs $(g, u)\in \mathcal{M}$ by $\psi\cdot (g, u) =( \psi^* g, \psi^*u)$. We will see in the next lemma that the static-harmonic gauge is not sufficient to find a unique representative in the orbit $\mathscr{D}\cdot (g, u)$,  but there is a unique one in a neighborhood of $(\bar{g}, 1)$ if we additionally impose the orthogonal gauge.

\begin{lemma}\label{lemma:both-gauge}

There is an open neighborhood $\mathcal{U}$ of $(\bar{g}, 1)$ in $\mathcal{M}$  and an open neighborhood $\mathscr{D}_0$ of $\mathrm{Id}_M$ in $\mathscr{D}$ such that for any $(g, u)\in \mathcal{U}$, there exists a unique diffeomorphism $\psi\in \mathscr{D}_0$ such that  $(\psi^*g, \psi^*u)$ satisfies both static-harmonic and orthogonal gauges.
\end{lemma}
\begin{proof}
We denote by $\langle h, k\rangle_{\mathcal{L}^2_\rho} = \int_{\mathbb{R}^n\setminus \Omega} (h\cdot k )\rho \,d\mathrm{vol}$, where the dot product is with respect to $\bar{g}$ and  we recall $\rho(x) = (1+|x|^2)^{-1}$. Since $\mathcal{X}_0$ is $N$-dimensional (see Lemma~\ref{lemma:X_0}),  we let $X^{(1)}, \dots, X^{(N)}$ be a basis of $\mathcal{X}_0$ that is orthonormal in the following sense: 
\[
	\big\langle L_{X^{(i)}}\bar{g} , L_{X^{(j)}}\bar{g} \big\rangle_{\mathcal{L}^2_\rho} =\delta_{ij}.
\] 
Consider the map $G: \mathscr{D}\times \mathcal{M}\to \C^{1,\alpha}_{-q-1}\times \mathbb{R}^N$ defined by
\[
	G(\psi, (g, u))= \big(\bi (\psi^* g ) + d(\psi^* u), (b_1, \dots, b_N) \big)
\]
where the numbers $b_i$ are given by
\[
b_i = \big\langle\psi^* g -\bar{g},  L_{X^{(i)}}\bar{g}\big\rangle_{\mathcal{L}^2_\rho}. 
\]
 Linearizing $\mathscr{D}$ in the first argument at $( \mathrm{Id}_{\mathbb R^n\setminus\Omega}, (\bar{g}, 1))$ gives $D_1 G: \mathcal{X} \to \C^{1,\alpha}_{-q-1}(\mathbb{R}^n\setminus \Omega) \times \mathbb{R}^N$:
\[
 	D_1 G(X) =  \big(-\Delta X, (c_1, \dots, c_N) \big)
\]	
where  $c_i = \langle L_X\bar{g},  L_{X^{(i)}}\bar{g}\rangle_{\mathcal{L}^2_\rho} $. 

We show that $D_1 G$ is an isomorphism: If $D_1G(X)=0$, then $\Delta X=0$ and $c_i=0$ for all $i$. The first equation implies  $X\in\mathcal X_0$, and hence, the equations $c_i=0$ imply  $X=0$. Thus $D_1G$ is injective. To see that it is surjective, for any $Z\in\C^{1,\alpha}_{-q-1}$ and $(c_1,\dots, c_N)\in  \mathbb R^N$, there is $Y\in \mathcal{X}$ such that $-\Delta Y=Z$  by Lemma~\ref{lemma:PDE}.  Let $d_i=c_i -\langle L_Y\bar{g},  L_{X^{(i)}}\bar{g}\rangle_{\mathcal{L}^2_\rho} $, then $X=Y+ d_1X^{(1)}+ \dots+ d_N X^{(N)}$ satisfies $D_1G(X)=(Z, (c_1,\dots, c_N))$. Then the lemma follows from Implicit Function Theorem.
\end{proof}

\subsection{Boundary value problem under static-harmonic gauge}\label{section:kernel}

Define the map $T:\mathcal M\to \C^{0,\alpha}_{-q-2}(\mathbb R^n\setminus\Omega)\times\C^{1,\alpha}(\Sigma)\times \big((\bar{g}^\intercal, H_{\bar{g}}) +\mathcal{B}(\Sigma)\big)$ as 
\begin{align}\label{equation:nonlinear}
T (g, u)=
\begin{array}{l}
	\left\{ \begin{array}{l}-u\Ric_g+ \nabla^2_g u - \D_g (\bi g+du)\\
	\Delta_g u	\end{array} \right. \quad \mbox{ in } \mathbb{R}^n\setminus \Omega \\
 	\left\{ \begin{array}{l} 
	\bi g+du\\
	g^\intercal\\
	H_g
	\end{array} \right. \quad \mbox{ on } \Sigma.
\end{array}
\end{align}
We note that the operator $T$ is essentially equivalent to the harmonic-gauged Ricci operator of the warped product space, introduced by Anderson and Khuri \cite[Section 3]{Anderson-Khuri:2013}.  Let us explain the codomain of $T$. We denote by $\C^{0,\alpha}_{-q-2}(\mathbb{R}^n \setminus \Omega)$ the codomain of the first two equations and by $\C^{1,\alpha}(\Sigma)\times  \big((\bar{g}^\intercal, H_{\bar{g}}) +\mathcal{B}(\Sigma)\big)$ the codomain of the boundary equations. More specifically, $\C^{1,\alpha}(\Sigma)$ consists of $\C^{1,\alpha}$ covectors of $\mathbb{R}^n$ defined along $\Sigma$ and $\mathcal{B}(\Sigma)$ consists of pairs $(\tau, \phi)$ where $\tau\in \C^{2,\alpha}(\Sigma)$ is a symmetric $(0,2)$-tensor on the tangent bundle of $\Sigma$  and $\phi \in \C^{1,\alpha}(\Sigma)$ is a scalar-valued function on $\Sigma$. 

The reader should compare $T$ with~\eqref{equation:vacuum-static-gauge}. While Theorem~\ref{theorem:criterion} concerns with solving both \eqref{equation:vacuum-static-gauge} and \eqref{equation:orthogonal-gauge}, in the rest of Section~\ref{section:existence}, we will focus on solving $T$ and an accompanying operator $\overline{T}$ modified from $T$. It is because that, as shown in the next lemma, if we  put aside the orthogonal gauge condition  \eqref{equation:orthogonal-gauge},   finding $(g, u)$ solving \eqref{equation:vacuum-static-gauge} is ``equivalent'' to finding some $(g, u)$ solving $T(g,u)=(0, 0, 0, \tau, \phi)$ (namely, the first three equations in \eqref{equation:nonlinear} are zero), provided that $(g, u)$ is sufficiently close to $(\bar{g}, 1)$. 

\begin{lemma}[Cf. {\cite[Proposition 2.1]{Anderson-Khuri:2013}}]\label{lemma:gauge}
There is an open neighborhood $\mathcal{U}$ of $(\bar{g}, 1)$ in $\mathcal{M}$ such that for $(g, u)\in \mathcal{U}$,  if $(g, u)$ satisfies $T(g, u) = (0, 0, 0, \tau, \phi)$, then $(g, u)$ solves \eqref{equation:vacuum-static-gauge}.

Conversely, if $(g, u)$ solves  \eqref{equation:vacuum-static-gauge}, then there exists a diffeomorphism $\psi\in \mathscr{D}_0$  such that $(\psi^* g, \psi^*u)$ satisfies $T(\psi^* g, \psi^*u) = (0, 0, 0, \tau, \phi)$. 
\end{lemma}

\begin{proof}
Let $(g, u)\in \mathcal{U}$ solve $T(g, u)=(0,0,0,\tau,\phi)$. We will  prove $\bi g+du=0$ in $\mathbb R^n\setminus\Omega$, provided $\mathcal{U}$ sufficiently small. 
Taking the trace of the first tensor equation in $T(g, u)=(0,0,0,\tau,\phi)$ and using the harmonic equation of $u$, we get 
\[ 
	R_g= -u^{-1} \tr_g \D_g (\bi g+du) = -u^{-1} \Div_g (\bi g+ du).
\]
We show that $\bi g+du$ weakly solves the following equation:
\begin{align}
	\bi_g \D_g (\bi g+du) &= -\tfrac{1}{2} R_g du\label{equation:1}\\
	&=\tfrac{1}{2} u^{-1} du \Div_g (\bi g+ du).\notag
\end{align}
The second identity follows by substituting $R_g$ by the earlier computation. Since by assumption $\mathcal D_g(\beta g+du)=-u{\rm Ric}_g+\nabla^2_gu$, to prove \eqref{equation:1} it suffices to show for all vector fields $X\in \C^\infty_c(\mathbb{R}^n\setminus \Omega)$
\begin{align*}
	\int_{\mathbb{R}^n\setminus \Omega} (-u\Ric_g + \nabla_g^2 u) \cdot \bi_g^*X \, d\mathrm{vol}_g = \int_{\mathbb{R}^n\setminus \Omega}- \tfrac{1}{2} R_g  X(u) \, d\mathrm{vol}_g.
\end{align*} 
This identity holds due to \eqref{equation:cokernel1} and that $u$ is harmonic. (Note that if $(g, u)$ were $\C^3_{\mathrm{loc}}$, we can simply apply $\bi_g$ on the first equation of $T(g, u)=(0,0,0,\tau,\phi)$ to obtain \eqref{equation:1}.)

Recall that $\bi_g \D_gV = -\tfrac{1}{2} \Delta_g V -\tfrac{1}{2} \Ric_g(V, \cdot)$ for any vector field $V$. If we denote by $V= \bi g+du \in \C^{1,\alpha}_{-q-1}(\mathbb{R}^n\setminus \Omega)$, then the previous equation with the boundary condition $ \bi g+du=0$ on $\Sigma$ implies that $V$ weakly solves the boundary value problem
\begin{align*}
	\Delta_g V + u^{-1} du \Div_g V + \Ric_g(V, \cdot)&=0\quad \mbox{ in } \mathbb{R}^n\setminus \Omega \\
	V&=0\quad \mbox{ on } \Sigma.
\end{align*}
By elliptic regularity,  $V\in \C^{2,\alpha}_{-q-1}(\mathbb{R}^n\setminus \Omega)$. If  $\mathcal{U}$ is sufficiently small, for $(g, u)\in \mathcal{U}$ the operator on $V$ in the previous first equation is sufficiently close to $\Delta$ in the operator norm, and thus by Lemma~\ref{lemma:PDE} the above boundary value problem has only the trivial solution $V=0$. 


For the converse statement, we can simply take the diffeomorphism $\psi$ obtained from Lemma~\ref{lemma:both-gauge}.
\end{proof}

According to the above discussions, we can reformulate Theorem~\ref{theorem:criterion} into the following theorem for the operator $T$.

\begin{manualtheorem}{\ref{theorem:criterion}$^\prime$}\label{theorem:T}
Suppose the boundary $\Sigma$ is static regular in $\mathbb R^n\setminus\Omega$. Then there exist positive constants $\epsilon_0, C>0$ such that for each $0<\epsilon<\epsilon_0$, if $(\tau, \phi)$ satisfies $\| (\tau,\phi)- (\bar{g}^\intercal, H) \|_{\C^{2,\alpha}(\Sigma)\times \C^{1,\alpha}(\Sigma)}<\epsilon$, there exists a unique $(g, u)\in \mathcal{M}$ with $\|(g, u) - (\bar{g}, 1) \|_{\C^{2,\alpha}_{-q}(\mathbb{R}^n \setminus \Omega)} < C\epsilon$ such that $T(g, u) = (0, 0, 0, \tau, \phi)$ and $(g, u)$ satisfies the orthogonal gauge \eqref{equation:orthogonal-gauge}.
\end{manualtheorem}

One may wish for a stronger statement that the linearization of $T$ at $(\bar{g},1)$ is surjective, and thus $T(g, u)$ is solvable for \emph{all} values sufficiently close to $T(\bar{g},1)$  by Local Surjectivity Theorem, not only the values that take the particular form $(0, 0, 0, \tau, \phi)$. Unfortunately, as we will see in the next lemmas, the linearization is never surjective.

Denote the  linearization  of $T$ at $(\bar{g}, 1)$ by $L:\C^{2,\alpha}_{-q}(\mathbb{R}^n\setminus \Omega) \to\C^{0,\alpha}_{-q-2}(\mathbb{R}^n\setminus \Omega)\times\C^{1,\alpha}(\Sigma)\times\mathcal{B}(\Sigma)$. For $(h, v) \in \C^{2,\alpha}_{-q}(\mathbb{R}^n \setminus \Omega)$,
\begin{align}\label{equation:linear}
L(h, v)=
\begin{array}{l}
	\left\{ \begin{array}{l}-D\Ric(h)+ \nabla^2 v - \D(\bi h+ dv) \\
	\Delta v	\end{array} \right. \quad \mbox{ in } \mathbb{R}^n\setminus \Omega \\
 	\left\{ \begin{array}{l} \bi h + dv\\
	h^\intercal\\
	DH(h)
	\end{array} \right. \quad \mbox{ on } \Sigma.
\end{array}
\end{align}

By \eqref{equation:Ricci-gauge}, the first equation in \eqref{equation:linear} is reduced to the Laplace equation:
\begin{align*}
-D\Ric(h)+ \nabla^2 v- \D(\bi h+ dv) = \tfrac{1}{2} \Delta h.
\end{align*}
Based on the observation of Anderson and Khuri for the corresponding harmonic-gauged Ricci operator in the warped product space~\cite[Proposition 3.1]{Anderson-Khuri:2013}, the map $L$ is Fredholm of index zero. 

\begin{lemma}[Cf.  {\cite[Proposition 3.1]{Anderson-Khuri:2013}}]\label{lemma:Fredholm}
The linearized operator $L: \C^{2,\alpha}_{-q}(\mathbb{R}^n\setminus \Omega) \to \C^{0,\alpha}_{-q-2}(\mathbb{R}^n \setminus \Omega)\times \C^{1,\alpha}(\Sigma)\times \mathcal{B}(\Sigma)$ is Fredholm of index zero.
\end{lemma}
\begin{proof}
The proof is essentially the same as that of~\cite[Proposition 3.1]{Anderson-Khuri:2013} for the corresponding operator.  So we only give an outline how it adapts in our setting.  By examining the principle symbol, the boundary condition of $L$ is elliptic (i.e., it satisfies the  Lopatinski-Shapiro condition), so the operator $L$ is Fredholm, see  \cite[Theorem 20.1.2]{Hormander:2007} and  \cite[Theorem 1.3]{Lockhart-McOwen:1985}.  To show that the index of $L$ is zero, we argue as in \cite{Anderson-Khuri:2013} that $L$ is homotopic to a Fredholm operator of index zero through a path of Fredholm operators. Since a path is not explicitly written down in \cite{Anderson-Khuri:2013}, we include it for completeness. Let $s\in [0,1]$ and define the family of operators $L^{(s)}: \C^{2,\alpha}_{-q}(\mathbb{R}^n\setminus \Omega) \to \C^{0,\alpha}_{-q-2}(\mathbb{R}^n \setminus \Omega)\times \C^{1,\alpha}(\Sigma)\times \mathcal{B}(\Sigma)$ by varying only the boundary equations:
\begin{align*}
\begin{array}{l}
	L^{(s)}(h,v)=
	\begin{array}{l}\left\{ \begin{array}{l}\tfrac{1}{2} \Delta h \\
	\Delta v	\end{array} \right. \quad \mbox{ in } \mathbb{R}^n\setminus \Omega \\
 	\left\{ \begin{array}{l} (1-s)(\bi h+dv)(\nu)+s\nu(v)\\
	(1-s)(\bi h+dv)^\intercal-s\nabla_\nu\omega\\
	h^\intercal\\
	(1-s)DH(h)+s(\nabla_\nu h)(\nu,\nu)
	\end{array} \right. \quad \mbox{ on } \Sigma
\end{array}
\end{array}
\end{align*}
where $\omega(e_a) := h(\nu, e_a)$ is a one-form defined on the tangent bundle of $\Sigma$. By checking the principal symbol, the boundary condition in $L^{(s)}$ is elliptic and thus $L^{(s)}$ is Fredholm for each $s\in[0,1]$. Note that $L^{(0)}=L$ and $L^{(1)}$ is the operator with a standard Dirichlet/Neumann boundary condition. It is straightforward to verify that $L^{(1)}$ is an isomorphism, and thus of index $0$. 
\end{proof}

The kernel space of the linearized operator $L$, however,   is never trivial and hence $L$ is not surjective.  We end this section with the fact that $\Ker L$ always contains an $N$-dimensional subspace, generated by $\mathcal{X}_0$. The lemma is a direct consequence of Corollary~\ref{corollary:trivial}. A fundamental reason that $\Ker L$ is not trivial is that, as is shown in Lemma~\ref{lemma:both-gauge}, the static-harmonic gauge by itself cannot guarantee uniqueness of solutions to the operator $T$.

\begin{lemma}The kernel space of $L$ contains an $N$-dimensional subspace:
\[
	\Ker L \supseteq \big\{ (h, 0): h=L_X \bar{g}, \mbox{ for some } X \in \mathcal{X}_0\big\}.
\] 
\end{lemma}

\subsection{Kernel and range for a static regular boundary}
From now on, we assume the boundary $\Sigma$ is static regular in $\mathbb R^n\setminus\Omega$. The next lemma says that  the kernel elements of $L$ are exactly those arising from $X\in \mathcal{X}_0$. 

\begin{lemma}\label{lemma:kernel}
If the boundary $\Sigma$ is static regular in  $\mathbb R^n\setminus\Omega$, then 
\[
	\Ker L=\big\{ (h, 0): h=L_X \bar{g}, \mbox{ for some } X \in \mathcal{X}_0\big\}.
\] 
As a direct consequence, $\Ker L$ is isomorphic to $\mathcal{X}_0$ and hence 
\[
\Dim \Ker L=N.
\] 
\end{lemma}
\begin{proof}
Let $(h, v) \in \Ker L$. By elliptic regularity, $(h, v)\in \C^\infty_{-q}(\mathbb{R}^n\setminus \Omega)$. By applying $\bi$ on the first equation of \eqref{equation:linear} and using $\bi (-D\Ric (h) + \nabla^2 v)=0$, we have 
\[
	\Delta (\bi h+ dv) = 0 \quad \mbox{ in } \mathbb{R}^n\setminus \Omega.
\]
 Then by the  boundary condition $\bi h+ dv=0$ on $\Sigma$ and the fall-off rate of $\bi h+ dv$, we see that $\bi h+ dv$ is identically zero in $\mathbb{R}^n\setminus \Omega$ and thus $(h,v)$ is a static vacuum deformation. By the assumption that $\Sigma$ is static regular, we have $DA(h)=0$ on $\Sigma$.

 The lemma would follow by showing that $v$ vanishes identically in $\mathbb{R}^n\setminus \Omega$. Once we show that $v$ is zero, we have $(h,0)\in{\rm Ker}L$, i.e. $h$ is a Ricci flat deformation and satisfies $\bi h=0$, and thus $h=L_X\bar{g}$ for some $X\in \mathcal{X}_0$  by Corollary~\ref{corollary:trivial}.
 
 Recall that $v$ is a harmonic function in $\mathbb{R}^n\setminus \Omega$. Although $v$ is not  assumed explicitly to satisfy a boundary condition, we observe that  other equations of $(h, v)$ imply that $v$ satisfies the following (hidden) boundary conditions on $\Sigma$:
\begin{align}
	\Delta_\Sigma v + \nu(v) H&=0  \label{equation:boundary}\\
	A ( \nabla^\Sigma v, \cdot) - d(\nu(v))&=0 \label{equation:boundary-extra}
\end{align}
where $\nabla^\Sigma, d$ are the covariant derivative and exterior derivative on $\Sigma$, respectively. Let $X$ be a compactly supported vector field in $\mathbb{R}^n$. 
Applying \eqref{equation:Green-boundary} to the static vacuum deformation pairs $(h,v)$ and $(k, w)=(L_X \bar{g}, 0)$ and noting $DA(h)=0$, we get  
\[
	\int_\Sigma \Big\langle \big(vA-\nu(v) \bar{g}^\intercal , 2v \big), \big((L_X \bar{g})^\intercal, DH(L_X \bar{g}) \big) \Big \rangle \, d\sigma=0.
\]
If we let $X=\eta \nu$ on $\Sigma$ for a smooth function $\eta$, then $L_X \bar{g} = 2\eta A$ and $DH(L_X \bar{g}) = -\Delta_\Sigma \eta - \eta |A|^2$ where $\Delta_\Sigma$ is the Laplace operator of $(\Sigma, \bar{g}^\intercal)$. Therefore, the previous identity implies that 
\[
	\int_\Sigma \eta \left(\Delta_\Sigma v + \nu(v) H\right)\, d\sigma=0.
\]
  Since $\eta$ is arbitrary, it implies \eqref{equation:boundary}.  (See \cite[Section 2]{Anderson:2015} for a different proof of  \eqref{equation:boundary}.) If we let $X$ be tangential to $\Sigma$ and apply integration by parts,  we derive
\begin{align*}
	0&= \int_\Sigma \Big\langle \big(vA-\nu(v) \bar{g}^\intercal , 2v \big), \big(L_{X^\intercal} \bar{g}^\intercal, X^\intercal (H) \big) \Big \rangle \, d\sigma\\
	&=\int_\Sigma 2 \Big(- \mathrm{div}_\Sigma \big(vA-\nu(v) \bar{g}^\intercal\big)  + v d H\Big)\cdot X^\intercal\, d\sigma.
\end{align*}
Since $X^\intercal$ is arbitrary and by the Codazzi equation $\mathrm{div}_\Sigma A = dH$ on~$\Sigma$, we derive \eqref{equation:boundary-extra}. Furthermore,  \eqref{equation:boundary} and \eqref{equation:boundary-extra} for a harmonic function are equivalent to the following boundary condition:
\[
\nabla^2 v(\nu, \cdot)=0 \quad \mbox{ on } \Sigma.
\] 
To see that, since $v$ is harmonic, we have $0=\Delta v = \Delta_\Sigma v+ \nabla^2 v(\nu, \nu) + H\nu(v)$. Together with \eqref{equation:boundary}, it implies $\nabla^2 v(\nu, \nu)=0$ on $\Sigma$.  Letting $\{ e_1, \dots, e_{n-1},\nu\}$ be a local orthonormal frame, we write \eqref{equation:boundary-extra} as 
\begin{align*}
	0 &=\sum_b A_{ab} \nabla_{e_b} v - \nabla_{e_a} (\nu (v))\\
	&=\sum_b A_{ab} \nabla_{e_b} v - \big( \nabla^2 v(e_a, \nu) + \nabla_{\nabla_{e_a} \nu} v\big)\\
	&= - \nabla^2 v(e_a, \nu). 
\end{align*}

Lastly, we show that a harmonic function $v\in \C_{-q}^{2,\alpha}(\mathbb{R}^n\setminus \Omega)$ satisfying $\nabla^2 v(\nu, \cdot)=0$ on $\Sigma$ must vanish identically in $\mathbb{R}^n\setminus \Omega$. Fix $i=1,\dots, n$ and define $f = \frac{\partial v}{\partial x_i}$ where $( x_1,\dots, x_n)$ are the Cartesian coordinates. Since $\Delta v=0$, we also have $\Delta f=0$ in $\mathbb{R}^n\setminus \Omega$. Compute the Neumann boundary data $\nu(f)=  \nabla^2 v(\nu,  \frac{\partial}{\partial x_i})=0$ on $\Sigma$. Since $f$ decays at infinity, we have $f\equiv 0$ in $\mathbb{R}^n\setminus \Omega$ and thus $\nabla v=0$ by varying $i$. Since $v\to 0$ at infinity, we then conclude that $v$ is identically zero.

\end{proof}

In the following we define the Banach spaces $\mathcal{E}$ and $\mathcal{Q}$ associated with the operator $L$. We recall that the \emph{dual space} of a normed linear space $\mathcal{A}$, denoted by $\mathcal{A}^*$, is the space of all bounded linear functionals on~$\mathcal{A}$. 
\begin{definition}\label{definition:spaces}
\begin{enumerate}
\item The Banach space $\mathcal{E}$ is a complementing space of $\Ker L$ in $\C^{2,\alpha}_{-q}(\mathbb{R}^n\setminus \Omega)$,
consisting of pairs of a symmetric $(0,2)$-tensor $h$ and a scalar-valued function $v$ in $\mathbb{R}^n\setminus \Omega$, defined by
\[
	\mathcal{E}=\left\{ (h, v)\in \C^{2,\alpha}_{-q}(\mathbb{R}^n\setminus \Omega): \int_{\mathbb{R}^n\setminus \Omega} \big( h\cdot L_X \bar{g}  \big) \rho \, d\mathrm{vol} =0  \mbox{ for all } X\in \mathcal{X}_0\right\}.
\]
Recall the weight function $\rho(x) =(1+|x|^2)^{-1}$. 
\item
 The Banach space $\mathcal{Q}$, as a subspace of 
 the dual space of the codomain of $L$, is defined by 
\begin{align*}
	\mathcal{Q}=\Big\{ \kappa
	\in &\big(\C^{0,\alpha}_{-q-2}(\mathbb{R}^n \setminus \Omega)\times \C^{1,\alpha}(\Sigma)\times \mathcal{B}(\Sigma)\big)^*
	: \\
	\kappa&=\big(2\bi^* X,\, -\Div X, \, -2\bi^* X(\nu, \cdot), \, 0, \, 0\big) \mbox{ for some } X\in \mathcal{X}_0\Big\}.
\end{align*}
We may write $\kappa(X)$ for elements in $\mathcal{Q}$ to emphasize its dependence on~$X$. Here, $(2\bi^* X, -\Div X)\in \big(\C^{0,\alpha}_{-q-2}(\mathbb{R}^n \setminus \Omega)\big)^*$ is defined on $\mathbb{R}^n\setminus \Omega$, $2\bi^* X(\nu, \cdot)\in \big(\C^{1,\alpha}(\Sigma)\big)^*$ denotes the one-form that acts on  vector fields along $\Sigma$ (normal or tangential), and the last two components $(0,0)$ of $\kappa(X)$ belong to $(\mathcal{B}(\Sigma))^*$. The evaluation of $\kappa\in \mathcal{Q}$ on $\C^{0,\alpha}_{-q-2}(\mathbb{R}^n \setminus \Omega)\times \C^{1,\alpha}(\Sigma)\times \mathcal{B}(\Sigma)$ is defined by the $\mathcal{L}^2$-pairing. 
\end{enumerate}
\end{definition}

 We remark that $\mathcal{E}$ is defined so that  $(g, u)\in  \big((\bar{g}, 1)+\mathcal{E}\big) \cap \mathcal{M}$ if and only if  $(g, u)$ satisfies the orthogonal gauge.  We also remark that $\mathcal{Q}$ is isomorphic to $\mathcal{X}_0$, and thus $\Dim \mathcal{Q}=N$.

Let $(\Range L)^\perp\subset  \big(\C^{0,\alpha}_{-q-2}(\mathbb{R}^n \setminus \Omega)\times \C^{1,\alpha}(\Sigma)\times \mathcal{B}(\Sigma)\big)^*$ denote the \emph{annihilator space  of $\Range  L$}. Namely, $ (\Range L)^\perp$ consists of bounded linear functionals  whose evaluation on $L(h, v)$ is zero for all $(h, v)\in \C^{2,\alpha}_{-q}(\mathbb{R}^n\setminus \Omega)$.

\begin{lemma}\label{lemma:perp}
If the boundary $\Sigma$ is static regular in $\mathbb{R}^n\setminus \Omega$, then
\[
	(\Range  L)^\perp=  \mathcal{Q}.
\]
\end{lemma}
\begin{proof}

Since $L$ is Fredholm of index $0$ by Lemma~\ref{lemma:Fredholm}, we have $\Dim (\Range L)^\perp=\Dim\Ker L =N$. Since $\Dim \mathcal{Q}=N$, we just need to show that $\mathcal{Q}\subset (\Range  L)^\perp$. Pairing $\kappa(X)\in \mathcal{Q}$ with an arbitrary $(h, v)\in\C^{2,\alpha}_{-q}$ gives
\begin{align*}
	 \langle \kappa(X), L(h, v)\rangle_{\mathcal{L}^2}&= \int_{\mathbb{R}^n\setminus \Omega} \Big(-D\Ric(h)+\nabla^2 v - \D(\bi h+dv)\Big) \cdot 2\bi^* X \, d\mathrm{vol} \\
	&\quad + \int_{\mathbb{R}^n\setminus \Omega}( \Delta v)(-\Div X)\, d\mathrm{vol} - \int_\Sigma (\bi h+dv) \cdot 2\bi^* X(\nu, \cdot)\, d\sigma.
\end{align*}
We rearrange the integrands and compute
\begin{align*}
	\langle \kappa(X), L(h, v)\rangle_{\mathcal{L}^2} &= \int_{\mathbb{R}^n\setminus \Omega} \Big \langle\big( -D\Ric(h)+\nabla^2 v ,  \Delta v \big),\big( 2\bi^* X, -\Div X\big) \Big\rangle \, d\mathrm{vol}\\
	&\;  - \int_{\mathbb{R}^n\setminus \Omega}  \D(\bi h+dv) \cdot 2\bi^* X \, d\mathrm{vol}-\int_\Sigma (\bi h+dv) \cdot 2\bi^* X(\nu, \cdot)\, d\sigma\\
	&= \int_{\mathbb{R}^n\setminus \Omega} (\bi h + dv)\cdot (2 \Div \bi^* X)\, d\mathrm{vol}\\
	&\; - \lim_{r\to \infty} \int_{|x|=r}(\bi h+dv) \cdot 2\bi^* X\left(\tfrac{x}{|x|}, \cdot\right)\, d\sigma\\
	&=0
\end{align*}
where the integral in the first line vanishes by \eqref{equation:cokernel2}, in the second equality we apply the integration by parts (noting that $\nu$ on $\Sigma$ points away from $\Omega$), and in the last equality, we use $ 2\Div \bi^* X = \Delta X=0$ and the boundary integral vanishes at infinity since the decay rate of the integrand is $O(r^{-2q-1})$. 
\end{proof}

\begin{remark}\label{remark:cokernelL}
The computation above explains the motivation for the definition of $\kappa(X)\in \mathcal{Q}$. We have shown in Proposition~\ref{proposition:cokernel} that the annihilator space of the operator $(h, v)\mapsto (-D\Ric(h) + \nabla^2 v, \Delta v)$ contains elements $\kappa_0(X) = (2\bi^* X, -\Div X)$ for $X\in \mathcal{X}$. We let the first two components of $\kappa(X)$ be exactly  $\kappa_0(X)$ and use the third component $-2\bi^* X(\nu, \cdot)$ of $\kappa(X)$ to take care of the  ``gauge terms'' in the definition of $L$. 
\end{remark}

In the next lemma, we identify a complementing space of $\Range L$. Intuitively, one may expect that the annihilator space $\mathcal{Q}$ complements $\Range L$, but the elements in $\mathcal{Q}$ may not even decay faster enough to be in the codomain of $L$. Nevertheless, we can  remedy the problem  by  multiplying a suitable weight function. 

Let $\eta$ be a positive, smooth function in $\mathbb{R}^n \setminus \Omega$ satisfying $\eta = 1$ near $\Sigma$ and {$\eta (x) = |x|^{-1}$} outside a large ball. Then the following space $\eta\mathcal{Q}$ is a well-defined subspace of  $\C^{0,\alpha}_{-q-2}(\mathbb{R}^n \setminus \Omega)\times \C^{1,\alpha}(\Sigma)\times \mathcal{B}(\Sigma)$:
\[
	\eta \mathcal{Q}= \{ \eta \kappa: \kappa \in \mathcal{Q}\}.
\]
Note that $\Dim \eta \mathcal{Q}=\Dim \mathcal{Q}= N$.  

\begin{lemma}\label{lemma:cokernel}
If the boundary $\Sigma$ is static regular in $\mathbb{R}^n\setminus \Omega$,  then $\Range L$ is complemented by $\eta \mathcal{Q}$. Namely,
\[
	\C^{0,\alpha}_{-q-2}(\mathbb{R}^n \setminus \Omega)\times \C^{1,\alpha}(\Sigma)\times \mathcal{B}(\Sigma)= \Range L \oplus \eta \mathcal{Q}.
\]
\end{lemma}
\begin{proof}
Using the weight function $\eta$, we define an inner product on $\mathcal{Q}$ as follows. Let $\kappa(X), \gamma(Z)\in \mathcal{Q}$  for $X, Z\in \mathcal{X}_0$. That is,
\begin{align*}
\kappa (X)&=\big(2\bi^* X,\, -\Div X, \, -2\bi^* X(\nu, \cdot), \,0,\, 0\big)\\
\gamma (Z)& =\big(2\bi^* Z,\, -\Div Z, \,-2\bi^* Z(\nu, \cdot),\, 0, \,0\big).
\end{align*}
Define the inner product between $\kappa(X)$ and $\gamma(Z)$ by 
\begin{align*}
& \int_{\mathbb{R}^n\setminus \Omega} \eta (2\bi^* X) \cdot (2\bi^* Z)\, d\mathrm{vol}  +\int_{\mathbb{R}^n\setminus \Omega} \eta (\Div X) (\Div Z) \, d\mathrm{vol}   \\
&+ \int_\Sigma  \big(2\bi^* X(\nu, \cdot)\big)\cdot \big(2\bi^* Z(\nu, \cdot)\big)\, d\sigma.
\end{align*}
With respect to this inner product, we  choose an orthonormal basis of $\mathcal{Q}$, denoted by $ \{ \kappa_1,\dots, \kappa_N\}$.
 
 For $f\in\C^{0,\alpha}_{-q-2}(\mathbb{R}^n \setminus \Omega)\times \C^{1,\alpha}(\Sigma)\times \mathcal{B}(\Sigma)$ and for each $\ell = 1, \dots, N$, define the constant $a_\ell =\langle  \kappa_\ell, f\rangle_{\mathcal{L}^2}$.   By construction, we have $\langle \kappa_m, f - \eta \sum_\ell a_\ell \kappa_\ell\rangle_{\mathcal{L}^2}=0$ for all $m$.  By Lemma~\ref{lemma:perp} and that $\Range L$ is closed, we have
  \[
 	f- \eta \sum_\ell a_\ell \kappa_\ell \in \Range L. 
 \]
A similar argument also shows that $(\Range L) \cap (\eta \mathcal{Q})$ contains only the zero element. This completes the proof. 
\end{proof}

Let us take a moment to summarize what we have obtained so far. We can re-express the linear map $L: \C^{2,\alpha}_{-q}(\mathbb{R}^n\setminus \Omega)\to \C^{0,\alpha}_{-q-2}(\mathbb{R}^n \setminus \Omega)\times \C^{1,\alpha}(\Sigma)\times \mathcal{B}(\Sigma)$ in the form:
\begin{align*}
	L : \mathcal{E} \oplus \Ker L \to \Range L \oplus \eta \mathcal{Q}
\end{align*}
where $\mathcal{E}, \mathcal{Q}$ are defined in Definition~\ref{definition:spaces} and $L$ is injective on $\mathcal{E}$. Since the elements in $\eta \mathcal{Q}$ all have their last two components equal to zero, $\Range L$ must contain the subspace $\{ 0\}\times \{ 0 \} \times \mathcal{B}(\Sigma)$. In other words, for any $(\tau,\phi) \in \mathcal{B}(\Sigma)$, there exists a unique deformation $(h, v)$ satisfying the orthogonal gauge such that $L(h, v) = (0, 0, 0, \tau, \phi)$. While that looks a lot like a version of Theorem~\ref{theorem:T} for the linearized operator $L$, we cannot directly conclude that $\Range T$ must contain $(0,0,0,\bar{g}^\intercal +\tau, H+\phi)$ for $(\tau, \phi)$ small. It is because that the splitting for $\Range L$ need not imply a splitting of the nonlinear operator~$T$ (needless to say that $\Range T$ is not even a linear subspace). Therefore, we will consider the modified operators as in the next section. 

\subsection{Orthogonality for the modified nonlinear map}\label{section:proof}

We will introduce a new operator $\overline{T}$, modified from $T$, and prove Theorem~\ref{theorem:T}. To motivative our approach below, we discuss the basic idea. This part of discussion holds for a general map $T$ and its linearization $L$ (say, linearized at  $0$), so the reader can ignore for now how the map $T$ was defined above. Suppose the linear map is given by
\[
L : \mathcal{E} \oplus \Ker L \to \Range L \oplus \mathcal{Z},
\]
 where $\mathcal{E}, \mathcal{Z}$ are Banach spaces. Without changing the codomain space, we can ``substitute'' $\Ker L$ with $\mathcal{Z}$ and define an isomorphism by
\begin{align*}
	\overline{L}: \mathcal{E}\times \mathcal{Z} &\to  \Range L \oplus \mathcal{Z}\\
	(e, z) &\mapsto \overline{L}(e, z) = L(e)+z.
\end{align*}
 If we define a nonlinear map by $\overline{T}(e, z)= T(e) +z$, then $\overline{T}$ is a local diffeomorphism at $(e, z)=0$ by Inverse Function Theorem. Thus, for $v$ sufficiently close to $T(0)$, there exists $(e, z)$ solving $T(e)+z = v$. If  $z$ is ``orthogonal'' to both $T(e)$ and $v$, then we would get $z=0$ and solve $T(e) = v$ as desired. However, this may be too rigid to get, so we would like to take the advantage of the flexibility in defining  $\overline{T}$. We just need its linearization equal to $\overline{L}$; in particular,  $\overline{T}$ can be nonlinear in~$z$. Our definition of $\overline{T}$ below  reflects this basic idea, 
with a slight complication that we need to enlarge the space $\mathcal{Z}$ for the domain of $\overline{T}$ in order to impose  nonlinear conditions on $z$.

 Let $\widehat{\mathcal{X}}$ be the linear space of vectors defined similarly as $\mathcal X$, except with less regularity:
\begin{align*}
	\widehat{\mathcal{X}} = \left\{\mbox{ $X=0$ on $\Sigma$, $X - K\in  \C^{2,\alpha}_{1-q} (\mathbb{R}^n\setminus \Omega)$ for some Killing vector $K$ of $\bar{g}$}\right\}.
\end{align*}
Define the map $\overline{T}:\Big( \big((\bar{g}, 1)+\mathcal{E}\big)\cap \mathcal{M}\Big) \times  \widehat{\mathcal{X}} \to \C^{0,\alpha}_{-q-2}(\mathbb{R}^n \setminus \Omega)\times \C^{0,\alpha}_{-q-1}(\mathbb{R}^n \setminus \Omega) \times \C^{1,\alpha}(\Sigma)\times \Big( ( \bar{g}^\intercal, H) +\mathcal{B}(\Sigma)\Big)$ by, for $(g, u) \in  \big((\bar{g}, 1)+\mathcal{E}\big)\cap \mathcal{M}$ and $W \in  \widehat{\mathcal{X}}$,
\begin{align*}
	&\overline{T} (g, u, W)\\
	&= \begin{array}{l}
	\left\{ \begin{array}{l}-u \Ric_g + \nabla_g^2 u  -\D_g (\bi g+du) + \eta \big(2\bi_g^* W- u^{-1} W(u) g \big)\\
	\Delta_g u + \eta \big(-\Div_g W + u^{-1} W(u)\big)\\
	\Div_g (2\bi_g^* W- u^{-1} W(u) g)
	\end{array}\right. \quad \mbox{ in } \mathbb{R}^n\setminus \Omega \\
	\left\{ \begin{array}{l} \bi g+ du - (2\bi_g^* W - u^{-1} W(u) g)(\nu, \cdot)\\
	g^\intercal\\
	H_g
	\end{array}\right.\quad \mbox{ on } \Sigma
	\end{array}
\end{align*}
where $\nu$ is the unit normal vector on $\Sigma$ pointing away from $\Omega$. We should compare $\overline{T}$ to the original map $T$ in \eqref{equation:nonlinear}. For a vector $W$, we define  
\begin{align*}
	\zeta(W)  &=\Big( 2\bi_g^* W - u^{-1} W(u) g, \,-\Div_g W + u^{-1} W(u),\,\\
	&\qquad \qquad \qquad \qquad  -(2\bi_g^* W -u^{-1} W(u) g )(\nu, \cdot),\, 0, \,0\Big).
\end{align*}
Then  $\overline{T}(g, u, W)$, after dropping the third equation for $\Div_g (2\bi_g^* W- u^{-1} W(u) g)$, is exactly  $T(g, u) +\eta \zeta(W)$. We also note that the first two components of $\zeta(W)$ is the same as $\zeta_0(W)$ from \eqref{equation:cokernel-elements}. The relation between $\zeta(W)$ and $\zeta_0(W)$ is analogous to that of $\kappa(X),\kappa_0(X)$  in Remark~\ref{remark:cokernelL}.

\begin{proposition}\label{proposition:isom}
Suppose the boundary $\Sigma$ is static regular in $\mathbb{R}^n\setminus \Omega$. Then the following holds:
\begin{enumerate}
\item \label{item:1} $\overline{T}$ is a local diffeomorphism at $\big(\bar{g}, 1, 0\big)$. 
\item \label{item:2} If $(g,u, W)$ solves $\overline{T}(g, u, W) = (0, 0, 0, 0, \tau,\phi)$; that is, $(g^\intercal, H_g) = (\tau, \phi)$ on $\Sigma$ and all other equations in $\overline{T}$ equal zero,  then  $T(g, u) = (0, 0, 0, \tau, \phi)$.
\end{enumerate}
\end{proposition}
\begin{proof}
Denote the linearization of $\overline{T}$ at $(\bar{g}, 1, 0)$ by $\overline{L}: \mathcal{E}\times \widehat {\mathcal{X}}\to \C^{0,\alpha}_{-q-2}(\mathbb{R}^n \setminus \Omega)\times\C^{0,\alpha}_{-q-1}(\mathbb{R}^n \setminus \Omega)\times \C^{1,\alpha}(\Sigma)\times \mathcal{B}(\Sigma)$. For $(h, v)\in \mathcal{E}$ and $X\in  \widehat {\mathcal{X}}$, 
\begin{align*}
	\overline{L} (h, v, X)= \begin{array}{l}
	\left\{ \begin{array}{l}-D \Ric(h) + \nabla^2 v  -\D (\bi h+dv) + 2\eta \bi^* X\\
	\Delta v - \eta \Div X \\
	2\Div \bi^* X
	\end{array}\right. \quad \mbox{ in } \mathbb{R}^n\setminus \Omega \\
	\left\{ \begin{array}{l} \bi h+ dv- 2\bi^* X (\nu, \cdot)\\
	h^\intercal\\
	DH(h)
	\end{array}\right.\quad \mbox{ on } \Sigma. 
	\end{array}
\end{align*}
Recall $\kappa(X) = (2\bi^* X, \, -\Div X,\, - 2\bi^* X (\nu, \cdot), 0, 0)$. 
Just as for $\overline{T}$ and $T$,  the linear map $\overline{L}(h, v, X)$, with the third equation for $2\Div \bi^* X$ dropped, is exactly $L(h, v)+ \eta \kappa(X)$, where $L$ is defined in~\eqref{equation:linear}.

We claim that $\overline{L}$ is an isomorphism. Then by  Inverse Function Theorem, it proves that $\overline{T}$ is a local diffeomorphism at $\big(\bar{g}, 1, 0\big)$. 

We show $\overline{L}$ is surjective: Since the third equation $2\Div \bi^* X = \Delta X$ is surjective by Lemma~\ref{lemma:PDE}, we just need to show that $\overline{L}$ is surjective onto other components for those $X$ solving $2\Div \bi^* X=0$, i.e. $X\in \mathcal{X}_0$. It is equivalent to that $L(h, v) + \eta \kappa(X)$ is surjective for  $(h,v)\in \mathcal{E}, X\in \mathcal{X}_0$, which follows from Lemma~\ref{lemma:cokernel}. We show  $\overline{L}$ is injective: For $(h, v, X)$ solving $\overline{L}(h , v, X)=0$,  we have $X\in \mathcal{X}_0$, and $(h, v, X)$ satisfies  $L(h, v) + \eta \kappa(X)=0$.  Again by  the decomposition from Lemma~\ref{lemma:cokernel}, we have $L(h, v)=0$ for $(h, v)\in \mathcal{E}$ and $\kappa(X)=0$, which imply that $(h, v, X)$ is zero.

Before we prove the second statement, we note that if $W$ satisfies $\Div_g (2\bi_g^* W- u^{-1} W(u) g)=0$, then $\zeta(W)$ is $\mathcal{L}^2$-orthogonal to $T(g, u)$:
\begin{align*}
	&\langle T(g, u), \zeta(W) \rangle_{\mathcal{L}^2}\\
	&=\int_{\mathbb{R}^n\setminus \Omega} \Big( - u \Ric_g + \nabla^2_g u - \D_g (\bi g + du) \Big) \cdot \Big(2\bi_g^* W - u^{-1} W(u) g\Big)\, d\mathrm{vol}_g \\
	&\quad + \int_{\mathbb{R}^n\setminus \Omega} \Delta_g u \Big(-\Div_g W - u^{-1} W(u)\Big)\, d\mathrm{vol}_g\\
	&\quad - \int_\Sigma (\bi g+ du)\cdot (2\bi_g^* W - u^{-1} W(u) g ) (\nu ,\cdot) \, d\sigma_g\\
	&=\int_{\mathbb{R}^n\setminus \Omega} \Big(- \D_g (\bi g + du) \Big) \cdot \Big(2\bi_g^* W - u^{-1} W(u) g\Big)\, d\mathrm{vol}_g \\
	&\quad - \int_\Sigma (\bi g+ du)\cdot (2\bi_g^* W - u^{-1} W(u) g ) (\nu ,\cdot) \, d\sigma_g\\
	&=0
\end{align*} 
where we use \eqref{equation:cokernel1} in the second equality and integration by parts in the last line. Note that the boundary integrals on $\Sigma$ cancel, and the boundary integral at infinity vanishes because of the fall-off rates.

Now suppose $(g, u, W)$ satisfies $\overline{T}(g, u, W) = (0, 0, 0, 0, \tau,\phi)$. This is equivalent to 
\begin{align*}
		T(g, u) + \eta \zeta(W) &= (0, 0, 0, \tau, \phi)\\
		\Div_g (2\bi_g^* W- u^{-1} W(u) g)&=0.
\end{align*}
  Take the $\mathcal{L}^2$-inner product of the first identity with $ \zeta(W)$, and note  that $\big\langle T(g, u), \zeta(W)\big\rangle_{\mathcal{L}^2}=0$ as  above and  $\big\langle (0,0,0,\tau, \phi), \zeta(W) \big\rangle_{\mathcal{L}^2} =0$ trivially. So we have $\big\langle \eta  \zeta(W),  \zeta(W) \big\rangle_{\mathcal{L}^2}=0$ and thus $\zeta(W)=0$. This implies the desired conclusion $T(g, u) = (0, 0, 0, \tau, \phi)$. In fact, this implies that $W$ is a Killing vector of $g$ vanishing on the boundary and thus $W\equiv 0$ in  $\mathbb R^n\setminus\Omega$.
 \end{proof}

Combining the above results, we give the proof of Theorem~\ref{theorem:T}.
\begin{proof}[Proof of Theorem~\ref{theorem:T}]
From Item~\eqref{item:1} in Proposition~\ref{proposition:isom}, $\overline{T}$ is a local diffeomorphism at $(\bar{g}, 1)$. More specifically, there are positive constants $\epsilon_0, C$  such that for each $\epsilon \in (0, \epsilon_0)$, there is an open neighborhood $\mathcal{U}\times \mathcal{V}$ of $\big(\bar{g}, 1, 0\big)$ in $\big((\bar{g}, 1)+\mathcal{E}\big)\times \widehat {\mathcal{X}}$ with the diameter less than $C\epsilon$, such that $\overline{T}$ is a diffeomorphism from $\mathcal{U}\times \mathcal{V}$ onto an open ball of radius $\epsilon$ in $\C^{0,\alpha}_{-q-2}(\mathbb{R}^n \setminus \Omega)\times \C^{0,\alpha}_{-q-1}(\mathbb{R}^n \setminus \Omega)\times\C^{1,\alpha}(\Sigma)\times \Big( (\bar{g}^\intercal, H) +\mathcal{B}(\Sigma)\Big)$, centered at $(0, 0, 0, 0, \bar{g}^\intercal, H)$. Therefore, given $(\tau, \phi)$ with $\|(\tau, \phi)-(\bar{g}^\intercal, H)\|_{\C^{2,\alpha}(\Sigma)\times \C^{1,\alpha}(\Sigma)}<\epsilon$, there exists a unique $(g, u, W)$ with $\|(g, u) - (\bar{g}, 1) \|_{\C^{2,\alpha}_{-q} (\mathbb{R}^n\setminus \Omega)} <C\epsilon$ and  $\| W \|_{ \C^{2,\alpha}_{1-q}(\mathbb{R}^n\setminus \Omega)}<C\epsilon$  satisfying 
\[
	\overline{T}(g, u, W) = (0, 0, 0, 0, \tau, \phi). 
\]
By Item~\eqref{item:2} in Proposition~\ref{proposition:isom},  $(g, u)$ satisfies $T(g, u) = (0, 0, 0, \tau, \phi)$. Since $(g, u) \in (\bar{g},1)+\mathcal{E}$, $(g, u)$ satisfies the orthogonal gauge. 
\end{proof}
\begin{remark}\label{re:constant}
Note that the constants $\epsilon_0, C$ depend on the global geometry $\Sigma$ in $\mathbb{R}^n\setminus \Omega$. More precisely by the (quantitive) inverse function theorem, the constants depend on the operator norms of $\overline{L}$, $\overline{L}^{-1}$, as well as the second Frech\'et derivative $D^2\overline{T}|_{(g, u)}$ for $(g, u)$ is a neighborhood of $(\bar g, 1)$,  between the  function spaces on $\mathbb{R}^n\setminus \Omega$ and $\Sigma$ as specified above.
\end{remark}

We explain the reason for the terminology 
\emph{static regular}. The following discussion will not be used elsewhere in the paper. For a general $\Omega$ whose boundary $\Sigma = \partial \overline{\Omega}$ is not necessarily static regular in $\mathbb{R}^n\setminus \Omega$, define the set of static vacuum pairs (of any boundary data) near $(\bar{g}, 1)$ by
\[
	\mathcal{S}(\mathbb{R}^n\setminus \Omega) = \left\{ (g, u)\in \mathcal{M}: -u\Ric_g + \nabla_g^2 u=0 \mbox{ and } \Delta_g u=0  \mbox{ in } \mathbb{R}^n\setminus \Omega \right\}.
\]
As in \cite{Anderson-Khuri:2013}, the  \emph{Bartnik boundary map} $\Pi: \mathcal{S}(\mathbb{R}^n\setminus \Omega) \to (\bar{g}^\intercal, H) +\mathcal{B}(\Sigma)$ is defined by $\Pi(g, u)= (g^\intercal, H_g)$. Conjecture~\ref{conjecture} becomes finding conditions of $\mathbb{R}^n\setminus \Omega$ so that the boundary map~$\Pi$ is locally surjective. In that formulation, Theorem~\ref{theorem:criterion} says if the boundary $\Sigma$ is static regular in  $\mathbb{R}^n\setminus \Omega$, then $(\bar{g}, 1)$ is a \emph{regular point} of the map~$\Pi$. Of course, the set $\mathcal{S}(\mathbb{R}^n\setminus \Omega)$ may not be a Banach manifold for us to talk about regular points. But we can restrict the boundary map~$\Pi$ on the subset $\mathcal{S}_0(\mathbb{R}^n\setminus \Omega)\subset \mathcal{S}(\mathbb{R}^n\setminus \Omega)$ that  consists of  $(g, u)\in \mathcal{S}(\mathbb{R}^n\setminus \Omega)$ satisfying both the static-harmonic gauge and orthogonal gauge. It is shown that $\mathcal{S}_0(\mathbb{R}^n\setminus \Omega)$ is a Banach manifold \cite{An:preprint} (Cf. \cite{Anderson-Khuri:2013}), so the above discussion can make sense.

\section{Convex surfaces in $\mathbb{R}^3$}\label{section:convex}

\subsection{Mean curvature comparison}

Let $\Omega$ be a bounded open subset in the Euclidean space $(\mathbb{R}^3, \bar{g})$ whose boundary $\Sigma =\partial \overline{\Omega}$ is a connected, embedded surface. Before we restrict to the case that $\Sigma$ is convex, we first prove the following general lemma. Given a harmonic function $v$ in $\mathbb{R}^3\setminus \Omega$, we will construct a family of scalar-flat metrics $\gamma(s)$ in $\mathbb{R}^3\setminus \Omega$ and a corresponding family of scalar-flat metrics $\gamma_{\I}(s)$ in the compact region $\overline{\Omega}$ that induce the same metrics on $\Sigma$. We shall use the notational convention that the quantities in the compact region $\overline{\Omega}$ are denoted with the subscript ${\I}$. 

\begin{lemma}\label{lemma:harmonic-generate}
Let $v\in \C^{2,\alpha}_{-q}(\mathbb{R}^3\setminus  \Omega)$ solve $\Delta v=0$ in $\mathbb{R}^3\setminus \Omega$. Let $v_{\I}\in \C^{2,\alpha}\big(\overline{\Omega}\big)$ satisfy $\Delta v_{\I}=0$ in $\Omega$ and $v_{\I} = v$ on $\Sigma$. 

For $|s|$ small, there exists  a smooth family of asymptotically flat metrics $\gamma(s)$ in $\mathbb{R}^3\setminus \Omega$ and a family of metrics $\gamma_{\I}(s)$ in $\overline{\Omega}$ such that the following holds:
\begin{itemize}
\item $\gamma(0)= \bar{g}$, $R_{\gamma(s)}=0$, $\gamma'(0) = 2v\bar{g}$ in $\mathbb{R}^3\setminus \Omega$.
\item $\gamma_{\I}(0)= \bar{g}$, $R_{\gamma_{\I}(s)}=0$, ${\gamma_{\I}}'(0) = 2v_{\I}\bar{g}$ in $\overline{\Omega}$.

\item On $\Sigma$, $\gamma(s)^\intercal =\gamma_{\I}(s)^\intercal= (1+sv)^2 \bar{g}^\intercal$ for all $s$. 
\item Let $H_{\gamma_{\I}(s)}, H_{\gamma(s)}$ be the mean curvatures of $\Sigma$ in $\big(\overline{\Omega}, \gamma_{\I}(s)\big), (\mathbb{R}^3\setminus \Omega, \gamma(s))$, respectively (both with respect to the unit normals that points away from $\Omega$). Then 
\begin{align*}
	\left. \2ds \right|_{s=0} \bigg( -16\pi m_{\mathrm{ADM}}(\gamma(s)) + &\int_{\mathbb{R}^3\setminus \Omega} R_{\gamma(s)}\, d\mathrm{vol}_{\gamma(s)}\\
	&+  \int_\Sigma 2\Big(H_{\gamma_{\I}(s)} - H_{\gamma(s)}\Big)\, d\sigma_{\gamma(s)} \bigg)&\ge 0.
\end{align*}	 
\end{itemize}
\end{lemma}
\begin{proof}
Consider the asymptotically flat metric $\hat{g}(s)=(1+sv)^2 \bar{g}$. By the conformal  formula and that $v$ is harmonic, the scalar curvature of $\hat{g}(s)$ is
\[
	R_{\hat{g}(s)} = \frac{2s^2 |\nabla v|^2}{(1+sv)^4}.
\]
For $|s|$ sufficiently small, there exists a unique function $u(s)$ with  $\| u(s)-1\|_{\C^{2,\alpha}_{-q}(\mathbb{R}^3\setminus \Omega)}\ll 1$ that solves
\begin{align*}
	\Delta_{\hat{g}(s)} u(s) - \tfrac{1}{8} R_{\hat{g}(s)} u(s) &= 0 \quad \mbox{ in } \mathbb{R}^3\setminus \Omega \\
	u(s) &= 1 \quad \mbox{ on } \Sigma.
\end{align*}
By differentiating the above equations in $s$ and using that $\left. \ds\right|_{s=0}R_{\hat{g}(s)}=0$ and $u(0)=1$, we see that $u'(0)=O^{2,\alpha}(|x|^{-q})$ is harmonic with $u'(0)=0$ on $\Sigma$, and thus $u'(0)$ is identically zero. 
Define $\gamma(s) = u(s)^4 \hat{g}(s)$ in $\mathbb{R}^3\setminus \Omega$.  Then $\gamma(s)$ has zero scalar curvature by the conformal formula, and we have $\gamma'(0)  = 2v\bar{g}$. This proves  the first item.

Similarly, we define $\hat{g}_{\I}(s)=(1+sv_{\I})^2 \bar{g}$ in $\overline{\Omega}$ and have
\[
	R_{\hat{g}_{\I}(s)} = \frac{2s^2 |\nabla v_{\I}|^2}{(1+sv_{\I})^4}.
\]
Then, for $|s|$ small, we let $u_{\I}(s)$ with  $\| u_{\I}(s)-1\|_{\C^{2,\alpha}\big(\overline{\Omega}\big)}\ll 1$ be the unique solution to 
\begin{align*}
	\Delta_{\hat{g}_{\I}(s)} u_{\I}(s) - \tfrac{1}{8} R_{\hat{g}_{\I}(s)} u_{\I}(s) &= 0 \quad \mbox{ in } \Omega \\
	u_{\I}(s) &= 1 \quad \mbox{ on } \Sigma.
\end{align*}
The same argument  as above shows that $u_{\I}'(0)=0$, and we can define $\gamma_{\I}(s) = u_{\I}(s)^4 \hat{g}_{\I}(s)$ corresponding. It proves the second item. 

By the boundary conditions $v=v_{\I}$ and $u(s) = 1 = u_{\I}(s)$ on $\Sigma$, we have 
\[
	\gamma(s)^\intercal =\gamma_{\I}(s)^\intercal= \hat{g}(s)^\intercal
\]
and the third item follows.

We prove the last item. We begin by computing the ADM mass of $\gamma(s)$. 
\begin{align*}
	16\pi m_{\mathrm{ADM}}(\gamma(s)) &= \lim_{r\to \infty} \int_{|x|=r}\sum_{i,j} \left(\frac{\partial \gamma(s)_{ij}}{\partial x_i} - \frac{\partial \gamma(s)_{ii}}{\partial x_j} \right)\frac{x_j}{|x|} \, d\sigma\\
	&= -2  \lim_{r\to \infty} \int_{|x|=r}\left(  2s \frac{\partial v}{\partial r} + 4 \frac{\partial u(s)}{\partial r} \right)\, d\sigma.
\end{align*}

To compute the mean curvatures, by the conformal formula and that $u(s)=1=u_{\I}(s), v=v_{\I}$ on~$\Sigma$, we have 
\begin{align*}
H_{\gamma_{\I}(s)} &=(1+sv)^{-1} H + 2(1+sv)^{-2} s \nu (v_{\I}) + 4 \nu_{\hat{g}_{\I}(s)}\big(u_{\I}(s)\big)\\
	H_{\gamma(s)} &= (1+sv)^{-1} H + 2(1+sv)^{-2} s \nu (v) + 4\nu_{\hat{g}(s)}\big(u(s)\big)
\end{align*}
where $H$ is the mean curvature $\Sigma$ in $(\mathbb{R}^3,\bar{g})$ and note $\nu_{\hat{g}} = \nu_{\hat{g}_{\I}} = (1+sv)^{-1} \nu$. Thus, 
\[
	H_{\gamma_{\I}(s)} - H_{\gamma(s)}= 2(1+sv)^{-2} s \Big( \nu(v_{\I})-\nu(v) \Big) + 4\Big( \nu_{\hat{g}_{\I}(s)}\big(u_{\I}(s)\big)- \nu_{\hat{g}(s)}\big(u(s)\big)\Big).
\]
Integrating the previous identity over $(\Sigma, \gamma(s)^\intercal)$ and using that 
\[
	d\sigma_{\gamma(s)}=  (1+sv)^2 d\sigma
\]
where $d\sigma$ is the area form of $(\Sigma, \bar{g})$, we get 
\begin{align*}
	&\int_\Sigma \Big( H_{\gamma_{\I}(s)} -H_{\gamma(s)}\Big)\, d\sigma_{\gamma(s)} \\
	&= \int_\Sigma 2s \Big( \nu(v_{\I})-\nu(v) \Big)\, d\sigma + \int_\Sigma 4\Big( \nu_{\hat{g}_{\I}(s)}\big(u_{\I}(s)\big)- \nu_{\hat{g}(s)}\big(u(s)\big)\Big)\,d\sigma_{\gamma(s)}.
\end{align*}
Apply the divergence theorem to the right hand side  and note that the unit normals point away from $\Omega$ (and $\nu_{\hat{g}_{\I}(s)}=\nu_{\hat{g}(s)}=\nu_{\gamma(s)} = \nu_{\gamma_{\I}(s)}$ on $\Sigma$). Since $v, v_{\I}$ are harmonic functions of $\bar{g}$, we derive
\vspace{8pt}
\begin{align*} 
	\int_\Sigma \Big( H_{\gamma_{\I}(s)} -&H_{\gamma(s)}\Big)\, d\sigma_{\gamma(s)}= -\lim_{r\to \infty} \int_{S_r} \left( 2s\frac{\partial v}{\partial r}  + 4\frac{\partial u(s)}{\partial r}\right)\, d\sigma \\
	&\quad \;\;+ 4\int_{\overline{\Omega}} \Delta_{\gamma_{\I}(s)} u_{\I}(s)\, d\mathrm{vol}_{\gamma_{\I}(s)}+ 4\int_{\mathbb{R}^3\setminus \Omega} \Delta_{\gamma(s)} u(s) \, d\mathrm{vol}_{\gamma(s)} .
\end{align*}
Substituting the boundary integral with the ADM mass identity computed above and the Laplace terms using the differential equations that define $u(s)$ and $u_{\I}(s)$, we obtain
\begin{align*}
	&-16\pi m_{\mathrm{ADM}}(\gamma(s)) + \int_\Sigma 2\Big( H_{\gamma_{\I}(s)}-H_{\gamma(s)} \Big)\, d\sigma_{\gamma(s)}\\
	&\qquad \qquad =   \int_{\overline{\Omega}} R_{\hat{g}_{\I}(s)} u_{\I}(s)\, d\mathrm{vol}_{\gamma_{\I}(s)}+ \int_{\mathbb{R}^3\setminus \Omega} R_{\hat{g}(s)} u(s) \, d\mathrm{vol}_{\gamma(s)}.
\end{align*}
Taking two-derivative of the previous identity in $s$ and using  $u(0)=u_{\I}(0)=1, u'(0)=u_{\I}'(0)=0$, $\left.\2ds\right|_{s=0} R_{\hat{g}(s)}\ge 0$, $\left. \2ds\right|_{s=0} R_{\hat{g}_{\I}(s)}\ge 0$, and that $R_{\gamma(s)}=0$ for all $s$,   we finish the proof.  
\end{proof}

For the rest of the section, we assume that $\Sigma$ is convex, i.e., the Gauss curvature of $(\Sigma, \bar{g}^\intercal)$ is positive everywhere. If $g$ is an asymptotically flat metric such that $(\Sigma, g^\intercal)$ has positive Gauss curvature,  it can be isometrically embedded $(\Sigma, g^\intercal)$ into $(\mathbb{R}^3, \bar{g})$.  We define $ \mathring{A}_{g^\intercal}, \mathring{H}_{g^\intercal}$ to be, respectively, the second fundamental form and mean curvature of the isometric image in $(\mathbb{R}^3, \bar{g})$ with respect to the unit normal pointing to infinity. They are well-defined by uniqueness (up to a rigid motion) of isometric embeddings. Define the following functional
\begin{align*}
		\mathscr{G}(g) &= -16\pi m_{\mathrm{ADM}}(g) + \int_{\mathbb{R}^3\setminus \Omega} R_g \, d\mathrm{vol}_g + \int_\Sigma 2\big(\mathring{H}_{g^\intercal} - H_g\big) \, d\sigma_{g}.
\end{align*}

We will compute variations of $\mathscr{G}$ along a family of metrics $\gamma(s)$ whose induced metrics $\gamma(s)^\intercal$ vary smoothly in $s$. As a preparation, we note the following two lemmas on the smooth dependence of isometric embeddings and of diffeomorphism extensions.

\begin{lemma}[Smooth dependence of isometric embeddings]\label{lemma:embedding}
Let $s\in (-\epsilon, \epsilon)$ and let $\tau(s)\in \C^{2,\alpha}(S^2)$ be a family of metrics on $S^2$ with positive Gauss curvature such that $\tau(s)$ is $\C^k$ in $s$. Suppose the Gauss curvatures are bounded below by a positive constant uniformly in $s$. Then there exists a  family of isometric embeddings $f(s) : (S^2, \tau(s))\to \mathbb{R}^3$ such that each $f(s)\in \C^{2,\alpha}(S^2)$ and $f(s)$ is $\C^k$ in $s$. 
\end{lemma}
\begin{proof}
The existence of an isometric embedding  $f(s)$ for each $s$ is by Nirenberg~\cite{Nirenberg:1953}  (for  metrics of $\C^{2,\alpha}$ regularity, see, for example, \cite[Theorem 2]{Heinz:1962}  and \cite[Theorem 10.3.2]{Schulz:1990}).  Here we explain that if the family of metrics is $\C^k$ in the parameter $s$, then the embedding $f(s)$ obtained in ~\cite{Nirenberg:1953} is also $\C^k$ in~$s$. Let $\mathcal{B}_R $ consist of  maps from $S^2$ to $\mathbb{R}^3$ defined by $\mathcal{B}_R = \{Z\in \C^{2,\alpha}( S^2; \mathbb{R}^3): \| Z \|_{\C^{2,\alpha}(S^2)} \le R\}$. Let $s_0\in (-\epsilon, \epsilon)$. By Sections 5-9 in  \cite{Nirenberg:1953}, 
 there exists a constant $R>0$ such that for each $s$ sufficiently close to $s_0$ one can construct a contraction map  $F_s: \mathcal{B}_R\to \mathcal{B}_R$ that sends $Z\in \mathcal{B}_R$ to $F_s(Z)$ where $F_s(Z)$ solves
\[
	2 df(s_0)\cdot dF_s(Z) = \tau(s) - \tau(s_0) - dZ\cdot dZ \quad \mbox{ on } S^2.
\]
Moreover, going through the construction in \cite{Nirenberg:1953}, one can verify that the contraction constants of $F_s$ can be chosen uniformly in~$s$, and  the family of  contraction maps $\{F_s\}$ is $\C^k$ in $s$.  Let $Z(s)$ be the unique fixed point of $F_s$. A general fact about contraction maps then says that the fixed point $Z(s)$ is $\C^k$ in $s$. Lastly, define $f(s) = f(s_0)+Z(s)$. Then $f(s)$ is $\C^k$ in $s$ and  $df(s)\cdot df(s) = \tau(s)$, i.e., $f(s)$  is an isometric embedding of $(S^2, \tau(s))$.

\end{proof}

\begin{lemma}[Smooth dependence of diffeomorphism extensions]\label{lemma:extension}
 Let  $f_s:\Sigma \to \mathbb{R}^3$ be a smooth family of $\C^{2,\alpha}$ embeddings with $f_0={\rm Id}_{\Sigma}$. Denote by $\Sigma_s=f_s(\Sigma)$ and by $\Omega_s$ the region bounded by $\Sigma_s$.  Suppose that each $\Sigma_s$ is a convex surface enclosing  the origin and that $\{\Sigma_s\}$ is contained in a compact subset. Then, there is a smooth family of maps $\psi_s: \mathbb{R}^3\setminus \Omega \to \mathbb{R}^3\setminus \Omega_s$ such that $\psi_0$ is the identity map, each $\psi_s$ is a diffeomorphism for $|s|$ small, $\psi_s|_{\Sigma} = f_s$, and $\psi_s$ is the identity map outside a large compact set. 
 \end{lemma}
\begin{proof}
By convexity, we can express the points of $\Sigma\subset \mathbb{R}^3$ as $(r(\theta), \theta)$ in spherical coordinates. We can also write $f_s:\Sigma \to \mathbb{R}^3$ as $f_s(r(\theta), \theta) = \big(\hat{\rho}(\theta, s), \hat{\xi}(\theta, s)\big)$ in spherical coordinates where $\hat{\rho}, \hat{\xi}$  are smooth in $s$ and $\C^{2,\alpha}$ in~$\theta$. Since $f_0 = {\rm Id}_{\Sigma}$, we know $\hat{\xi}(\theta, 0)=\theta$. 
 
 Since $\{\Sigma_s\}$ is contained in a compact subset,  there is a positive number $a$ so that $\hat{\rho}(\theta, s)< a$ for all $(\theta, s)$. We define a family of maps  $\zeta_s: \mathbb{R}^3\setminus  B_a \to \mathbb{R}^3\setminus \Omega_s$ by
\[
	\zeta_s( r, \theta) = \big(\rho(r, \theta, s), \xi(r,\theta, s)\big)
\] 
such that $\rho(r, \theta, s)$ and $ \xi(r,\theta, s)$ are smooth in $s$ and $\C^{2,\alpha}$ in~$(r,\theta)$ satisfying the following properties:
\begin{itemize}
\item $\zeta_0$ is a diffeomorphism.
\item $\big(\rho(a, \theta, s), \xi(a, \theta, s) \big)= \big(\hat{\rho}(\theta, s), \hat{\xi}(\theta, s)\big)$.
\item For $r>2a$, $\big(\rho(r, \theta, s), \xi(r, \theta, s) \big) = (r,\theta)$ for all $s$. 
\end{itemize}
For example, we can define
\begin{align*}
\rho(r, \theta, s) &= \Big(\lambda(r) a^{-1} \hat{\rho}(\theta, s) + \big(1-\lambda(r)\big)\Big)r\\
\xi(r,\theta, s)&=\hat{\xi}(\theta, \lambda(r) s)
\end{align*}
where $\lambda(r)$ is a smooth, weakly decreasing function such that $\lambda(a)=1$ and $\lambda(r)\equiv 0$ for $r>2a$. 

It is direct to verify that  $\{ \zeta_s\}$ is smooth in $s$, each $\zeta_s$ is a diffeomorphism for $|s|$ sufficiently small,  $\left.\zeta_s\right|_{\partial B_a} = f_s$,  and~$\zeta_s$ equals to the identity map outside a compact set. Then we define $\psi_s = \zeta_s\circ \zeta_0^{-1}:\mathbb{R}^3\setminus \Omega\to \mathbb{R}^3\setminus \Omega_s$. 

\end{proof}

We now proceed to compute the first and second variations of $\mathscr{G}$. We'll omit the subscript~$\bar{g}$ when the geometric quantities and linearizations are taken with respect to $\bar{g}$. 

\begin{proposition}\label{proposition:variation2}
Let $g(s)$ be a smooth family of asymptotically flat metrics in $\mathbb{R}^3\setminus \Omega$ with $g(0)= \bar{g}$. Denote by $h = g'(0)$. Then we have the first and second variational formulas:
\begin{align*}
	\left.\ds\right|_{s=0} \mathscr{G}&(g(s))=0\\
	\left.\2ds\right|_{s=0} \mathscr{G}&(g(s))=\int_{\mathbb{R}^3\setminus \Omega} h\cdot\Big(-D\Ric(h)+\tfrac{1}{2} DR(h) \bar{g}\Big)\, d\mathrm{vol}\\
	& +\int_\Sigma \left[\Big( D\mathring{H}(h^{\intercal}) - DH(h)\Big)\tr h^\intercal + \Big(DA (h)- D \mathring{A}(h^\intercal)\Big)\cdot h^\intercal \right] \, d\sigma.
\end{align*}

\end{proposition}
\begin{proof}
We first compute the variations for one of the boundary integrals in the functional $\mathscr{G}$:
\begin{align*}
	&\left.\ds \right|_{s=0} \int_\Sigma 2H_{g(s)} \, d\sigma_{g(s)} =\int_\Sigma \Big(2DH(h) + H\tr h^\intercal\Big) \, d\sigma\\
	&\left.\2ds \right|_{s=0} \int_\Sigma 2H_{g(s)} \, d\sigma_{g(s)} \\
	&\quad =\int_\Sigma \left(\big(2H_{g(s)}\big)'' + \Big(\tr_{g(s)}g'(s)^\intercal\Big)'H + 2DH(h)\tr h^\intercal +  \tfrac{1}{2}  \big(\tr h^\intercal \big)^2 H \right)\, d\sigma.
\end{align*}
Here and below,  $'$ denotes the $s$-derivative, evaluated at $s=0$. Combining those with the variations of $\mathscr{F}$ in Corollary~\ref{corollary:variations}, we have
\begin{align}\label{equation:variation-G}
\begin{split}
	&\left.\ds \right|_{s=0} \left[\mathscr{F}(g(s), 1) - \int_\Sigma 2H_{g(s)} \, d\sigma_{g(s)} \right]= \int_\Sigma \big(A\cdot h^\intercal - H \tr h^\intercal\big) \, d\sigma\\
		&\left.\2ds\right|_{s=0} \left[\mathscr{F}(g(s), 1) - \int_\Sigma 2H_{g(s)} \, d\sigma_{g(s)} \right]\\
		& \quad = \int_{\mathbb{R}^3\setminus \Omega} h\cdot \Big(-D\Ric(h)+\tfrac{1}{2} DR(h) \bar{g}\Big)\, d\mathrm{vol}\\
		& \quad  +\int_{\Sigma}\Big( \big(DA(h)  - 2A \circ h^\intercal + \tfrac{1}{2} \tr h^\intercal A \big)\cdot h^\intercal +  \tr h^\intercal DH(h) \Big)\, d\sigma \\
		&\quad + \int_\Sigma \left(A\cdot \big(g^\intercal \big)'' - \Big(\tr_{g(s)}g'(s)^\intercal\Big)' H - 2DH(h) \tr h^\intercal- \tfrac{1}{2} \big(\tr h^\intercal\big)^2 H\right) \, d\sigma.
	\end{split}
\end{align}

For $|s|$ sufficiently small,  $(\Sigma, g(s)^\intercal)$ is convex. From Lemma~\ref{lemma:embedding},  there is a smooth family of isometric embeddings  $f_s:(\Sigma, g(s)^\intercal)\to (\mathbb{R}^3, \bar{g})$.  By Lemma~\ref{lemma:extension}, there exists  $\psi_s: \mathbb{R}^3\setminus  \Omega\to \mathbb{R}^3$, with $\psi_0$ being the identity map,  that is diffeomorphic onto its image such that for each $s$,  $\psi_s|_\Sigma = f_s$ and $\psi_s$ is the identity map outside a compact subset. Define $\xi(s) = (\psi_s)^*(\bar{g})$ on $\mathbb{R}^3\setminus \Omega$, as the pull-back of the Euclidean metric, with $\xi(0)=\bar{g}$. By construction,  $\xi(s)$ is isometric to the Euclidean metric and induces the same metric as $g(s)$ on $\Sigma$:
 \begin{align*}
 	\xi(s)^\intercal&= g(s)^\intercal.  
\end{align*}
Furthermore, the second fundamental form $A_{\xi(s)}$ of $\Sigma$ in $(\mathbb{R}^3, \xi(s))$ is just $\mathring{A}_{\xi(s)^\intercal}$ since $\xi(s)$ is itself a Euclidean metric. Using that $\xi(s)^\intercal= g(s)^\intercal$, we have 
\begin{align*}
	A_{\xi(s)}= \mathring{A}_{g(s)^\intercal},\quad \mbox{and}\quad  \ H_{\xi(s)} = \mathring{H}_{g(s)^\intercal}.
\end{align*} 
By differentiating the above identities in $s$, we get
\begin{align}\label{equation:g-functional}
\begin{split}
	\big(\xi(s)^\intercal\big)' &= h^\intercal\\
	\big(\xi(s)^\intercal\big)''&= \big(g(s)^\intercal\big)''\\
	\big(A_{\xi(s)}\big)' &= D\mathring{A} (h^\intercal)\\
	\big(H_{\xi(s)}\big)' &=D\mathring{H}(h^\intercal). 
\end{split}
\end{align}

 Apply the computations of \eqref{equation:variation-G} to the family of asymptotically flat pairs $(\xi(s), 1)$, and substitute the terms by \eqref{equation:g-functional}. Because $\mathscr{F}(\xi(s), 1)\equiv 0$, $D\Ric(\xi'(0))=0$, $DR(\xi'(0))=0$,  we derive
\begin{align*}
	&\left.\ds \right|_{s=0}  \int_\Sigma 2\mathring{H}_{\xi(s)^\intercal} \, d\sigma_{g(s)} = \int_\Sigma (-A\cdot h^\intercal+ H \tr h^\intercal) \, d\sigma\\
		&\left.\2ds\right|_{s=0}  \int_\Sigma 2\mathring{H}_{\xi(s)^\intercal}\, d\sigma_{g(s)}\\
		&\quad=-  \int_{\Sigma}\Big( \big(D\mathring{A}(h^\intercal)  -2A \circ h^\intercal+ \tfrac{1}{2} \tr h^\intercal A \big)\cdot h^\intercal +\tr h^\intercal D\mathring{H}\big(h^\intercal \big)  \Big)\, d\sigma  \\
		&\quad  + \int_\Sigma \left(- A\cdot \big(g^\intercal \big)'' + \Big(\tr_{g(s)}g'(s)^\intercal\Big)' H + 2D\mathring{H}(h^\intercal) \tr h^\intercal+\tfrac{1}{2} \big(\tr h^\intercal\big)^2 H\right) \, d\sigma.
\end{align*}
The desired formulas follow from adding the above identities to \eqref{equation:variation-G} and then simplifying. 
\end{proof}

We now proceed to prove Proposition~\ref{proposition:harmonic-function}. We recall its statement.

\begin{manualproposition}{\ref{proposition:harmonic-function}}
Let  $\Omega$ be a bounded open subset in $\mathbb{R}^3$ whose boundary $\Sigma = \partial \overline{\Omega}$ has  positive Gauss curvature. Let $v\in\C^{2,\alpha}_{-q}(\mathbb R^3\setminus\Omega)$ satisfy $\Delta v=0$. Then
\[
\int_\Sigma 2v\Big(D \mathring{H}\big(2v\bar{g}^\intercal\big) - DH(2v\bar{g}) \Big) \, d\sigma\geq 0.
\]
\end{manualproposition}
\begin{proof}
Let $\gamma(s)$ be the family of metrics in $\mathbb{R}^3\setminus \Omega$ constructed from   the harmonic function $v$ as in Lemma~\ref{lemma:harmonic-generate}, and let $\gamma_{\I}$ be the corresponding  metrics in the compact region $\overline{\Omega}$.  Recall that $\gamma'(0) = 2v\bar{g}$. We will show that the functional~$\mathscr{G}$ is ``concave'' at $\bar{g}$ along $\gamma(s)$ in the following sense:
\begin{align}\label{equation:concave}
	\left.\2ds\right|_{s=0} \mathscr{G}(\gamma(s)) \ge 0.
\end{align}
The functional $\mathscr{G}$ along $\gamma(s)$ can be re-written as 
\begin{align*}
	\mathscr{G}(\gamma(s))   &=-16\pi m_{\mathrm{ADM}}(\gamma(s))+ \int_{\mathbb{R}^3\setminus \Omega} R_{\gamma(s)}\, d\mathrm{vol}_{\gamma(s)}+  \int_\Sigma 2\Big(H_{\gamma_{\I}(s)} - H_{\gamma(s)}\Big)\, d\sigma_{\gamma(s)}\\
	& \quad + \int_\Sigma 2\Big(\mathring{H}_{\gamma(s)^\intercal} - H_{\gamma_{\I}(s)}\Big)\, d\sigma_{\gamma(s)}.
\end{align*}
Since  $\left.\2ds\right|_{s=0} $ of first line in the right hand side is nonnegative by Lemma~\ref{lemma:harmonic-generate}, we just need to compute $\left.\2ds\right|_{s=0}  \int_\Sigma 2\Big(\mathring{H}_{\gamma(s)^\intercal} - H_{\gamma_{\I}(s)}\Big)\, d\sigma_{\gamma(s)}$.  Using that $(\overline{\Omega}, \gamma_{\I}(s))$ has zero scalar curvature and $\gamma_{\I}(s)^\intercal = \gamma(s)^\intercal$ on $\Sigma$, we can apply Shi-Tam's theorem, Theorem~\ref{theorem:Shi-Tam}, to see that 
\[
\left.\2ds\right|_{s=0}\int_\Sigma 2\Big(\mathring{H}_{\gamma(s)^\intercal} - H_{\gamma_{\I}(s)}\Big)\, d\sigma_{\gamma(s)}\ge 0
\]
as the integral achieves the global minimum at $s=0$. Thus we prove~\eqref{equation:concave}.

From Proposition~\ref{proposition:variation2},  we have the second variation formula \begin{align*}
	\left.\2ds\right|_{s=0} \mathscr{G}(\gamma(s))&=\int_{\mathbb{R}^3\setminus \Omega} h\cdot\Big(-D\Ric(h)+\tfrac{1}{2} DR(h) \bar{g}\Big)\, d\mathrm{vol}\\
	&\quad +\int_\Sigma \left[\Big( D\mathring{H}(h^{\intercal}) - DH(h)\Big)\tr h^\intercal + \Big(DA (h)- D \mathring{A}(h^\intercal)\Big)\cdot h^\intercal \right] \, d\sigma
\end{align*}
where $h=\gamma'(0) = 2v\bar{g}$. Since $v$ is harmonic, we have $\bar{g}\cdot D\Ric (2vg) = 0$ and $DR(2vg)=0$ by the linearized curvature formulas in Lemma~\ref{lemma:formula}. So the interior integral above is zero. For the boundary integral, obviously 
$\Big( D\mathring{H}(h^{\intercal}) - DH(h)\Big)\tr h^\intercal =4v\Big( D\mathring{ H} (2v\bar{g}^\intercal)- DH\big(2v\bar{g}\big)\Big)$  and 
\begin{align*}
	& \Big(D A(2v\bar{g}) - D\mathring{A}\big(2v\bar{g}^\intercal \big) \Big) \cdot h^\intercal \\
	 &= 2v \left. \ds\right|_{s=0} \Big(\big(A_{\gamma(s)} -\mathring{A}_{\gamma(s)^\intercal}\big) \cdot \gamma(s)^\intercal\Big)+ 2v \big(A_{\bar{g}} -\mathring{A}_{\bar{g}^\intercal}\big) \cdot \left. \ds\right|_{s=0} \gamma(s)^\intercal \\ 
	 &= 2v \left. \ds\right|_{s=0} \big(H_{\gamma(s)} -\mathring{H}_{\gamma(s)^\intercal}\big) \\
	 &= 2v \Big( D H (2v\bar{g})- D\mathring{H}\big(2v\bar{g}^\intercal\big)\Big).
\end{align*}
Therefore,
\[	
	\left.\p2s\right|_{s=0} \mathscr{G}(\gamma(s))=\int_\Sigma 2v\Big(D \mathring{H}\big(2v\bar{g}^\intercal\big) - DH(2v\bar{g}) \Big) \, d\sigma.
\]
Combining this with \eqref{equation:concave} gives the proposition. 
\end{proof}

\subsection{Rigidity of Ricci flat deformations}

In order to apply Proposition~\ref{proposition:harmonic-function} to prove Theorem~\ref{theorem:static-convex}, we will first derive  effective formulas to compute  $D\mathring H (h^\intercal)$.

Observe that for deformations that take the form $L_X\bar g$ (those ``generated from diffeomorphisms''), 
\begin{align}\label{equation:embedding}
	D\mathring{H}\big((L_X\bar{g})^\intercal\big) = DH(L_X\bar{g}). 
\end{align}
To see this, we let $\psi_t: \mathbb{R}^3\setminus \Omega\to \mathbb{R}^3$ be be the flow of $X$. Since $\psi_t^*\bar{g}$ is a Euclidean metric on $\mathbb{R}^3\setminus \Omega$, by definition of $\mathring{H}$, we have $\mathring{H}_{(\psi_t^*\bar{g})^\intercal} =H_{\psi_t^*\bar{g}}$. Differentiating in $t$ gives the above formula.

We will see in the next lemma that because of the convexity of $\Sigma$, any Ricci flat deformation $h$ must be generated from diffeomorphisms, i.e. $h=L_X\bar g$, and any tensor $\tau$ on $\Sigma$ is induced from a ``unique'' Ricci flat deformation $h$ such that $h^\intercal = \tau$. To state the lemma more precisely, we define the Banach spaces $\mathcal{S}_0$ and $\mathcal{S}_1$  that consist of symmetric (0,2)-tensors in $\mathbb{R}^n\setminus \Omega$:
 \begin{align*}
 \mathcal{S}_0&=\{ L_Z \bar{g}: \mbox{$Z\in \C^{3,\alpha}_{-q}(\mathbb{R}^3\setminus \Omega)$  satisfies $\Delta Z=0$ in $\mathbb{R}^3\setminus \Omega$}\}\\
  \mathcal{S}_1&=\{ h\in \C^{2,\alpha}_{-q}(\mathbb{R}^3\setminus \Omega): D\Ric(h)=0, \bi h=0 \mbox{ in } \mathbb{R}^3\setminus \Omega\}/\{L_X \bar{g}: X\in \mathcal{X}_0 \},
 \end{align*}
 where  the equivalence relation in $\mathcal{S}_1$ says that $h_1\sim h_2$ if and only if $h_1-h_2 = L_X \bar{g}$ for some  $X \in \mathcal X_0$. (Recall the definition of $\mathcal{X}_0$ in Definition~\ref{definition:harmonic}.) Geometrically, $\mathcal{S}_0$ consists of deformations generated from diffeomorphisms  satisfying the harmonic gauge, and  $\mathcal{S}_1$ consists of Ricci flat deformations satisfying the harmonic gauge (up to the equivalent relation).  We also define the Banach space of tensors on the boundary $\Sigma$:
  \begin{align*}
 \S_2=\{ \tau\in \C^{2,\alpha}(\Sigma): \tau \mbox{ is a symmetric $(0,2)$-tensor on } \Sigma\}.
 \end{align*}

 \begin{lemma}\label{lemma:convex}
Let  $\Sigma$ be convex. Then $\mathcal{S}_0$ is surjective onto both $\mathcal{S}_1$ and $\mathcal{S}_2(\Sigma)$, and $\mathcal{S}_1,\mathcal{S}_2(\Sigma)$ are isomorphic, via the maps given by
 \begin{align*}	
                         &b:L_Z\bar{g} \in \mathcal{S}_0\longrightarrow  [L_Z \bar{g}]\in \mathcal{S}_1\\
 	 		&b_0:L_Z \bar{g}\in  \mathcal{S}_0 \longrightarrow  (L_Z\bar{g})^\intercal\in \S_2\\
			&b_1:[h] \in \mathcal{S}_1 \longrightarrow  h^\intercal \in \S_2
 	 \end{align*}
 \end{lemma}
 \begin{remark}
 We note that the statement about surjectivity of $ b:\mathcal{S}_0\to \mathcal{S}_1$ can be  compared to Corollary~\ref{corollary:trivial}. It says that the Cauchy boundary condition in Corollary~\ref{corollary:trivial} can be dropped if $\Sigma$ is convex.
 \end{remark}
 \begin{proof}
 We will show  that $b_0$
 is surjective and $b_1$
 is injective.
Since $b_0=b_1\circ b$, it gives the desired properties. 
 
The map $b_0$ 
is surjective: Given $\tau \in \S_2$, we consider the family of metrics $ \bar{g}^\intercal + t\tau$ on $\Sigma$. Since $(\Sigma, \bar{g}^\intercal)$ has positive Gauss curvature, $ \bar{g}^\intercal + t\tau$ also has positive Gauss curvature for $|t|$  small. So there is a smooth family of isometric embeddings $f_t: (\Sigma, \bar{g}^\intercal + t\tau) \to (\mathbb{R}^3, \bar{g})$ by Lemma~\ref{lemma:embedding}. We identify $\Sigma$ with $f_0(\Sigma)$. Let $V: =\left. \frac{\partial }{\partial t}\right|_{t=0} f_t$ be the velocity vector field defined along $\Sigma$, and we have $V\in \C^{2,\alpha}(\Sigma)$ initially. Since  $(L_V \bar{g})^\intercal=\tau \in \C^{2,\alpha}(\Sigma)$ on $\Sigma$, we can argue that in fact $V\in \C^{3,\alpha}(\Sigma)$.  Solving $Z\in \C^{3,\alpha}_{-q}(\mathbb{R}^3 \setminus \Omega)$ to the Dirichlet boundary value problem 
 \begin{align*}
 	\Delta Z&=0 \mbox{ in } \mathbb{R}^3 \setminus \Omega\\
	Z &= V \mbox{ on } \Sigma,
 \end{align*}
we find $[L_Z\bar{g}]\in \mathcal{S}_0$ so that $(L_Z\bar{g})^\intercal = \tau$ on $\Sigma$.

The map $b_1$
is injective: If $h$ is a representative of $[h] \in \mathcal{S}_1$ and $h^\intercal=0$, then we will show that $h=L_X \bar{g}$ for some $X\in \mathcal{X}_0$, and thus $[h]=0$. For any $\tau \in \S_2$, we have shown that there is $L_Z \bar{g} \in \mathcal{S}_0$ such that $(L_Z \bar{g})^\intercal = \tau$. Since both $h$ and  $L_Z\bar{g}$ are Ricci flat deformations, we can apply \eqref{equation:Ricci-flat} (with $k=L_Z\bar{g}$ and hence $k^\intercal = \tau$)  to get
 \[
 	0=\int_\Sigma \Big(DA(h)\cdot \tau-  (\tr \tau) DH(h)\Big)\, d\sigma=\int_\Sigma \big(DA(h)-DH(h){\bar g}^\intercal\big)\cdot\tau \, d\sigma.
 \]
Since $\tau$ is arbitrary, $DA (h)- DH(h) \bar{g}^{\intercal}$ must vanish identically on $\Sigma$. Using the assumption that $h^\intercal=0$, we see that $DH(h)$ and hence $DA(h)=0$ both vanish on $\Sigma$. Then Corollary~\ref{corollary:trivial} implies that $h=L_X \bar{g}$ for some $X\in \mathcal{X}_0$. 
 \end{proof}
\begin{remark}
While we will not use it elsewhere in the paper, we note the following fact.  The map~$b:\mathcal{S}_0\to \mathcal{S}_1$ has an $N$-dimensional kernel space  
\[
\{ L_Z\bar{g}\in \mathcal{S}_0: Z= K \mbox{ on $\Sigma$ for some Killing vector $K$} \}.
\]
 Thus, the  maps $b$ and $b_0$  induce isomorphisms on the quotient space, $\mathcal{S}_0$ modulo the kernel space, to  $\mathcal{S}_1$ and $\mathcal{S}_2(\Sigma)$, respectively.
\end{remark}

Together with an earlier observation \eqref{equation:embedding}, Lemma~\ref{lemma:convex} gives the following effective formulas for $D\mathring{H}$.

\begin{corollary}\label{corollary:convex}
Let $\Sigma$ be convex and  $h\in \C^{2,\alpha}_{-q}(\mathbb{R}^3\setminus \Omega)$ be a symmetric $(0,2)$-tensor. Then the following holds:
\begin{enumerate}
\item  $h^\intercal = (L_X \bar{g})^\intercal$ on $\Sigma$ for some $X\in \C^{3,\alpha}_{-q}(\mathbb{R}^3 \setminus \Omega)$ and
\[
	D\mathring{H}(h^\intercal) = DH(L_X\bar{g})\quad  \mbox{ on }\Sigma.
\]
\item If $h$ solves $D\Ric(h)=0$ in $\mathbb{R}^3\setminus \Omega$, then $h = L_X \bar{g}$ in $\mathbb{R}^3\setminus \Omega$ for some $X\in \C^{3,\alpha}_{1-q}(\mathbb{R}^3 \setminus \Omega)$ and
\[
	D\mathring{H}(h^\intercal) = DH(h) \quad \mbox{ on } \Sigma.
\]
\end{enumerate}
 \end{corollary}
 \begin{proof}
 The first statement is an immediate consequence from Lemma~\ref{lemma:convex} and \eqref{equation:embedding}. 

For the second statement,  by Lemma~\ref{lemma:harmonic-gauge}, there is a vector field $V\in \C^{3,\alpha}_{1-q}(\mathbb{R}^3\setminus \Omega)$ with $V=0$ on $\Sigma$ such that $\bi (h+L_V \bar{g})=0$ in $\mathbb{R}^3\setminus \Omega$, and thus $[h+L_V\bar{g}] \in \mathcal{S}_1$.  Since $\mathcal{S}_0$ is surjective onto $\mathcal{S}_1$ as shown in Lemma~\ref{lemma:convex}, there is  $Z\in \C^{3,\alpha}_{-q}(\mathbb{R}^3\setminus \Omega)$ such that 
\[
	[h+L_{V}\bar{g}]= [L_Z\bar{g}]. 
\]
It is straightforward to verify that $h=L_X \bar{g}$ for some $X= O^{3,\alpha}(|x|^{1-q})$. The identity for the linearized mean curvature follows \eqref{equation:embedding}.
\end{proof}

\begin{proof}[Proof of Theorem~\ref{theorem:static-convex}]
Suppose $(h, v)$ is a static vacuum deformation satisfying $h^\intercal=0, DH(h)=0$ on $\Sigma$.  We will show that $DA(h)=0$ on $\Sigma$. 

If $v$ is identically zero, then $h$ is a Ricci flat deformation.  We may assume that $h$ satisfies the harmonic gauge $\bi h=0$ by Lemma~\ref{lemma:harmonic-gauge}. Applying Lemma~\ref{lemma:convex} that $\mathcal{S}_1\to \mathcal{S}_2(\Sigma)$ is isomorphic, we have $[h]=0$ and thus $h = L_X\bar{g}$ for some $X\in \mathcal{X}_0$. In particular, $DA(h)=0$ on $\Sigma$.  

From now on, we assume that $v$ is not identically zero. 
By Lemma~\ref{lemma:geodesic}, we may assume that $h$ satisfies the geodesic gauge on $\Sigma$, and therefore, by \eqref{equation:normal}, \eqref{equation:linearized-sff},  and \eqref{equation:Riccati}, we have
\begin{align}\label{equation:linearization-geodesic}
\begin{split}
	\nu\big(DH(h) \big)&= -A \cdot DA(h)\\
	(L_\nu h)^\intercal &= 2DA(h).
\end{split}
\end{align}

Let $\nu$ be the outward unit normal to $\Sigma$ parallelly extended in a tubular neighborhood of $\Sigma$.  Let $Y$ be a smooth, compactly supported vector field on $\mathbb{R}^3$ such that $Y = \nu$ in a tubular neighborhood of $\Sigma$. Let $\psi_t$ be the flow of $Y$ for $ t\in (-\epsilon, \epsilon)$ and we denote by $\Sigma_t = \psi_t(\Sigma)$  and $\Omega_t = \psi_t(\Omega)$. Let $v_t\in \C^{2,\alpha}_{-q}(\mathbb{R}^3\setminus \Omega_t)$ be a $\C^1$-family of harmonic functions such that the Dirichlet boundary value $v_t|_{\Sigma_t} = v|_{\Sigma_t}$ for $t\ge 0$. In particular, $v_t = v|_{\mathbb{R}^3\setminus \Omega_t}$ for $t\ge 0$. We define the  scalar-valued  $\C^1$-function:
\begin{align}\label{equation:H1}
	\mathcal{H}(t)=\int_{\Sigma_t} \left[2v_t\Big(D\mathring{H}\big(2v_t\bar{g}^\intercal\big) - D H(2v_t\bar{g})\Big)\right] \, d\sigma.
\end{align}
Applying Proposition~\ref{proposition:harmonic-function} on each $\Sigma_t$,  we have  $\mathcal{H}(t)\ge 0$ for $t\in (-\epsilon, \epsilon)$. Since $(h, v)$ is a static vacuum deformation and thus $h+2v\bar{g}$ is a Ricci flat deformation in $\mathbb{R}^3\setminus\Omega$, we can apply Corollary~\ref{corollary:convex} to get, for each $t\ge 0$, 
\[
	 D\mathring{H}\big( (h+2v\bar{g})^\intercal \big)=DH(h+2v\bar{g})\quad \mbox{ on } \Sigma_t,
\]
and thus
\begin{align*}
	  D\mathring{H} \big(2v\bar{g}^\intercal\big)- D H(2v\bar{g}) =   DH(h) -  D\mathring{H} (h^\intercal)\quad \mbox{ on } \Sigma_t.
\end{align*}
This implies that the function $\mathcal{H}(t)$ has an alternative expression for $t\ge 0$:
\begin{align}\label{equation:H2}
\begin{split}
	\mathcal{H}(t) &= \int_{\Sigma_t} 2v \Big( DH(h) -  D\mathring{H} (h^\intercal)\Big) \, d\sigma\\
	&=\int_\Sigma 2\psi_t^* \bigg(v   \Big( DH(h) -  D\mathring{H} (h^\intercal)\Big)\bigg) \, d\sigma_{\psi_t^*(\bar{g})}.
\end{split}
\end{align}
Note that $\psi_t^* \Big(v   \big( DH(h) -  D\mathring{H} (h^\intercal)\big)\Big) =(\psi_t^* v )\Big( \psi_t^* DH(h) - \psi_t^* D\mathring{H}(h^\intercal)\Big)$ on $\Sigma$, where $\psi_t^*DH(h)$ on $\Sigma$ is the pull-back of the scalar-valued function $DH|_{\bar{g}}(h)$ on $\Sigma_t\subset \mathbb{R}^n\setminus \Omega_t$, and similarly for $\psi^*_tD\mathring{H}(h^\intercal)$. (See also Lemma~\ref{lemma:pullback}.)  From the alternative expression~\eqref{equation:H2}, we see that $\mathcal{H}(0)=0$ because of the assumption $h^\intercal=0, DH(h)=0$ on $\Sigma$. Thus, $t=0$ is a critical point $\mathcal{H}'(0)=0$. 

We \emph{claim} 
\begin{align*}
	\mathcal{H}'(0) &= 2\int_\Sigma |DA(h)|^2\, d\sigma.
\end{align*}
Once  the claim is verified, we can combine with the fact  $\mathcal{H}'(0)=0$ to conclude that $DA(h)$ vanishes on $\Sigma$, which completes the proof.

We compute $\mathcal{H}'(0)$. Differentiating \eqref{equation:H2} in $t$  for $t\ge 0$ and using that $DH (h) - D\mathring{H} (h^\intercal)=0$ on $\Sigma$, we get 
\begin{align*}
	\mathcal{H}'(0)&=\int_\Sigma 	2v \left.\pt\right|_{t=0}\psi_t^* \left( DH (h) - D\mathring{H} (h^\intercal)\right) \, d\sigma.
\end{align*}
Recall that $\left.\pt\right|_{t=0} \psi_t = Y = \nu$ on $\Sigma$.  By  \eqref{equation:linearization-geodesic}, 
\begin{align*}
	\left.\pt\right|_{t=0}\psi_t^* ( DH (h) )=\nu(DH(h))=  -A \cdot DA(h).
\end{align*}
To compute the term involving $D\mathring{H}$, we apply Lemma~\ref{lemma:convex}. For each $t$, there exists a vector field $X_t$ with $h^\intercal = (L_{X_t} \bar{g})^\intercal$ on $\Sigma_t$ such that 
\[	
	 D\mathring{H}\big(h^\intercal\big) = DH( L_{X_t} \bar{g}) \quad \mbox{ on } \Sigma_t.
\]
Then Lemma~\ref{lemma:pullback} implies that
\begin{align*}
 	\psi_t^* \left(D\mathring{H}\big(h^\intercal\big) \right)&=\psi_t^* \left( DH( L_{X_t} \bar{g})\right)\\
	&=DH|_{\psi_t^*\bar{g}} \big(\psi_t^* L_{X_t} \bar{g}\big)\\
	&=D\mathring{H}|_{\psi_t^*\bar{g}^\intercal} \big(\psi_t^* (L_{X_t} \bar{g} )^\intercal \big)\\
	&= D\mathring{H}|_{\psi_t^*\bar{g}^\intercal } \big(\psi_t^* h^\intercal \big) \quad \mbox{ on } \Sigma,
\end{align*}
where in the third equation we use that $\psi_t^* L_{X_t} \bar{g}  = L_{\psi_t^* X_t} (\psi_t^* \bar{g})$ is Ricci flat deformation  at $\psi_t^* \bar{g}$ and Corollary~\ref{corollary:convex}. Differentiating the previous identity in $t$ gives
\begin{align*}
	\left. \pt \right|_{t=0} \psi_{t}^* \left( D\mathring{H}\big(h^\intercal\big)\right) &=\left.\pt\right|_{t=0} D\mathring{H}|_{\psi_{t}^* \bar{g}^\intercal } \big(\psi_{t}^* h^\intercal \big) \\
	&=D^2 \mathring{H}\left((L_Y \bar{g})^\intercal, h^\intercal \right) + D\mathring{H} \big((L_\nu h)^\intercal \big)\\
	&= D\mathring{H} \big(2DA(h)\big)
\end{align*}
where we use that $h^\intercal=0$ on $\Sigma$   and \eqref{equation:linearization-geodesic}. By the first statement in Corollary~\ref{corollary:convex}, there is a vector $Z$ such that $ 2DA(h)=(L_Z\bar{g})^\intercal $ and  $D\mathring{H} \big(2DA(h)\big) = DH(L_Z\bar{g})$ on $\Sigma$. Apply \eqref{equation:Green-boundary} for static vacuum deformations $(h, v)$ and $(L_Z\bar{g}, 0)$. Since $h^\intercal=0, DH(h)=0$ on $\Sigma$ by assumption, we get
\begin{align*}
	0&=	\int_\Sigma  \Big\langle \big(vA + DA(h) -\nu(v) \bar{g}^\intercal, 2v \big), \big((L_Z\bar{g})^\intercal, DH(L_Z\bar{g})\big) \Big\rangle \, \da\\
	&=\int_\Sigma  \Big\langle \big(vA + DA(h) -\nu(v) \bar{g}^\intercal, 2v \big), \big(2DA(h), D\mathring{H} (2DA(h))\big)\Big\rangle \, \da.
\end{align*}	
Rearranging the terms yields
\[
	-\int_\Sigma 2v D\mathring{H}\big(2DA(h)\big) = \int_\Sigma \Big(2DA(h) \cdot \big(vA+ DA (h) - \nu(v) \bar{g}^\intercal \big)\Big)\, d\sigma.
\]

Combining the previous computations gives
\begin{align*}
		\mathcal{H}'(0)&=\int_\Sigma 	2v \left.\pt\right|_{t=0}\psi_t^* \left( DH (h) - D\mathring{H} (h^\intercal)\right) \, d\sigma\\
		&=\int_\Sigma \Big(-2v A\cdot DA(h) +2   DA(h)\cdot \big(vA+ DA (h) - \nu(v) \bar{g}^\intercal \big) \Big)\, d\sigma\\
		&=\int_\Sigma 2|DA(h)|^2\, d\sigma
\end{align*}
where we have used that $ DA(h)\cdot \bar{g}^\intercal  =DH(h) - A\cdot  h^\intercal=0$ on $\Sigma$. It completes the proof to the claim.  

\end{proof}

\section{Perturbations of hypersurfaces in $\mathbb{R}^n$}\label{section:generic}

We prove Theorem~\ref{theorem:static-generic} and Corollary~\ref{cor:dilation} in this section. Let us recall the statement in which we also spell out Definition~\ref{definition:static-regular}. 

\begin{manualtheorem}{\ref{theorem:static-generic}}
Let $t\in [-\delta, \delta]$ and each $\Omega_t\subset \mathbb{R}^n$ be a bounded open subset with embedded hypersurface boundary $\Sigma_t = \partial\overline{ \Omega_t}$.  Suppose  the boundaries $\{ \Sigma_t\}$ form a smooth generalized foliation. Then  there is an open dense subset $J\subset (-\delta,\delta)$ such that for any $t\in J$,  $\Sigma_t$ is static regular in $\mathbb{R}^n\setminus \Omega_t$. Namely, if $(h, v)\in \C^{2,\alpha}_{-q}(\mathbb{R}^n\setminus \Omega_t)$ is a static vacuum deformation at $\bar{g}$ such that $h$ satisfies the Bartnik boundary condition on $\Sigma_t$, then $h$ satisfies the Cauchy boundary condition on $\Sigma_t$.
\end{manualtheorem}

Let  $X$ be the deformation vector of $\{ \Sigma_t\}$. On each $\Sigma_t$, we can write $X= \zeta \nu + X^\intercal$, where  the scalar-valued function $\zeta>0$, except possibly zero on a set of $(n-1)$-dimensional Hausdorff measure zero, $\nu$ is the unit normal to $\Sigma_t$ pointing way from $\Omega_t$, and $X^\intercal$ is tangential.  We smoothly extend $X$ as a smooth, compactly supported vector field in $\mathbb{R}^n$ (still denoted by $X$). Let $\psi_t:\mathbb{R}^n\to \mathbb{R}^n$ be the flow of $X$. If we denote by $\Omega:=\Omega_0$ and $\Sigma:= \Sigma_0$, then $\Omega_t = \psi_t(\Omega)$ and $\Sigma_t = \psi_t (\Sigma)$.  Denote by $g_t = \psi_t^* (\bar{g}|_{\mathbb{R}^n\setminus \Omega_t})$,   the pull-back metric defined on $\mathbb{R}^n\setminus \Omega$, and note $g_0 = \bar{g}$. 
Define the family of maps $S_t$ as, for $(h, v)\in \C^{2,\alpha}_{-q}(\mathbb{R}^n\setminus \Omega)$,
\begin{align}\label{equation:generic}
 S_t(h,v)=
\begin{array}{l}
	\left\{ \begin{array}{l}-D\Ric|_{g_t}(h)+ \nabla^2_{g_t} v  \\
	\Delta_{g_t} v	\end{array} \right. \quad \mbox{ in } \mathbb{R}^n\setminus \Omega \\
	\left\{ \begin{array}{l} 
         h^\intercal\\
	DH|_{g_t}(h)
	\end{array} \right. \quad \mbox{ on } \Sigma.
\end{array}
\end{align}
To prove that $\Sigma_t$ is static regular, we will need  to show that for any $(h, v)\in \C^{2,\alpha}_{-q}(\mathbb{R}^n\setminus \Omega)$ solving $S_t(h,v)=0$,
 we must have $DA|_{g_t}(h)=0$ on $\Sigma$.

\subsection{Differentiation of kernel elements along perturbations}

We first prove the following theorem, which implies Theorem~\ref{theorem:static-generic} under an extra assumption~$(\star)$. Its proof provides an outline for the proof of the general theorem. In order  to remove that assumption, we will establish basic facts about the kernel spaces of elliptic PDE systems in Section~\ref{section:technical} and choose the open dense subset $J\subset (-\delta, \delta)$. In Proposition~\ref{proposition:kernel},  we will show that  a ``discrete version" of the assumption $(\star)$ holds for $t\in J$ and then give a complete proof  of Theorem~\ref{theorem:static-generic}.

\begin{manualtheorem}{\ref{theorem:static-generic}$^\prime$}\label{theorem:special}
Let $t\in [-\delta, \delta]$, $\Omega_t$, and $\Sigma_t$ be as in Theorem~\ref{theorem:static-generic}. Let $(h, v)\in \C^{2,\alpha}_{-q}(\mathbb{R}^n\setminus \Omega)$ solve $S_a(h,v)=0$ for some $a\in (-\delta,\delta)$.   Suppose the following holds:
\begin{align} \tag{$\star$}
\begin{split}
&\mbox{There exists  $(h(t), v(t))\in \C^{2,\alpha}_{-q}(\mathbb{R}^n\setminus \Omega)$ with $S_t(h(t), v(t))=0$ for all $t$}\\
&\mbox{such that $(h(t), v(t)) \to (h, v)$ in $\C^{2,\alpha}_{-q}(\mathbb{R}^n\setminus \Omega)$  as $t\searrow  a$,}\\
&\mbox{and $(h(t), v(t))$ is differentiable at $t=a$ }\\
&\mbox{with $(h'(a), v'(a)) = (p, z)\in \C^{2,\alpha}_{-q}(\mathbb{R}^n\setminus \Omega)$}. 
\end{split}
\end{align}
Then $DA|_{g_a}(h)=0$ on $\Sigma$. 
\end{manualtheorem}

\begin{proof}
By re-parametrizing $t$, we may assume that $a=0$. Recall that $g_0= \bar{g}$ in our notation. Also recall that we omit the subscript $\bar{g}$ in linearizations and geometric operators. By Lemma~\ref{lemma:geodesic}, we may also assume that $h$ satisfies the geodesic gauge on $\Sigma$, and thus by \eqref{equation:normal}, \eqref{equation:linearized-sff},   \eqref{equation:Riccati}, and $h^\intercal =0$ on $\Sigma$, we have
\begin{align}\label{equation:linearization-geodesic2}
\begin{split}
	\begin{array}{rl} 
	(L_{\zeta\nu} h)^\intercal &= \zeta(L_{\nu} h)^\intercal =2\zeta DA(h)\\
	\nu\big(DH(h) \big)&= -A \cdot DA(h)
	\end{array} \quad \mbox{ on } \Sigma.
\end{split}
\end{align}

We \emph{claim} that $(p-L_X h, z-L_X v)$ is a static vacuum deformation in $\mathbb{R}^n\setminus \Omega$ whose Bartnik boundary data on $\Sigma$ satisfy
\begin{subequations}
\begin{align}
	(p-L_X h\big)^\intercal &= -2\zeta DA(h)\label{equation:boundary1}\\
	DH(p-L_X h)&=\zeta A\cdot DA(h). \label{equation:boundary2}
\end{align}
\end{subequations}

Let  $(h(t), v(t))$, $t>0$, be the family from the assumption~$(\star)$. Then
\begin{align*}
	\left\{ \begin{array}{rl}-\left.D\Ric\right|_{g_t}(h(t)) + \nabla^2_{g_t} v(t)&=0\\
	\Delta_{g_t} v(t)&=0\end{array}\right. \quad \mbox{ in } \mathbb{R}^n\setminus \Omega.
\end{align*}
Note the assumption $t>0$ implies that $\mathbb R^n\setminus\Omega_t\subset\mathbb R^n\setminus\Omega$, so the pull-back pair $(\psi_t^*h, \psi_t^*v)$ in $\mathbb{R}^n\setminus \Omega$ from  the restriction $(h, v)|_{\mathbb{R}^n\setminus \Omega_t}$ is trivially a static vacuum deformation at $g_t$:
\begin{align*}
	\left\{\begin{array}{rl}-\left.D\Ric\right|_{g_t}(\psi_t^*h) + \nabla^2_{g_t} (\psi_t^* v)&=\psi^*_{t} \left(-\left.D\Ric\right(h) + \nabla^2  v\right)=0\\
	\Delta_{g_t} (\psi_t^* v)&=\psi^*_{t} (\Delta v)=0\end{array}\right. \; \mbox{ in } \mathbb{R}^n\setminus \Omega.
\end{align*}
Subtracting the previous two systems yields
\begin{align*}
	\left\{\begin{array}{rl}-\left.D\Ric\right|_{g_t}\big(h(t)-\psi_t^*h\big) + \nabla^2_{g_t}\big (v(t)-\psi_t^* v\big)&=0\\
	\Delta_{g_t} \big(v(t)-\psi_t^* v\big)&=0\end{array}\right. \quad \mbox{ in } \mathbb{R}^n\setminus \Omega.
\end{align*}
Differentiating it in $t$ at $t=0$ and noting that $h(0)-\psi_0^* h=0, v(0)-\psi_0^*v=0$, we prove that $(p-L_X h, z-L_X v)$  is a static vacuum deformation at $\bar{g}$ in $\mathbb{R}^n\setminus \Omega$. Next, we compute the boundary data. By using $h(t)^\intercal =0$ on $\Sigma$ by \eqref{equation:generic} and differentiating in $t$ at $t=0$, we see $p^\intercal =0$ on  $\Sigma$. Since $X= \zeta \nu$ on $\Sigma$, we apply \eqref{equation:linearization-geodesic2}  to get \eqref{equation:boundary1}:
\[
	(p-L_X h)^\intercal = -(L_Xh)^\intercal = -2\zeta DA(h).
\]
To prove \eqref{equation:boundary2}, by the  boundary condition of \eqref{equation:generic} and Lemma~\ref{lemma:pullback}, we have, on $\Sigma$,
\begin{align*}
	DH|_{g_t}(h(t))=0,\quad
	DH|_{g_t}(\psi_t^* h) = \psi_t^* \big(DH(h)\big).
\end{align*}
Subtracting the previous identities gives 
\[
DH|_{g_t}\big(h(t)-\psi_t^* h\big) = -\psi_t^* \big(DH(h)\big) \mbox{ on }  \Sigma.
\]
Differentiating it in $t$ at $t=0$ gives
\begin{align*}
	DH(p-L_Xh) =- \zeta \nu \big(DH(h)\big) = \zeta A\cdot DA(h) \quad \mbox{ on } \Sigma
\end{align*}
where we use \eqref{equation:linearization-geodesic2} in the last identity. It completes the proof of the claim.

Lastly, we apply \eqref{equation:Green-boundary} (by substituting $(k, w)=(p-L_X h, z-L_X v)$) and use the boundary condition \eqref{equation:boundary1}, \eqref{equation:boundary2} to obtain
\begin{align*}
	0&= \int_\Sigma \Big\langle \big(vA + DA(h) -\nu (v) \bar{g}^\intercal, 2v\big), \big (k^\intercal, DH(k)\big) \Big\rangle \, d\sigma\\
	&=\int_\Sigma \Big\langle \big(vA + DA(h) -\nu (v) \bar{g}^\intercal, 2v\big), \big (-2\zeta DA(h), \zeta A\cdot DA(h)\big) \Big\rangle \, d\sigma\\
	&=-\int_\Sigma  2 \zeta |DA(h)|^2 \, d\sigma
\end{align*}
where we have used that $\bar{g}^\intercal \cdot DA(h) = DH(h) + h^\intercal \cdot A= 0$ on $\Sigma$ by the Bartnik boundary condition.  Since $\zeta>0$, except an $(n-1)$-dimensional measure zero set, we conclude that $DA(h)=0$ on $\Sigma$.
\end{proof}


\subsection{Weak differentiation in an open dense parameter set}\label{section:technical}

To remove the assumption $(\star)$, we will need to understand how solutions of $ S_t(h(t), v(t))=0$ behave  as $t\to a$. Instead of analyzing general solutions, it suffices to analyze only those solutions satisfying the static-harmonic gauge,  as in Section~\ref{section:existence}. 

To begin, we define $L_t:\C^{2,\alpha}_{-q} (\mathbb{R}^n\setminus \Omega) \to \C^{0,\alpha}_{-q-2} (\mathbb{R}^n\setminus \Omega) \times \C^{1,\alpha}(\Sigma)\times \mathcal{B}(\Sigma)$ by
\begin{align}
\label{equation:linear-t}
L_t (h, v)=\begin{array}{l}
	\left\{ \begin{array}{l}-D\Ric|_{g_t}(h)+ \nabla^2_{g_t} v - \D_{g_t}(\bi_{g_t} h+ dv) \\
	\Delta_{g_t} v	\end{array} \right. \quad \mbox{ in } \mathbb{R}^n\setminus \Omega \\
 	\left\{ \begin{array}{l} \bi_{g_t} h + dv\\
	h^\intercal\\
	DH|_{g_t}(h)
	\end{array} \right. \quad \mbox{ on } \Sigma
\end{array}
\end{align}
where the covariant derivatives and the linearizations are taken with respect to~$g_t$ and we recall the definitions of the operators $ \D_{g_t}, \bi_{g_t}$ in \eqref{equation:operators}. For ease of notation, we shall denote the codomain of $L_t$ by 
\[
	\mathcal{Z}=\C^{0,\alpha}_{-q-2} (\mathbb{R}^n\setminus \Omega) \times \C^{1,\alpha}(\Sigma)\times \mathcal{B}(\Sigma).
\]
The operator $L_t$, defined for $\mathbb{R}^n\setminus \Omega$, can be viewed as the pull-back of the operator $L$ as defined in \eqref{equation:linear} for  $\mathbb{R}^n\setminus \Omega_t$ via the diffeomorphism $\psi_t$.  In particular, since $L$ is elliptic and Fredholm by Lemma~\ref{lemma:Fredholm}, so is each $L_t$. 

 The following lemma, essentially a linearized version of Lemma~\ref{lemma:gauge}, says that the space of solutions to $S_t(h,v)=0$ is equivalent to that of $L_t(h, v)=0$, up to deformations generated from diffeomorphisms.
  \begin{lemma}\label{lemma:static-harmonic}
If $(h, v)\in \C^{2,\alpha}_{-q} (\mathbb{R}^n\setminus \Omega) $ solves $L_t(h, v)=0$, then $(h,v)$ satisfies $S_t(h,v)=0$. In fact, we have $S_t(h+L_V g_t, v)=0$ 
and $DA|_{g_t}(h+L_V g_t)=DA|_{g_t}(h)$ on $\Sigma$, 
for any $V\in \C^{3,\alpha}_{1-q}(\mathbb{R}^n\setminus \Omega)$ satisfying $V=0$ on $\Sigma$.

Conversely, if $(h, v)$ solves $S_t(h,v)=0$,
then there exists a vector field $V\in \C^{3,\alpha}_{1-q}(\mathbb{R}^n\setminus \Omega)$ with $V=0$ on~$\Sigma$ such that $L_t(h+L_V g_t, v)=0$  and  $DA|_{g_t}(h+L_V g_t)=DA|_{g_t}(h)$ on $\Sigma$. 
\end{lemma}
\begin{proof}
 Suppose $L_t(h,v)=0$. By the same argument as in the first paragraph in the proof of Lemma~\ref{lemma:kernel},  $\bi_{g_t} h + dv$ is identically zero in $\mathbb{R}^n\setminus \Omega$. Hence $(h, v)$ solves $S_t(h,v)=0$. For any vector field $V$ vanishing on $\Sigma$, we have $S_t(L_V g_t,0)=0$ and $DA|_{g_t}(L_V g_t)=0$ on $\Sigma$ (see Example~\ref{example:vector}).
We prove the first statement.

For the second statement, let $V$ solve $\Delta_{g_t} V=\bi_{g_t} h + dv$ in $\mathbb{R}^n\setminus\Omega$ and $V=0$ on $\Sigma$. Then we can verify that $\bi_{g_t} (h+L_V g_t)+dv =0$ everywhere. The rest of the statement follows. 
\end{proof}

In the next four lemmas, we establish general facts about the kernel spaces $\Ker L_t$. For an arbitrary $a\in (-\delta, \delta)$, $\Ker L_a$ is not necessarily the limit of $\Ker L_t$ as $t\to a$ because the nullity of $L_a$ may ``jump up'' in the limit. Nevertheless, we shall show that  such jumping  is rare, in a precise sense that the nullity  is upper semi-continuous in $t$.  While some facts are standard for elliptic operators on bounded regions with the standard H\"older norms, we include some proofs in our setting for completeness. Let $\overline{B}_R\subset \mathbb{R}^n$ denote the closed round ball of radius $R$.

\begin{lemma}\label{lemma:scale-broken}
There are positive constants $R, C>0$, uniformly in $t\in (-\delta, \delta)$, such that 
\[
	\| (h,v) \|_{\C^{2,\alpha}_{-q}(\mathbb{R}^n\setminus \Omega)} \le C\Big( \| (h,v) \|_{\C^0 (\overline{B}_R \setminus \Omega)} + \| L_t (h, v) \|_{\mathcal{Z}} \Big)
\]
\end{lemma}
\begin{proof}
By the boundary Schauder estimate \cite{Agmon-Douglis-Nirenberg:1964} in the standard H\"older norms and a scaling argument, we have the boundary Schauder estimate in the weighted H\"older norms:
\[
	\| (h,v) \|_{\C^{2,\alpha}_{-q}(\mathbb{R}^n\setminus \Omega)} \le C\Big( \| (h,v) \|_{\C^0 _{-q}(\mathbb{R}^n \setminus \Omega)} + \| L_t (h, v) \|_{\mathcal{Z}} \Big),
\]
where the constant $C$ can be chosen uniformly in $t$ because the coefficients of $L_t$ are uniformly bounded in suitable weighted norms.  Following closely the arguments involving  the cut-off trick and an interpolation inequality as in \cite[Theorem 1.10]{Bartnik:1986},  the $\C^{0}_{-q}(\mathbb{R}^n \setminus \Omega) $-norm in the right hand side can be replaced by the ``scale-broken'' $\C^{0}(\overline{B}_R\setminus \Omega)$-norm, provided that $R$ is sufficiently large.
\end{proof}

\begin{lemma}\label{lemma:subsequence}
Let $\{t_j\}\subset (-\delta,\delta)$ be a sequence converging to $a\in (-\delta,\delta)$, and let $C$ be a positive real number. Suppose $(h_j, v_j)\in \C^{2,\alpha}_{-q}(\mathbb{R}^n\setminus \Omega)$ and $f_j, f\in \mathcal{Z}$ satisfy
\begin{align*}
	\| (h_j, v_j) \|_{\C^{2,\alpha}_{-q}(\mathbb{R}^n\setminus \Omega)} &< C\\
	L_{t_j} (h_j, v_j)&=f_j\\
	\| f_j - f\|_{\mathcal{Z}} &\to 0\mbox{ as } j\to \infty.
\end{align*}
Then there is a subsequence $(h_{j_k}, v_{j_k})$ and $(h, v)\in \C^{2,\alpha}_{-q}(\mathbb{R}^n\setminus \Omega)$ such that 
\begin{align*}
	\| (h_{j_k}, v_{j_k}) - (h,v)\|_{\C^{2,\alpha}_{-q}(\mathbb{R}^n\setminus \Omega)} &\to 0\mbox{ as } k\to \infty\\
	L_{a} (h, v)&=f.
\end{align*}
Furthermore, if $\| (h_j, v_j) \|_{\C^{2,\alpha}_{-q}(\mathbb{R}^n\setminus \Omega)} =1$ for all $j$, then $\| (h, v) \|_{\C^{2,\alpha}_{-q}(\mathbb{R}^n\setminus \Omega)} =1$. 
\end{lemma}
\begin{proof}
In this proof, all the weighted H\"older norms are taken in the region $\mathbb{R}^n\setminus \Omega$.
Fix $\alpha'<\alpha$ and $q'<q$. By compact embedding of $\C^{2,\alpha}_{-q}$ in $\C^{2,\alpha'}_{-q'}$, there is a subsequence $(h_{j_k}, v_{j_k})$ and $(h, v)\in \C^{2,\alpha}_{-q}$ such that $\| (h_{j_k}, v_{j_k}) - (h,v)\|_{\C^{2,\alpha'}_{-q'}}\to 0$. In particular, $\| (h_{j_k}, v_{j_k}) - (h,v)\|_{\C^0(\overline{B}_R \setminus \Omega)}\to 0$. Together with Lemma~\ref{lemma:scale-broken} and the fact that the operator norm $\|L_{t_{j_k}} -L_a \|_{\mathrm{op}}\to 0$, we obtain $\| (h_{j_k}, v_{j_k}) - (h,v)\|_{\C^{2,\alpha}_{-q}}\to 0$  as $k\to \infty$. The other statements follow 
directly. 
\end{proof}

The following coercivity-type estimate for $(h, v)$ in a complementing space to $\Ker L_t$ is standard if~$t$ is fixed. By slightly extending the standard argument, we show that  a \emph{uniform} coercivity estimate holds.
\begin{lemma}\label{lemma:no-kernel}
Let $\{t_j\}$ be a sequence in $(-\delta,\delta)$ converging to $a\in (-\delta,\delta)$.
Suppose in $\C^{2,\alpha}_{-q}(\mathbb{R}^n\setminus \Omega)$ there is a closed subspace $\mathcal{Y}$ complementing to $\Ker L_a$ and to $\Ker L_{t_j}$ for each $j$, i.e.
\begin{align*}
	\C^{2,\alpha}_{-q}(\mathbb{R}^n\setminus \Omega) =\mathcal{Y}\oplus \Ker L_{a}= \mathcal{Y}\oplus \Ker L_{t_j}.
\end{align*}
 Then there is a positive constant $C$, uniformly in $j$, such that for all $(h, v)\in \mathcal{Y}$, 
\[
	\| (h, v) \|_{\C^{2,\alpha}_{-q}(\mathbb{R}^n\setminus \Omega)} \le C \| L_{t_j} (h, v) \|_{\mathcal{Z}}.
\]
\end{lemma}

\begin{proof}
Suppose, to get a contradiction, that there is a subsequence of $t_j$, still labelled by $j$, and  $(h_j, v_j)\in \mathcal{Y}$ with $\|(h_j, v_j)\|_{\C^{2,\alpha}_{-q}(\mathbb{R}^n\setminus \Omega)}=1$ such that  
\[
 \| L_{t_j}(h_j, v_j)\| \le \tfrac{1}{j} \|(h_j, v_j)\|_{\C^{2,\alpha}_{-q}(\mathbb{R}^n\setminus \Omega)}= \tfrac{1}{j} \to 0 \quad \mbox{ as } j\to \infty.
 \]
By Lemma~\ref{lemma:subsequence}, a subsequence of $(h_j ,v_j)$ converges to $(h, v)\neq 0$ and $(h, v) \in \mathcal{Y}\cap \Ker L_a$. A contradiction. 
\end{proof}

For the next lemma, we define the nullity function $N(t)$ for $t\in (-\delta,\delta)$ by
\[
	N(t) = \Dim \Ker L_t.
\]
We give some definitions in order to describe the kernel spaces. Let $\rho(x) =(1+ |x|^2)^{-1}$. We define the $\mathcal{L}^2_\rho$-inner product for $(h, v), (k, w)\in \C^{2,\alpha}_{-q}(\mathbb{R}^n\setminus \Omega)$ by
\[
	\big\langle (h, v), (k, w)\big \rangle_{\mathcal{L}^2_\rho} = \int_{\mathbb{R}^n\setminus \Omega} (h\cdot k + vw)\rho \, d\mathrm{vol},
\]
where the dot product is with respect to $\bar{g}$. The elements $(h, v)$ and $(k,w)$ are said to be \emph{orthogonal} if $\big\langle (h, v), (k, w)\big \rangle_{\mathcal{L}^2_\rho} =0$. We say that the set $\{ (h_1, v_1),\dots, (h_\ell, v_\ell)\}$  is \emph{orthonormal}  if $\|(h_i, v_i)\|_{\C^{2,\alpha}_{-q}(\mathbb{R}^n\setminus \Omega)} =1$ for all $i$, and $\big\langle (h_i, v_i), (h_k, v_k)\big \rangle_{\mathcal{L}^2_\rho} =0$ for all $i\neq k$. 

Let  $ S_j\subset \C^{2,\alpha}_{-q}(\mathbb{R}^n\setminus \Omega)$ be an $\ell$-dimensional linear space, spanned by an orthonormal basis $\big\{ (h_1^{(j)}, v_1^{(j)}),\dots, (h_\ell^{(j)}, v_\ell^{(j)})\big\}$. We say that a sequence  \emph{$\{S_j\}$ converges to a subspace of the linear space $S\subset \C^{2,\alpha}_{-q}(\mathbb{R}^n\setminus \Omega)$} if there are $ (h_1, v_1),\dots, (h_\ell, v_\ell)\in S$ such that $(h_i^{(j)}, v_i^{(j)})\to (h_i, v_i)$ in $\C^{2,\alpha}_{-q}(\mathbb{R}^n\setminus \Omega)$ as $j\to \infty$ for each $i= 1,\dots, \ell$. As a consequence,  $\{ (h_1, v_1),\dots, (h_{\ell}, v_\ell)\}$ is also orthonormal and hence $\Dim S\ge \ell$. If $\Dim S=\ell$, we simply say that \emph{$\{S_j\}$ converges to $S$}, and note that in this case, any element $(h,v)\in S$ is the limit of a sequence $(h^{(j)}, v^{(j)})\in S_j$ in the $\C^{2,\alpha}_{-q}(\mathbb{R}^n\setminus \Omega)$-norm.

\begin{lemma}\label{lemma:nullity}
The nullity function $N(t)$ satisfies the following properties:
\begin{enumerate}
\item For any $a\in (-\delta,\delta)$, $\limsup_{t\to a} N(t) <+\infty$. \label{item:finite}
\item Suppose $\limsup_{t\to a} N(t)=\ell$ for some  $\ell$. Then there is a sequence $\{t_j\}$ in $(-\delta,\delta)$ with $t_j\to a$ such that $N(t_j) =\ell$ and that $\Ker L_{t_j}$ converges to an $\ell$-dimensional subspace of $\Ker L_a$.  \label{item:convergence}
 \item $N(t)$ is upper semi-continuous. \label{item:upper}
 \item Define the subset $J\subset (-\delta,\delta)$ by
 \begin{align}\label{equation:J}
 \begin{split}
 	J=\big\{ t\in (-\delta,\delta): &  \mbox{ the nullity function is }\\ 
	&\mbox{constant in an open neighborhood of $t$}\big\}.
\end{split}
 \end{align}
 Then $J$ is open and dense. \label{item:J}
 \end{enumerate}
\end{lemma}
\begin{proof}
We prove Item~\eqref{item:finite}: Suppose on the contrary that $\limsup_{t\to a}N(t) = +\infty$ for some $a$. Then there is a sequence $t_j\to a$ such that $N(t_j)$ is increasing in $j$ and $N(t_j)\ge j$ for all positive integers~$j$.  For each~$j$, denote an orthonormal basis for $\Ker L_{t_j}$ by
\[
	\left\{ \big(h_1(t_j), v_1(t_j)\big),  \dots, \big(h_{N(t_j)}(t_j), v_{N(t_j)}(t_j)\big)\right\}.
\]
Fix a positive integer $\ell$. For $j\ge \ell$, consider the $\ell$-dimensional subspace $S_j$ of $\Ker L_{t_j}$, spanned by $\left\{ \big(h_1(t_j), v_1(t_j)\big), \dots, \big(h_\ell(t_j), v_\ell(t_j)\big)\right\}$. Letting $j\to \infty$ and applying Lemma~\ref{lemma:subsequence}, after passing to a subsequence, $S_j$ converges to an $\ell$-dimensional subspace of  $\Ker L_a$. It implies $\Dim \Ker L_a\ge \ell$. Since $\ell$ is arbitrary, we get  $\Ker L_a = +\infty$, which contradicts to that $L_a$ is Fredholm.

We prove Items~\eqref{item:convergence} and \eqref{item:upper}: Since $N(t)$ takes values on integers, if $\limsup_{t\to a} N(t) = \ell$,  there is a sequence $t_j\to a$ such that $N(t_j) =\ell$ for all $j$. A similar argument as above proves Item~\eqref{item:convergence}  (after passing to a subsequence). It also shows that 
\[
	\limsup_{t\to a} N(t) = \ell \le N(a),
\]
which implies Item~\eqref{item:upper}.

We prove Item~\eqref{item:J}: The subset $J$ is open by definition. We show that $J$ is dense. Suppose, on the contrary, that $J$ is not dense. Then there is an interval $[a,b] \subset (-\delta,\delta)\setminus J$ where $a<b$. Since $N(t)$ takes values on nonnegative integers, it attains its minimum at some $s\in [a,b]$, say $N(s) = \min_{[a,b]} N(t) = m$. Define the sublevel set $U = \{ t\in [a,b]: N(t) < m+\frac{1}{2} \}\neq \emptyset$. By upper semi-continuity of $N(t)$, $U$ is a nonempty open subset of $[a,b]$, and thus $\Int U = U\cap (a,b)$ is a nonempty open subset of $(-\delta,\delta)$.  Since $N(t)$ is constant, equal to the minimum $m$, on $\Int U$, we see that $\Int U\subset J$. It contradicts to the assumption that $[a,b]$ lies in the complement of $J$.
\end{proof}

We have shown that for $t$ in the open dense subset $J$, the spaces $\Ker L_t$ behaves well. In particular, for any $a\in J$, $\Ker L_t$ converges to $\Ker L_a$ as $t\to a$. This convergence fulfills the first part of the assumption $(\star)$. The other half of the assumption $(\star)$ requires the $t$-derivative of the kernel elements also converge, which we can replace with the convergence of the difference quotients. Thus, the following proposition can be viewed as a ``discrete version'' of the assumption $(\star)$.

\begin{proposition}~\label{proposition:kernel}
Let $J\subset (-\delta,\delta)$ be the open dense subset defined by \eqref{equation:J}. Then for every $a\in J$, $(h,v) \in \Ker L_a$,  there is a sequence $\{t_j\}$ in $J$ with $t_j \searrow a$,  $(h(t_j), v(t_j))\in \Ker L_{t_j}$, and $(p, z)\in \C^{2,\alpha}_{-q}(\mathbb{R}^n\setminus \Omega)$ such that, as $t_j \searrow a$, 
\begin{align*}
	(h(t_j), v(t_j)) &\to (h,v) \\
	 \frac{(h(t_j), v(t_j) )- (h, v)}{t_j - a} &\to (p, z),
\end{align*}
where both convergences are taken in the $\C^{2,\alpha}_{-q}(\mathbb{R}^n\setminus \Omega)$-norm. 
\end{proposition}

\begin{proof}
 Since $a\in J$, there is an integer $\ell$ such that $ N(t) = \ell$ for $t$ in an open neighborhood of $a$ in $(-\delta,\delta)$. Then there is a sequence $t_j\searrow a$ such that $N(t_j)=\ell=N(a)$, and thus by Lemma~\ref{lemma:nullity}, $\Ker L_{t_j}$ converges to $\Ker L_a$. Then for any $(h, v) \in \Ker L_a$, there is a sequence $(\hat{h}(t_j), \hat{v}(t_j))$ in $\Ker L_{t_j}$ such that $(\hat{h}(t_j), \hat{v}(t_j))\to (h, v)$ in $\C^{2,\alpha}_{-q}(\mathbb{R}^n\setminus \Omega)$. 

Since the difference quotient $\frac{ (\hat{h}(t_j), \hat{v}(t_j))-(h,v)} {t_j -a}$ 
need not converge in general,  we will  modify the sequence $  (\hat{h}(t_j), \hat{v}(t_j))$ by adding a sequence of ``correction'' terms from $\Ker L_{t_j}$.  To proceed, observe that since $\Ker L_{t_j}$ converges to $\Ker L_a$, there is a closed subspace $\mathcal{Y}$ such that, for all $j$ sufficiently large,  $\mathcal{Y}$ is complementing to $\Ker L_{t_j}$ and to $\Ker L_{a}$. Therefore, there is $\big(k(t_j), w(t_j)\big)\in \Ker L_{t_j}$ such that 
\begin{align} \label{equation:E}
	\frac{ (\hat{h}(t_j), \hat{v}(t_j))-(h,v)} {t_j -a} - \big(k(t_j), w(t_j)\big) \in \mathcal{Y}.
\end{align}
We apply $L_{t_j}$ to the above element: 
\begin{align*}
	&L_{t_j} \left(\frac{ (\hat{h}(t_j), \hat{v}(t_j))-(h,v)} {t_j -a} - \big(k(t_j), w(t_j)\big)  \right)\\
	&=-\frac{1}{t_j-a} L_{t_j}(h,v) = -\frac{1}{t_j-a} (L_{t_j} - L_a)(h,v).
\end{align*}
Note that $f_j:= -\frac{1}{t_j-a} (L_{t_j} - L_a)(h,v)$ converges to some $f$ in $\mathcal{Z}$ as $j\to \infty$ because $L_{t}$ is a differentiable family of operators.  Thus, $\| f_j \|_{\mathcal{Z}}$ is bounded above by a constant uniformly in $j$. By Lemma~\ref{lemma:no-kernel} and \eqref{equation:E}, there is a positive constant $C$, uniformly in $j$, such that 
\begin{align*}
	\left\|\frac{ (\hat{h}(t_j), \hat{v}(t_j))-(h,v)} {t_j -a} - \big(k(t_j), w(t_j)\big)   \right\|_{\C^{2,\alpha}_{-q}(\mathbb{R}^n\setminus \Omega)}&\le C \left\|f_j \right\|_{\mathcal{Z}}.
\end{align*}
Since the right hand side is bounded from above, uniformly in $j$, we see that multiplying the inequality above with $(t_j-a)$ gives
\[
	\|(t_j -a) \big(k(t_j), w(t_j)\big)\|_{\C^{2,\alpha}_{-q}(\mathbb{R}^n\setminus \Omega)}\to 0 \quad \mbox{ as } j\to \infty.
\]	
Also, by uniform boundedness of $\| f_j \|_{\mathcal{Z}}$ and  Lemma~\ref{lemma:subsequence}, there is a subsequence of $j$, which we still label by $j$, such that 
\[
	\frac{ (\hat{h}(t_j), \hat{v}(t_j))-(h,v)} {t_j -a} - \big(k(t_j), w(t_j)\big) \to (p, z)
\]
 in $\C^{2,\alpha}_{-q}(\mathbb{R}^n\setminus \Omega)$. Finally, let
\[
	(h(t_j), v(t_j)) =  (\hat{h}(t_j), \hat{v}(t_j))-(t_j-a) \big(k(t_j), w(t_j)\big).
\]
We can verify that $(h(t_j), v(t_j))$ satisfies the desired properties. 

\end{proof}

We show how to apply Proposition~\ref{proposition:kernel}  to remove the assumption $(\star)$ in Theorem~\ref{theorem:special} and thus establish Theorem~\ref{theorem:static-generic}.

\begin{proof}[Proof of Theorem~\ref{theorem:static-generic}]

Let $J$ be the open dense subset of $(-\delta, \delta)$ from Proposition~\ref{proposition:kernel}, and let  $a\in J$.  By re-parametrizing~$t$, we may without loss of generality assume  $a=0$ and hence $g_a = \bar{g}, \Sigma_a = \Sigma, L_a= L$.  We will show that if $(\hat{h}, v)$ solves $S_0(\hat{h}, v)=0$,  then $DA(\hat{h})=0$ on $\Sigma$. 

By the second statement in Lemma~\ref{lemma:static-harmonic}, after changing to the static-harmonic gauge, we may assume that $(\hat{h}, v)\in \Ker L$. By Proposition~\ref{proposition:kernel}, there exist  a sequence $(\hat{h}(t_j), v(t_j))\in \Ker L_{t_j}$ and $(\hat{p}, z)\in \C^{2,\alpha}_{-q}(\mathbb{R}^n\setminus \Omega)$ such that, as $t_j\searrow 0$, 
\begin{align*}
	(\hat{h}(t_j), v(t_j)) &\to (\hat{h}, v)\\
	\left( \frac{\hat{h}(t_j)-\hat{h}}{t_j},  \frac{v(t_j) - v}{t_j}\right)&\to (\hat{p}, z)
\end{align*}
 in the $\C^{2,\alpha}_{-q}(\mathbb{R}^n\setminus \Omega)$-norm. 
 
 To follow closely the proof of Theorem~\ref{theorem:special}, we will modify $(\hat{h}, v)$ to satisfy the geodesic gauge: by Lemma~\ref{lemma:geodesic}, there is  a  vector field $V\in \C^{3,\alpha}_c(\mathbb{R}^n\setminus \Omega)$ with $V=0$ on $\Sigma$ such that $h:=\hat{h}+L_V \bar{g}$ satisfies the geodesic gauge on $\Sigma$. We correspondingly define
\[
	h(t_j) = \hat{h}(t_j) + L_V g_{t_j}.
\]
Note that $h(t_j)$ need not satisfy the geodesic gauge on $\Sigma$ for $t_j\neq 0$. By the first statement of Lemma~\ref{lemma:static-harmonic}, $S_0(h, v)=0$ 
and $S_{t_j}(h(t_j),v(t_j))=0$.
Furthermore, we have, as $t_j\searrow 0$, 
\begin{align*}
	(h(t_j), v(t_j)) &\to (h, v)\\
	\left( \frac{h(t_j)-h}{t_j},  \frac{v(t_j) - v}{t_j}\right)&\to (p, z)
\end{align*}
 in the $\C^{2,\alpha}_{-q}(\mathbb{R}^n\setminus \Omega)$-norm, where $p:= \hat{p} + L_V L_X \bar{g}$.

Follow exactly the same argument as  in the proof of Theorem~\ref{theorem:special}, except that wherever we need to ``differentiate in $t$ at $t=0$", instead we ``take the limit of the difference quotients as $t_j \to 0$''.  We see that $(p-L_X h, z-L_X v)$ is a static vacuum deformation at $\bar{g}$ in $\mathbb{R}^n\setminus \Omega$ whose  Bartnik boundary data on $\Sigma$ satisfy
\begin{align*}
	(p-L_X h\big)^\intercal &= -2\zeta DA(h)\\
	DH(p-L_X h)&=\zeta A\cdot DA(h).
\end{align*}
Then we apply \eqref{equation:Green-boundary} exactly as in the proof of Theorem~\ref{theorem:special}  for $(h, v)$ and $(k, w) = (p-L_X h, z-L_X v)$ to conclude that $DA(h)=0$, and thus $DA(\hat{h})=0$ on $\Sigma$.

\end{proof}

\begin{proof}[Proof of Corollary~\ref{cor:dilation}]
Let $\Omega$ be a bounded open subset in $\mathbb{R}^n$ such that the boundary $\Sigma$ is a star-shaped hypersurface (with respect to the origin).  Let $\psi_t: \mathbb{R}^n\to \mathbb{R}^n$ be the dilation map defined by $\psi_t(x_1,\dots, x_n) = t(x_1,\dots, x_n)$, and define $\Sigma_t = \psi_t(\Sigma)$ and $\Omega_t=\psi_t(\Omega)$. 
Since $\Sigma$ is star-shaped, for $0<\delta<1$ small and for $ t\in [1-\delta, 1+\delta]$, $\Sigma_t$ is a smooth foliation in a tubular neighborhood of $\Sigma$.

By Theorem~\ref{theorem:static-generic}, there exists an open dense subset $J\subset [1-\delta, 1+\delta]$ such that $\Sigma_t$ is static regular in $(\mathbb{R}^n\setminus \Omega_t, \bar g)$ for $t\in J$. We will use scaling from $\Sigma_t$ to show that  $\Sigma$ is static regular in $(\mathbb{R}^n\setminus \Omega, \bar g)$. Namely, for any $(h, v)\in \C^{2,\alpha}_{-q}(\mathbb{R}^n\setminus \Omega)$ satisfies 
\begin{align*}
	&\left\{ \begin{array}{l}  D\Ric(h) - \nabla^2 v =0\\
	\Delta v=0 \end{array}\right. \quad \mbox{ in } \mathbb{R}^n\setminus \Omega \\
	&\left\{ \begin{array}{l} h^\intercal =0\\
	DH(h)=0 \end{array}\right. \quad \mbox{ on } \Sigma,
\end{align*}
we will show $DA(h)=0$ on $\Sigma$, where the linearizations are all  taken at $\bar g$.

For a fixed $t\in J$, we push forward the above system by $\psi_t$. Define $(h_t, v_t) = (\psi_t)_* (h, v)$ and  
\[
	g_t = (\psi_t)_* \bar g = t^{-2} \bar g. 
\]
Then the above system becomes the system for $(h_t, v_t)$ at $g_t$:
\begin{align*}
	&\left\{ \begin{array}{l} - D\Ric|_{g_t}(h_t) + \nabla^2_{g_t} v_t =0\\
	\Delta_{g_t} v_t=0 \end{array}\right. \quad \mbox{ in } \mathbb{R}^n\setminus \Omega_t \\
	&\left\{ \begin{array}{l} h_t^\intercal =0\\
	DH|_{g_t}(h_t)=0 \end{array}\right. \quad \mbox{ on } \Sigma_t,
\end{align*}
Using $g_t = t^{-2} \bar g$, the linearization formulas for $D\Ric$ and $DH$ in Lemma~\ref{lemma:formula}, and the conformal transformation formulas for $\nabla_{g_t}^2, \Delta_{g_t}$, we have that $(h_t, t^{-2} v_t)$ satisfies the following linearized system at $\bar g$:
\begin{align*}
	&\left\{ \begin{array}{l}- D\Ric|_{\bar g}(h_t) + \nabla^2_{\bar g} (t^{-2} v_t )=0\\
	 \Delta_{\bar g} v_t=0 \end{array}\right. \quad \mbox{ in } \mathbb{R}^n\setminus \Omega_t \\
	&\left\{ \begin{array}{l} h_t^\intercal =0\\
	 DH|_{\bar g}(h_t)=0 \end{array}\right. \quad \mbox{ on } \Sigma_t.
\end{align*}
Since $\Sigma_t$ is static regular in $(\mathbb{R}^n\setminus \Omega_t, \bar g)$, we have $DA|_{\bar g} (h_t)$. To see that it implies $DA|_{\bar g}(h)=0$ on~$\Sigma$, we push forward it by $\psi_t$ and get  
\[
	(\psi_t)_* (DA|_{\bar g} (h)) = DA|_{g_t} (h_t) = t DA|_{\bar g} (h_t)=0 \quad \mbox{ on } \Sigma_t.
\]	

\end{proof}

\bibliographystyle{amsplain}
\bibliography{2020}
\end{document}